\documentclass[bj]{imsart}

\RequirePackage{amsthm,amsmath,amsfonts,amssymb}
\RequirePackage[numbers]{natbib}

\RequirePackage[colorlinks,citecolor=blue,urlcolor=blue]{hyperref}
\RequirePackage{graphicx}

\startlocaldefs
\numberwithin{equation}{section}

\theoremstyle{plain}
\newtheorem{thm}{Theorem}[section]

\newtheorem{prop}[thm]{Proposition}

\theoremstyle{definition}

\theoremstyle{remark}
\newtheorem{rem}{Remark}[section]

\usepackage{bm}

   \newcommand{\R}{\mathbb{R}}

\newcommand{\N}{\mathbb{N}}
\newcommand{\F}{\mathfrak{F}}
\newcommand{\1}{\mathbf{1}}
\newcommand{\Prob}{\mathbb{P}}

\newcommand{\ra}{\rightarrow}
\newcommand{\ar}{\rightarrow \infty}
\newcommand{\dx}{\mathrm{d}}

\bmdefine\mub{\mu}
\bmdefine\etab{\eta}
\bmdefine\varthetab{\vartheta}
\bmdefine\varphib{\varphi}
\bmdefine\betab{\beta}
\bmdefine\lab{\lambda}
\bmdefine\sigmab{\sigma}
\bmdefine\varsigmab{\varsigma}
\bmdefine\taub{\tau}
\bmdefine\rhob{\rho}
\bmdefine\pib{\pi}
\bmdefine\varrhob{\varrho}
\bmdefine\gammab{\gamma}
\bmdefine\Gammab{\Gamma}

\endlocaldefs

\begin{document}

\begin{frontmatter}
\title{Degenerate Competing Three-Particle Systems}
\runtitle{Degenerate Competing Three-Particle Systems}

\begin{aug}
\author[A]{\fnms{Tomoyuki} \snm{Ichiba}\ead[label=e1, mark]{ichiba@pstat.ucsb.edu}} 
\and
\author[B]{\fnms{Ioannis} \snm{Karatzas}\ead[label=e2]{ik1@columbia.edu}}
\address[A]{Department of Statistics and Applied Probability, South Hall, University of California, Santa Barbara, CA 93106, USA, \printead{e1}}

\address[B]{Department of Mathematics,  Columbia University, New York, NY 10027, USA, \printead{e2}}
\end{aug}

\begin{abstract}
We study systems of three interacting particles, in which drifts and variances are assigned by rank. These systems are degenerate: the variances corresponding to one   or two ranks can vanish, so the corresponding ranked motions become ballistic rather than diffusive. Depending on which ranks are allowed to ``go ballistic"   the  systems exhibit markedly different behavior,  which we study here  in some detail. Also studied are   stability properties  for the resulting planar process of gaps between successive ranks.
\end{abstract}


\end{frontmatter}


\section{Introduction}
 \label{sec0}

Systems of three or more interacting particles  that assign local characteristics to individual motions by rank, rather than by index (``name"), have received considerable attention in recent years  under the rubric of ``competing particle systems"; see for instance \citep{MR2473654}, \citep{MR2807968}, \citep{MR3055258}, 
\citep{MR3327514}, 
\citep{MR3325099}, \citep{MR3687255}, 
\citep{MR3503022}, 
\citep{MR3449305}, 
\citep{MR3706753}, 
\citep{MR3773380}, 
\citep{MR3957396}, 
and the references cited there. A crucial common feature of    these  studies  is that   particles of all ranks are assigned some strictly positive  local variance. This nondegeneracy smooths out the transition probabilities of   particles. 

We study here, and to the best of our knowledge for the first time, systems of   such competing particles   which are allowed to degenerate, meaning that the variances assigned to  one or two ranks can vanish. This kind of degeneracy calls for an entirely new theory; we initiate such a theory    in the context of systems consisting of three particles. Even with this simplification,   the range of behavior these systems can exhibit is quite  rich. We illustrate just how rich, by studying in detail the construction and properties of three such systems --- respectively, in Sections \ref{sec1}, \ref{sec5},  and in an Appendix (Section \ref{skew-elastic}). 




The systems of Sections \ref{sec1} and \ref{skew-elastic}  assign  ballistic motions to the leader and laggard particles, and   diffusive motion to the    middle particle. The quadratic variations of both   leader and laggard   are zero, as are the cross-variations between any two   particles; the resulting diffusion matrix is thus of rank one. The intensity  of collisions between the middle particle, and the leader or the laggard, is measured by the growth of the local time for the respective gap in ranks. Skew-elastic colliding behavior  between the middle and   laggard particles is described by  collision local times in Section \ref{skew-elastic}. 

By contrast, the system of Section \ref{sec5}  assigns independent  diffusive motions to the leader and laggard particles,   and ballistic ranked motion to the middle particle.  The quadratic variation of the middle particle is zero, and so are the cross-variations between any two   particles; thus, the diffusion matrix  is now of rank two. These different kinds of behavior  are summarized in Table \ref{Table1}.

A salient feature   emerging from the analysis,  is that  two purely ballistic ranked motions can never  ``pinch"  a Brownian motion running between, and reflected off  from,  them  (Proposition \ref{notriple}); whereas two Brownian ranked motions {\it can} pinch a purely ballistic one running in their midst and reflected off from them, resulting in a massive    collision that involves all three particles (subsection \ref{sec: 6.2}). Using simple  excursion-theoretic ideas, we show in Theorem \ref{basic_result} how such a system can extricate itself from   such a  triple collision; but also that the solution to the  stochastic equations that describe its motion, which is demonstrably strong up until the first time a triple collision occurs, {\it ceases to be strong after that time}.  This last question on the structure of filtrations had been open for several years. We also show  that,  even when triple collisions do occur in the  systems studied here, they are ``soft": the associated  triple collision local time is identically equal to zero.

The analysis of these three-particle systems is   connected to that of planar, semimartingale reflecting Brownian motions (SRBMs) and their local times on the faces of the nonnegative orthant. The survey paper \citep{MR1381009} 
and the monograph \citep{MR3157450} 
are  excellent entry points  to this   subject and its applications, alongside the foundational papers 
\citep{MR606992}, 
\citep{MR792398}, 
\citep{MR912049}, \citep{MR877593}, 
\citep{MR921820}. 
In Sections \ref{sec1} and \ref{skew-elastic}  the planar process of gaps is a degenerate SRBM driven by a one-dimensional Brownian motion, while it is a non-degenerate SRBM driven by a two-dimensional Brownian motion    in Section \ref{sec5}. The  directions of reflection are the same in Sections \ref{sec1} and \ref{sec5}, but different in Section \ref{skew-elastic}  because of the skew-elastic collisions between the middle and   laggard particles. 

Under appropriate conditions on the SRBM, the   planar process of gaps between ranked particles has an invariant distribution. We exhibit this joint distribution explicitly in one instance (subsection \ref{sec44}) under the so-called ``skew-symmetry" condition, and offer a conjecture for it in another (Remark \ref{conj}). We show  that the former  is the product  (\ref{InvMeas}) of its exponential marginals, while the latter is determined by the distribution of the  sum of its marginals as in (\ref{sym1}) and is {\it not} of   product  form  (Remark \ref{rem: no product form}).  

The  three-particle systems studied in this paper reveal some of the rich probabilistic structures that  degenerate, three-dimensional continuous \textsc{Markov} processes can exhibit.  It will be very interesting to extend the analysis of the present paper to $n-$particle systems with $n \ge 4$. We expect certain of the    features studied here to hold also in higher dimensions, but leave  such extensions to future work.

\begin{table}
 \begin{center} 
\begin{tabular}{cccccl} \hline 
Section & Leader & Middle & Laggard & Double Collisions & Triple Collisions \\ \hline 
\ref{sec1} & Ballistic & Diffusive & Ballistic & Elastic & No triple collisions  \\
 \ref{sec5} & Diffusive & Ballistic & Diffusive & Elastic & Triple collisions \\
 \ref{skew-elastic} & Ballistic & Diffusive & Ballistic & Skewed &   (open question)  
  \\ \hline 
\end{tabular}
\caption{Different behavior  for   leader, middle and laggard particles in each of three particle systems.} \label{Table1}
\end{center} 
\end{table}

\section{
 Diffusion in the Middle, with Ballistic Hedges
}
 \label{sec1}

Given real numbers $\, \delta_1, \,\delta_2,\,\delta_3\,$  and $ \, x_1 > x_2 > x_3\,$, we   start by constructing   a   probability space $  (\Omega, \F, \Prob) $ endowed with a right-continuous filtration $\mathbb{F}= \big\{ \F (t) \big\}_{0 \le t < \infty} ,$ to which are adapted     three independent  Brownian motions $\, B_1 (\cdot), \,B_2 (\cdot), \,B_3 (\cdot)   ,$ and   three continuous  processes $\, X_1 (\cdot), \,X_2 (\cdot), \,X_3 (\cdot)  \, $   that   satisfy 
\begin{equation}
\label{1}
X_i (\cdot) \,=\, x_i + \sum_{k=1}^3\, \delta_k \int_0^{\, \cdot} \1_{ \{ X_i (t) = R^X_{k} (t)\} } \, \dx t +  \int_0^{\, \cdot} \1_{ \{ X_i (t) = R^X_{2} (t)\} } \, \dx B_i (t)\,,\quad i=1, 2, 3\,, 
\end{equation}
\begin{equation}
\label{2}
\int_0^{\, \infty} \1_{ \{ R^X_{k} (t) = R^X_{\ell} (t)\} } \, \dx  t \,=\, 0\,, ~~~~~\forall ~~ k < \ell, 
\end{equation}
\begin{equation}
\label{3}
\left\{ t \in (0, \infty)\,:\, R^X_1 (t) = R^X_3 (t)   \right\} \,=\, \emptyset 
\end{equation}
with probability one. Here  we denote  the descending order statistics  by 
\begin{equation}
\label{ranks}
\max_{j=1,2,3} X_j  (t) =: R^X_{1} (t) \ge   R^X_{2} (t) \ge   R^X_{3} (t) :=  \min_{j=1,2,3} X_j (t)\,, \qquad t \in [0, \infty),  
\end{equation}
and   adopt  the convention of resolving ties always in favor of the lowest index $  i\,$; for instance, we set 
$$
R^X_{1} (t) = X_1 (t)\,, ~~ R^X_{2} (t) = X_3 (t)\,, ~~ R^X_{3} (t) = X_2 (t)\qquad \hbox{on}~~~\big\{ X_1 (t) = X_3 (t) > X_2 (t) \big\}\,.
$$

The dynamics of (\ref{1}) mandate ballistic motions for the leader and   laggard particles with drifts $\, \delta_1\,$ and $  \delta_3\,$, respectively,  which act here as ``outer hedges";  and   diffusive (Brownian) motion   with   drift $\, \delta_2\,,$ for the   particle in the middle. The condition (\ref{2}) posits that    collisions of particles are {\it non-sticky,} in   that the  set of all collision times has zero \textsc{Lebesgue} measure; while   (\ref{3}) proscribes triple collisions. 

 As a  canonical example, it is useful to keep in mind  the symmetric configuration 
\begin{equation}
\label{4}
 \, \delta_3 = - \delta_1 = \gamma >0 = \delta_2\,, 
\end{equation} 
for which the system of equations (\ref{1}) takes the appealing, symmetric form
\begin{equation}
\label{1a}
X_i (\cdot) = x_i + \gamma    \int_0^{\, \cdot} \Big( \1_{ \{ X_i (t) = R^X_{3} (t)\} } - \1_{ \{ X_i (t) = R^X_{1} (t)\} } \Big)  \dx t  + \int_0^{\, \cdot} \1_{ \{ X_i (t) = R^X_{2} (t)\} } \, \dx B_i (t)\,.
\end{equation}
This system was introduced and studied for the first time in the technical report \citep{FernholzER11}. 
Figure \ref{Fig: 1}, which    illustrates its path behavior, is taken from that report. %

A very  salient feature of the dynamics in (\ref{1}) is that its dispersion structure is both degenerate and discontinuous. It should come then as no surprise, that the analysis of the system (\ref{1})--(\ref{3}) might be not   entirely trivial; in particular, it is not covered by the results in either \citep{MR532498} 
or \citep{MR917679}.  
The question then, is whether a process $X(\cdot)$ with 
  (\ref{1}) and 
  (\ref{2}), (\ref{3}) exists; and if so, whether its distribution  --- and, a bit more ambitiously, its sample path structure ---  is determined uniquely.

\subsection{Analysis}
 \label{sec2}

Suppose that   a solution to the system of equation (\ref{1}) subject to the conditions of (\ref{2}), (\ref{3}),      has been constructed. Its descending  order-statistics  
are given then as 
\begin{equation}
\label{RX1}
R^X_{1} (t)\,=\, x_1 + \delta_1\, t + {1 \over \,2\,} \, \Lambda^{(1,2)}(t)\,, \qquad R^X_{3} (t)\,=\, x_3 + \delta_3\, t  -   {1 \over \,2\,} \, \Lambda^{(2,3)}(t)
\end{equation}
\begin{equation}
\label{RX2}
R^X_{2} (t)\,=\, x_2 + \delta_2\, t + W(t) - {1 \over \,2\,} \, \Lambda^{(1,2)}(t) + {1 \over \,2\,} \, \Lambda^{(2,3)}(t)
\end{equation} 
for $\, 0 \le t < \infty\,$, on the strength of the results in \citep{MR2428716}. 
Here, the scalar process 
\begin{equation}
\label{W2}
W (\cdot):=\sum_{i=1}^3 \int_0^{\, \cdot} \1_{ \{ X_i (t) = R^X_{2} (t)\} } \, \dx B_i (t)=\sum_{i=1}^3 \left( X_i (\cdot) - x_i - \sum_{k=1}^3\, \delta_k \int_0^{\, \cdot} \1_{ \{ X_i (t) = R^X_{k} (t)\} }   \dx t \right)
\end{equation}
is standard Brownian motion by the \textsc{P. L\'evy} theorem; and we denote  the local time accumulated  at the origin by the continuous, nonnegative semimartingale $\, R^X_{k} (\cdot)- R^X_{\ell} (\cdot)\,$ over the time interval $\,[0,t] $, by 
\begin{equation}
\label{Lambda}
\Lambda^{(k,\ell)}(t) \, \equiv \, L^{R^X_{k} - R^X_{\ell}  } (t)\,, \qquad k < \ell.
\end{equation}
  Throughout this paper we use the convention 
  \newpage
\begin{equation}
\label{LT}
 L^\Xi (\cdot) \equiv L^\Xi (\cdot\, ; 0) :=  \lim_{\varepsilon \downarrow 0} \frac{1}{2 \varepsilon} \int_0^\cdot \1_{ \{ \Xi (t) < \varepsilon \} }  \dx \langle M \rangle (t) =    \int_0^{\, \cdot}  \1_{ \{ \Xi (t) = 0\} }   \dx \Xi (t)  =   \int_0^{\, \cdot}  \1_{ \{ \Xi (t) = 0\} }   \dx C (t) 
\end{equation}
for the ``right" local time at the origin of a continuous, nonnegative semimartingale of the form $\, \Xi (\cdot) = \Xi (0) + M(\cdot) + C (\cdot)\,$, with $\, M(\cdot)\,$ a continuous local martingale and $\, C(\cdot)\,$ a  process of finite first variation on compact intervals. The local time process $\,  L^\Xi (\cdot) \,$ is   continuous, adapted and nondecreasing,   flat off the zero-set $\, \{ t \ge 0 : \Xi (t) =0\}\,$. 

We denote now the sizes of the gaps between the leader and the middle particle, and between the middle particle and the laggard,  by 
\begin{equation}
\label{GH}
G (\cdot)\,:=\, R^X_{1} (\cdot)- R^X_{2} (\cdot)\,, \qquad H (\cdot)\,:=\, R^X_{2} (\cdot)- R^X_{3} (\cdot),
\end{equation}
respectively,   and obtain from (\ref{RX1})--(\ref{Lambda})  the semimartingale representations
\begin{equation}
\label{G}
G (t) \,=\, x_1 - x_2 - \big(\delta_2 - \delta_1 \big) \,t  - W(t) - {1 \over \,2\,} \, L^{H}(t) +   L^{G}(t)\,, \qquad 0 \le t < \infty
\end{equation}
\begin{equation}
\label{H}
H (t) \,=\, x_2 - x_3 - \big(\delta_3 - \delta_2 \big)\, t  + W(t) - {1 \over \,2\,} \, L^{G}(t) +   L^{H}(t)\,, \qquad 0 \le t < \infty 
\end{equation}
where we recall the identifications $\, L^G (\cdot) \equiv \Lambda^{(1,2)}(\cdot) \,,$   $\, L^H (\cdot) \equiv \Lambda^{(2,3)}(\cdot)\, $ from (\ref{GH}), (\ref{Lambda}). We  introduce also the continuous semimartingales
$$
U(t)  = x_1 - x_2 - \big(\delta_2 - \delta_1 \big) \,t  - W(t) - {1 \over \,2\,} \, L^{H}(t)\,, \quad V(t)  = x_2 - x_3 - \big(\delta_3 - \delta_2 \big)\, t  + W(t) - {1 \over \,2\,} \, L^{G}(t)\,,
$$
and  observe
\begin{equation}
\label{SkorU}
G (\cdot) \,=\, U (\cdot) + L^G (\cdot) \, \ge \, 0\,, \qquad \int_0^\infty \1_{ \{ G(t)>0\} } \, \dx L^G (t) \,=\,0
\end{equation}
\begin{equation}
\label{SkorV}
H (\cdot) \,=\, V (\cdot)  + L^H (\cdot) \, \ge \, 0\,, \qquad \int_0^\infty \1_{ \{ H(t)>0\} } \, \dx L^H (t) \,=\,0\,.
\end{equation}

In other words, the ``gaps" $\, G(\cdot)\,$, $\, H(\cdot)\,$ are the \textsc{Skorokhod} reflections of the semimartingales $\, U(\cdot) $ and  $  V(\cdot)$, respectively. The theory of the \textsc{Skorokhod} reflection problem (e.g., Lemma 3.6.14  in \citep{MR1121940}) 
provides now the relationships
\begin{equation}
\label{5}
L^G (t) \,=\, \max_{0 \le s \le t} \big( - U (s) \big)^+\,=\, \max_{0 \le s \le t} \Big(-(x_1 - x_2) + \big(\delta_2 - \delta_1 \big) \,s  + W(s) + {1 \over \,2\,} \, L^{H}(s) \Big)^+ 
\end{equation}
\begin{equation}
\label{6}
L^H (t) \,=\, \max_{0 \le s \le t} \big( - V (s) \big)^+\,=\, \max_{0 \le s \le t} \Big(-(x_2 - x_3) + \big(\delta_3 - \delta_2 \big) \,s  - W(s) + {1 \over \,2\,} \, L^{G}(s) \Big)^+ 
\end{equation}
between the two local time processes $\, L^G (\cdot) \equiv \Lambda^{(1,2)}(\cdot) \,$ and $\, L^H (\cdot) \equiv \Lambda^{(2,3)}(\cdot)\,,$ once the scalar  Brownian motion $W(\cdot)$ has been specified.

\noindent
$\bullet~$ 
Finally, we note that   the equations of (\ref{G})-(\ref{H}) can be cast in the form
\begin{equation} 
\label{eq: mfrakQ}
\begin{pmatrix}
      G(t)   \\
      H(t)  
\end{pmatrix}
\,=:\, \mathfrak{G} (t) \,=\, \mathfrak{g} + \mathfrak{Z} (t) + \mathcal{R}\, \mathfrak{L} (t)\,, \qquad 0 \le t < \infty\,,
\end{equation}
of \citep{MR606992}, 
where 
\newpage
$$
\mathcal{R}\, =\, \mathcal{I} - \mathcal{Q}\,, \qquad \mathcal{Q}\, := \begin{pmatrix}
      0 &     1/2 \\
        1/2 &   0
\end{pmatrix}\,, \qquad \mathfrak{g} = \mathfrak{G}(0)\,, \qquad \mathfrak{L} (t) \,= \begin{pmatrix}
      L^G(t)    \\
      L^H (t)  
\end{pmatrix},
$$
\begin{equation} 
\label{Z}
 \mathfrak{Z} (t) \, = \begin{pmatrix}
       ( \delta_1 - \delta_2) t - W(t)      \\
       ( \delta_2 - \delta_3) t + W(t)  
\end{pmatrix}\,, \qquad 0 \le t < \infty\,.
\end{equation}
One reflects  off the faces of the nonnegative quadrant, in other words,  the degenerate, two-dimensional Brownian motion $\, \mathfrak{Z} (\cdot) \,$ with drift vector and covariance matrix given respectively by 
\begin{equation} 
\label{mA}
 {\bm m}  \,= \,\big(  \delta_{1} - \delta_{2} \,,  \,\delta_{2} - \delta_{3}\big)^{\prime}\, ,~~ ~~~~\mathcal{C}\, := \begin{pmatrix}
      1 &     -1 \\
        -1 &   1
\end{pmatrix}.
\end{equation}
 The directions of reflection are   the row vectors of the {\it reflection matrix} $\,\mathcal{R}\,$, and  
the matrix $\,  \mathcal{Q}=  \mathcal{I} - \mathcal{R}\,$ has spectral radius strictly less than $1$,   as postulated by \citep{MR606992}. 
The process $\,\mathfrak Z(\cdot)\,$ is allowed in \citep{MR606992} to be degenerate, i.e., its covariation  matrix   can be only nonnegative-definite.

\subsection{Synthesis} 
\label{sec3}

We trace now the steps of subsection \ref{sec2} in reverse: start with  given real numbers $\, \delta_1, \,\delta_2,\,\delta_3\,$, and $ \, x_1 > x_2 > x_3\,$, and construct   a filtered probability space $\, (\Omega, \F, \Prob),$ $\mathbb{F}= \big\{ \F (t) \big\}_{0 \le t < \infty}\,$ rich enough to support  a scalar, standard Brownian motion $\, W(\cdot) $.  In fact, we select the filtration $\,\mathbb{F} \,$ to be $\,\mathbb{F}^W= \big\{ \F^W (t) \big\}_{0 \le t < \infty},$ the smallest right-continuous filtration to which 
$W(\cdot)$ is adapted.

Informed by the analysis of the previous section we  consider, by analogy with (\ref{5})-(\ref{6}),  the two-dimensional \textsc{Skorokhod} reflection system 
 \begin{equation}
\label{7}
A (t) \,=\,   \max_{0 \le s \le t} \Big(-(x_1 - x_2) + \big(\delta_2 - \delta_1 \big) \,s  + W(s) + {1 \over \,2\,} \, \Gamma (s) \Big)^+\,,\qquad 0 \le t <\infty
\end{equation}
 \begin{equation}
\label{8}
\Gamma (t) \,=\,  \max_{0 \le s \le t} \Big(-(x_2 - x_3) + \big(\delta_3 - \delta_2 \big) \,s  - W(s) + {1 \over \,2\,} \,A(s) \Big)^+\,,\qquad 0 \le t <\infty
\end{equation}
 for two continuous, nondecreasing and $\,\mathbb{F}^W-$adapted processes $\, A(\cdot)\,$ and $\, \Gamma (\cdot)\,$ with $\, A(0) = \Gamma (0) =0\,$.  This system of equations  is of the type  studied in \citep{MR606992}. 
 From Theorem 1 of that paper, we know that it  
 possesses a unique, $\,\mathbb{F}^W-$adapted solution.

Once the solution $\, \big( A(\cdot), \Gamma (\cdot) \big)\,$ to this system   has been constructed, we define the processes 
\begin{equation}
\label{UV}
U(t) := x_1 - x_2 - \big(\delta_2 - \delta_1 \big) \,t  - W(t) - {1 \over \,2\,} \, \Gamma(t)\,, \quad V(t) := x_2 - x_3 - \big(\delta_3 - \delta_2 \big)\, t  + W(t) - {1 \over \,2\,} \, A(t) ~~
\end{equation}
 and then ``fold" them to obtain their \textsc{Skorokhod} reflections; that is, the    continuous semimartingales  
 \begin{equation}
\label{SkorU1}
G ( t) \,:=\, U ( t) +\max_{0 \le s \le t} \big( - U (s) \big)^+ \, = \, x_1 - x_2 - \big(\delta_2 - \delta_1 \big) \,t  - W(t) - {1 \over \,2\,} \, \Gamma(t) + A(t) \, \ge \, 0
\end{equation}
\begin{equation}
\label{SkorV1}
H ( t) \,:=\,   V ( t) +\max_{0 \le s \le t} \big( - V (s) \big)^+ \, = \, x_2 - x_3 - \big(\delta_3 - \delta_2 \big) \,t  + W(t) - {1 \over \,2\,} \, A(t) + \Gamma(t)\, \ge \, 0
\end{equation}
for $\, t \in [0, \infty)\,$, in accordance with (\ref{SkorU})--(\ref{6}). From the theory of the \textsc{Skorokhod} reflection problem once again,  we deduce   the  a.e. properties
\newpage
\begin{equation}
\label{10}
\int_0^\infty \1_{ \{ G(t)>0\} } \, \dx A(t) \,=\,0\,, \qquad \int_0^\infty \1_{ \{ H(t)>0\} } \, \dx \Gamma (t) \,=\,0\,;
\end{equation}
and  the theory of semimartingale local time (\citep{MR1121940}, 
Exercise 3.7.10), gives 
\begin{equation}
\label{9}
\int_0^\infty \1_{ \{ G(t)=0\} } \, \dx t= \int_0^\infty \1_{ \{ G(t)=0\} } \, \dx \langle G \rangle ( t)=0\,, \quad \int_0^\infty \1_{ \{ H(t)=0\} } \, \dx t = \int_0^\infty \1_{ \{ H(t)=0\} } \, \dx \langle H \rangle ( t)=0\,.
\end{equation}

\subsubsection{Constructing the Ranks} 
\label{sec: ConstRanks}

We introduce now, by analogy with \eqref{RX1}-\eqref{RX2},  the processes
\begin{equation}
\label{R1}
R_{1} (t)\,:=\, x_1 + \delta_1\, t + {1 \over \,2\,} \, A(t)\,, \qquad R_{3} (t)\,:=\, x_3 + \delta_3\, t  -   {1 \over \,2\,} \, \Gamma (t) \,,
\end{equation}
\begin{equation}
\label{R2}
R_{2} (t)\,:=\, x_2 + \delta_2\, t + W(t) - {1 \over \,2\,} \, A(t) + {1 \over \,2\,} \, \Gamma (t)
\end{equation}
\noindent
for $\, 0 \le t < \infty\,$ and note the relations $\, R_1 (\cdot) - R_2 (\cdot) = G (\cdot)\ge 0\,$, $\, R_2 (\cdot) - R_3 (\cdot) = H (\cdot) \ge 0\,$ in conjunction with 
 (\ref{SkorU1}) and  (\ref{SkorV1}). In other words, we have the a.e.\,comparisons, or ``descending rankings", 
 $ R_1 (\cdot)   \ge    R_2 (\cdot)  \ge    R_3 (\cdot)     $. 
It is clear from the discussion following (\ref{7}), (\ref{8}),   that these processes are   adapted to the filtration generated by the driving Brownian motion $W(\cdot)$, whence the   inclusion $\,\mathbb{F}^{\,(R_1, R_2, R_3)} \subseteq \mathbb{F}^{\,W}.   $

 \smallskip
Let us show that  these rankings never collapse. To put things a bit   colloquially:   {\it ``Two ballistic motions cannot squeeze a diffusive (Brownian) motion".}  We are indebted to Drs. Robert \textsc{Fernholz} (cf.$\,$ \citep{FernholzER10} 
) and Johannes \textsc{Ruf} for the argument that follows.

\begin{prop}
 \label{notriple}
With probability one,  we have: $\, R_1 (\cdot)- R_3 (\cdot) = G(\cdot) + H(\cdot) >0\,$.
 \end{prop}
 
\noindent
{\it Proof:} We shall show that there cannot possibly exist   numbers $\, T \in (0, \infty)\,$ and $\,r \in \R\,$, such that $\, R_1 (T) = R_2 (T) = R_3 (T)= r  \,$.

We  argue by contradiction: If such a configuration were possible for some $\, \omega \in \Omega\,$ and some $\, T = T (\omega) \in (0, \infty)\,$, $\, r = r (\omega) \in \R\,$, we would have
$$
r - \delta_3 (T-t) \, \le \, R_3 (t,\omega) \, \le \, R_2 (t,\omega) \, \le \, R_1 (t,\omega) \, \le \, r - \delta_1 (T-t)\,, \qquad 0 \le t < T\,. 
$$

\smallskip
\noindent
This is already impossible if $\, \delta_1 > \delta_3\,$, so let us assume $\, \delta_1 \le \delta_3\,$ and try to arrive at a contradiction  in this case as well. The above quadruple inequality implies, a fortiori, 
$$
r - \delta_3 (T-t) \, \le \, \overline{R} (t,\omega) := \frac{1}{\,3\,}\big( R_3 (t,\omega) +R_2 (t,\omega) +R_1 (t,\omega)  \big) \, \le \, r - \delta_1 (T-t)\,, \qquad 0 \le t < T\,.
$$
But we have $\, r - \overline{R} (t, \omega) = \overline{R} (T, \omega) - \overline{R} (t, \omega) = \overline{\delta} (T-t) + \big( W(T, \omega)-W(t, \omega) \big) /3\,$, where $\,  \overline{\delta} := ( \delta_1 + \delta_2 + \delta_3) / 3\,$, and back into the above inequality this gives 
$$
3 \, \big (\delta_1 - \overline{\delta}\, \big) \, \le \, \frac{\, W(T, \omega) - W(t, \omega)\,}{T-t}\, \le \,3 \, \big (\delta_3 - \overline{\delta}\, \big) \,, \qquad 0 \le t < T\,.
$$
However,  from the \textsc{Payley-Wiener-Zygmund} theorem for the Brownian motion $\, W(\cdot)\,$ (\citep{MR1121940},    
Theorem 2.9.18, p.\,110), this double inequality would force $\, \omega\,$ into a $\Prob-$null set.   \qed 

\newpage 

Since in this case the sequence $\,(\sigma_{1}^{2}, \sigma_{2}^{2}, \sigma_{3}^{2}) = (0, 1, 0) \,$ of local covariances-by-rank is concave, the lack of triple collisions just established is in formal accordance with known results;  although not obviously expected, let alone deduced, from them, because the local variance structure here is degenerate.

\subsubsection{Identifying the Increasing Processes $\, A(\cdot)\,$,$\,\, \Gamma (\cdot) \,$  as Local Times}
 \label{3.2}

 We claim   that, in addition to (\ref{10}) and (\ref{9}), the properties   
 \begin{equation}
\label{11}
\int_0^\infty \1_{ \{ H(t)=0\} } \, \dx A (t) \,=\,0\,, \qquad
\int_0^\infty \1_{ \{ G(t)=0\} } \, \dx \Gamma(t) \,=\,0\ 
\end{equation}
are also valid a.e. Indeed, we know from (\ref{10}) that $\, A(\cdot)\,$ is flat off the set $\, \{ t \ge 0: G(t) = 0\}\,$, so we have 
$\,
\int_0^\infty \1_{ \{ H(t)=0\} } \, \dx A (t) = \int_0^\infty \1_{ \{ H(t)=G(t)=0\} } \, \dx A (t) \,$;  but this last expression is a.e. equal to zero    because, as we have shown, $  \{ t \ge 0: G(t) = H(t) =0\} = \emptyset \,$ holds mod.$\,\Prob\,$. This proves the first equality in (\ref{11}); the second is argued similarly.

But now, the local time at the origin of the continuous, nonnegative semimartingale $\, G(\cdot)\,$ is given as
$$
L^G (\cdot) = \int_0^{\, \cdot} \1_{ \{ G(t)=0\} }\, \dx G(t) = \int_0^{\, \cdot} \1_{ \{ G(t)=0\} }\, \dx A(t) 
- \int_0^{\, \cdot} \1_{ \{ G(t)=0\} }\, \frac{\,\dx \Gamma (t)\,}{2} + \big( \delta_1 - \delta_2 \big) \int_0^{\, \cdot} \1_{ \{ G(t)=0\} }\, \dx t
$$

\medskip
\noindent
 from (\ref{LT}) and (\ref{SkorU1}). From (\ref{9}) and (\ref{11})  the last two integrals vanish, so  (\ref{10}) leads to 
 \begin{equation}
\label{LGA}
L^G (\cdot) \,=\,\int_0^{\, \cdot} \1_{ \{ G(t)=0\} }\, \dx A(t) \,=\, A(\cdot) \,; \qquad \text{and } \qquad L^H (\cdot)   \,=\, \Gamma(\cdot)
\end{equation}
  is shown similarly.     In the light of Proposition \ref{notriple}, these local times satisfy the rather interesting property
 \begin{equation}
\label{LGA+LHG}
\frac{1}{\,2\,} \Big( L^G (t) + L^H (t) \Big) \, > \, x_3 - x_1 + \big( \delta_3 - \delta_1 \big)\, t\,, \qquad 0 < t < \infty\,.
 \end{equation}

\begin{rem}
\label{StrFilt}
{\it The Structure of Filtrations:} 
  We have identified the components of  
  the pair   
   $  (A(\cdot)  ,$ $\Gamma (\cdot))  ,$  solution  of the system   (\ref{7})-(\ref{8}),  as the local times at the origin of the continuous semimartingales $  R_1 (\cdot) -  R_2 (\cdot) = G(\cdot) \ge 0 \,$ and $ \,R_2 (\cdot) -  R_3 (\cdot) = H(\cdot) \ge 0\,.$ In particular, this  implies the filtration inclusions  $\,  \mathbb{F}^{\,(A, \Gamma)} = \mathbb{F}^{\,(R_1, R_3)}   \subseteq  \mathbb{F}^{\,(G, H )}  \subseteq  \mathbb{F}^{\,(R_1, R_2, R_3)}  ;$ and back in (\ref{R2}), it gives $\,  \mathbb{F}^{\,W} \subseteq  \mathbb{F}^{\,(R_1, R_2, R_3)}    .$     But we have already noted the reverse of this inclusion, so we conclude that the process of ranks generates exactly the same filtration as the driving scalar Brownian motion: $\,
\mathbb{F}^{\,(R_1, R_2, R_3)}  = \mathbb{F}^{\,W}  .$

\end{rem}

\subsubsection{Constructing the Individual Motions  (``Names")}
 \label{3.3}

Once the ``ranks" $\, R_{1} (\cdot) \ge   R_{2} (\cdot) \ge   R_{3} (\cdot) \,$ have been constructed in  \S \ref{sec: ConstRanks}  on   the   filtered probability  space $\, (\Omega, \F, \Prob)\,$, $\mathbb{F}=  \{ \F (t)  \}_{0 \le t < \infty}\,$, with   $\,\mathbb{F} \,$   selected as   the smallest right-continuous filtration $\,\mathbb{F}^W=  \{ \F^W (t)  \}_{0 \le t < \infty}\, $ to which the scalar Brownian motion $W(\cdot)$ is adapted,  we can construct as in  the proof of Theorem 5 in \citep{MR3449305} 
the ``names''   that generate these ranks --- that is, processes $ X_1 (\cdot), X_2 (\cdot), X_3 (\cdot)\,$, as well as a three-dimensional Brownian motion $\, ( B_1 (\cdot), B_2 (\cdot), B_3 (\cdot) )\,$ defined  on this same space, such  that  the equation (\ref{1}) is satisfied  and $\,  R^X_k (\cdot) \equiv R_k (\cdot)\,, ~ k=1, 2, 3\,$. 
 
It is also clear from our construction that the conditions (\ref{2}) and (\ref{3}) are also satisfied: the first thanks to the properties of (\ref{9}), the second because of Proposition \ref{notriple}.

\smallskip
\noindent
$\bullet~$ 
Alternatively, the construction of a pathwise unique, strong solution for the system (\ref{1}) can be carried out along the lines of Proposition 8 in \citep{MR3055258}. 
We start  at time $\, \tau_0 \equiv 0\,$ and follow  the paths of the top particle and of the pair consisting of the bottom two particles {\it separately}, until the top particle collides with the leader of the bottom pair (at time $\, \varrho_0)$. Then we follow the paths of the bottom particle and of the pair consisting of the  top two particles {\it separately}, until the bottom particle collides with the laggard of the top pair (at time $\, \tau_1)$. We repeat the procedure until the first time we see a triple collision,   obtain two interlaced sequences of stopping times $\, \{ \tau_k \}_{k \in \N_0}\,$ and $\, \{ \varrho_k \}_{k \in \N_0}\,$ with 
\begin{equation}
\label{seq}
0 = \tau_0 \le \varrho_0 \le \tau_1 \le \varrho_1 \le \cdots \le \tau_k \le \varrho_k \le \cdots \,\,,  
\end{equation}
and denote by the first time   a triple collision occurs
\begin{equation}
\label{exp}
\mathcal{S} \,:=\, \inf \big\{ t \in (0, \infty): X_1 (t) = X_2 (t) = X_3 (t)   \big\}= \lim_{k \ar} \tau_k = \lim_{k \ar} \varrho_k  \, . 
\end{equation}
  During each interval of the form $\, [\tau_k, \varrho_k)\,$ or $\, [ \varrho_k, \tau_{k+1})\,$, a pathwise unique, strong solution of the corresponding two-particle system is constructed as in Theorem 4.1 in \citep{MR3055262}. 
  
  We end up in this manner with a three-dimensional Brownian motion $\, ( B_1 (\cdot), B_2 (\cdot), B_3 (\cdot) ),$ and with three processes $\, X_1 (\cdot), X_2 (\cdot), X_3 (\cdot)\,$ that satisfy the system of (\ref{1}) as well as the requirement (\ref{2}), once again thanks to results in \citep{MR3055262}. 
  For this system, the ranked processes $\, R^X_{1} (\cdot) \ge   R^X_{2} (\cdot) \ge   R^X_{3} (\cdot) \,$ as in (\ref{ranks}),  satisfy the equations we studied in  \S \,\ref{sec: ConstRanks}, and generate the same filtration   $\,\mathbb{F}^W=  \{ \F^W (t)  \}_{0 \le t < \infty}\, $ as  the scalar Brownian motion $W(\cdot)$ above (Remark \ref{StrFilt}).    
  We have seen in Proposition  \ref{notriple} that for such a system there are no triple collisions:  $\,\mathcal{S} = \infty\,$. Thus the  condition (\ref{3}) is satisfied as well, 
  all   inequalities in (\ref{seq}) are strict, and we have proved the following result.

\begin{thm}
The system of equations \eqref{1} admits a pathwise unique  strong solution,   satisfying the requirements   \eqref{2}, \eqref{3}.   
\end{thm}

Figure 1, reproduced here from \citep{FernholzER10}, 
illustrates   trajectories of 
$\, X_1(\cdot), \, X_2(\cdot), \,X_3(\cdot)  \,$  for the  canonical example (\ref{4}). It is   clear from this picture and from the construction in  \S   \ref{sec: ConstRanks}, that the middle particle $\, R_2(\cdot)\,$ undergoes Brownian motion $\, W(\cdot)\,$ with reflection at the upper and lower boundaries, respectively $\, R_1 (\cdot)\,$ and $\, R_3 (\cdot)$, of a time-dependent domain.

In contrast to  the ``double \textsc{Skorokhod} map" studied by \citep{MR2349573}, 
where the upper and lower reflecting boundaries are given constants, these boundaries $\, R_1 (\cdot)\,$ and $\, R_3 (\cdot)\,$   are here random continuous functions of time,  of finite first variation on compact intervals. They are   ``sculpted" by the Brownian motion $\, W(\cdot)\,$ itself    via  the local times $\, L^G (\cdot) \equiv A(\cdot)\,$ and $\, L^H (\cdot) \equiv \Gamma (\cdot)\,,$  in the manner of the system   (\ref{7}), (\ref{8}). The upper (respectively, lower) boundary decreases (resp., increases) by linear segments at a $45^o-$angle, and increases (resp., decreases) by a singularly continuous  \textsc{Cantor}-like random function, governed by the local time $\, L^G (\cdot) \equiv A (\cdot) \,$ $($resp., $\, L^H (\cdot) \equiv \Gamma (\cdot) )$.


 \begin{figure}
  \begin{center}
  \includegraphics[scale=0.43]{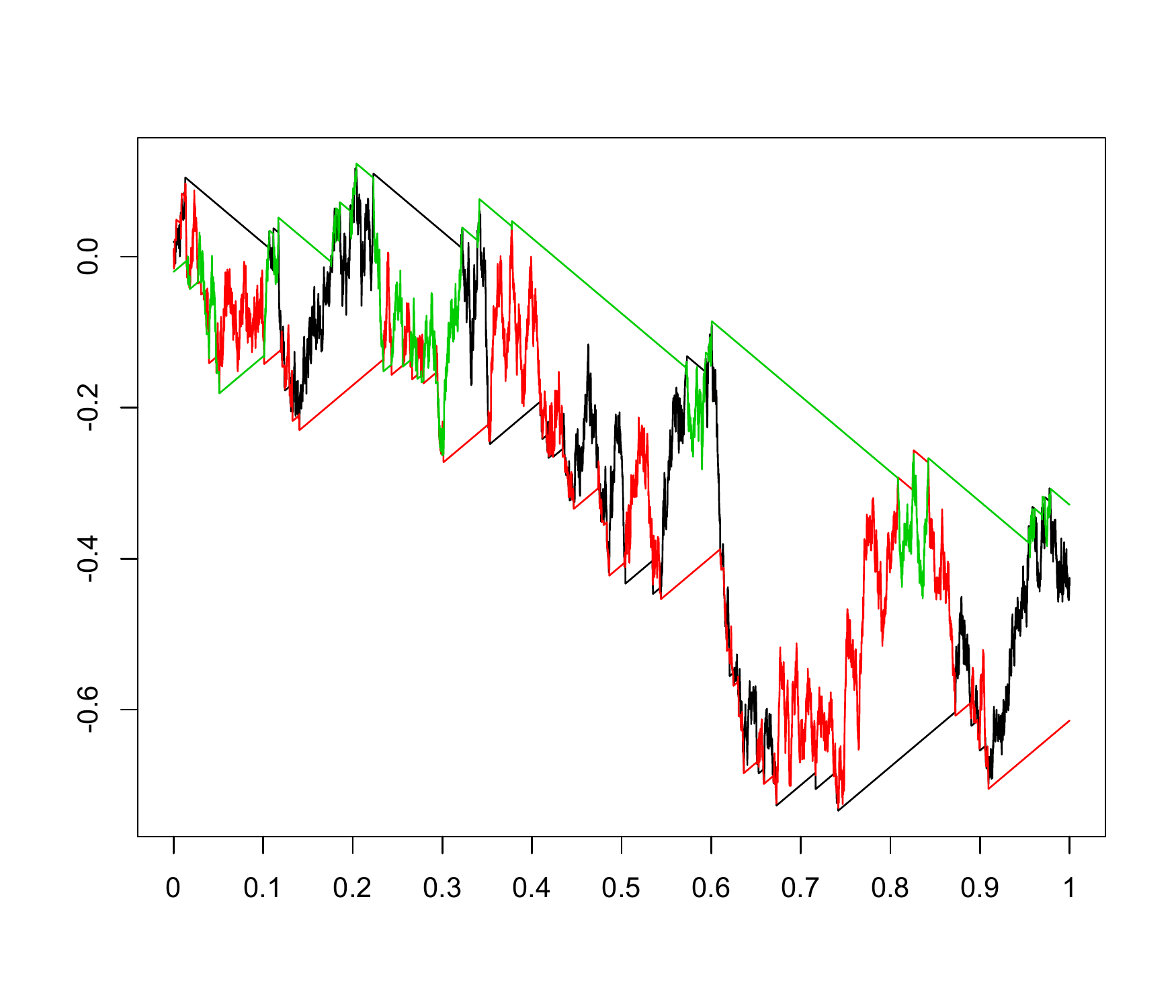}
    \caption{Simulated processes for the system in (\ref{1a}) with $\,\gamma =1\,$: Black $\,=X_1(\cdot)\,$,  Red $\,=X_2 (\cdot)\,$,  Green $\,=X_3(\cdot)$.   The 3-D process  $(X_1(\cdot), X_2 (\cdot), X_3 (\cdot))$ carries the same information content as a scalar Brownian Motion. 
    We are indebted to Dr.\,\textsc{E.R. Fernholz} for this picture, which illustrates the ``ballistic and \textsc{Cantor}-function-like nature of the outer hedges" and the diffusive motion in the middle. 
    } \label{Fig: 1}
  \end{center}
\end{figure}


\subsection{Positive Recurrence,  Ergodicity, Laws of Large Numbers}
 \label{sec4}

We present now a criterion for the process $\,(G(\cdot), H(\cdot))\,$ in (\ref{G})-(\ref{H}) to reach an arbitrary open neighborhood of the origin in finite expected time. We carry out this analysis  along the lines of \citep{MR1198420}. 
The system studied there, 
is a {\it non-degenerate} reflected Brownian motion $\, ({\bm X}_{\cdot}, {\bm Y}_{\cdot}) \,$ in the first orthant driven by a planar Brownian motion $\,({\bm B}_{\cdot}, {\bm W}_{\cdot})\,$, namely 
\begin{equation} 
\label{eq: H-R}
{\bm X}_{t} \, =\, {\bm x} + {\bm B}_{t} + {\bm \mu}\,  t + {\bm L}^{\bm X}_{t} + {\bm \alpha} {\bm L}^{\bm Y}_{t} \, , ~~~~
{\bm Y}_{t} \, =\, {\bm y} + {\bm W}_{t} + {\bm \nu} \,t + {\bm \beta}   {\bm L}^{\bm X}_{t}+ {\bm L}^{\bm Y}_{t}  \, , \quad 0 \le t < \infty \, . 
\end{equation}
Here $\,({\bm x}, {\bm y})\,$ is the initial state in the nonnegative quadrant, and $\,{\bm \mu}\, $, $\, {\bm \nu}\,$, $\,{\bm \alpha}\,$, $\,{\bm \beta}\,$ are real  constants.  A necessary and sufficient condition for the positive recurrence of $\,({\bm X}_{\cdot}, {\bm Y}_{\cdot})\,$ in (\ref{eq: H-R}) is 
\begin{equation} 
\label{eq: P-R}
{\bm \mu} + {\bm \alpha} {\bm \nu}^{-} < 0 \, , \qquad 
{\bm \nu} + {\bm \beta}{\bm \mu}^{-} < 0\,;  
\end{equation}
 and $\,x^{-}= \max(-x, 0)\,$ is the negative part of $\,x \in \mathbb R\,$ 
 (see Proposition 2.3 of \citep{MR1198420}; 
 according the interpretation offered in that paper, the quantities posited to be negative in \eqref{eq: P-R} are the ``effective drifts" of the process in the $x$ and $y$ directions, respectively).

By contrast, our system (\ref{G})-(\ref{H}) is   driven by the single Brownian motion $\,W(\cdot) ,$ thus  {\it degenerate,} and has the form  
\begin{equation} 
\label{eq: H-R1}
{\bm X}_{t} \, =\, {\bm x} - {\bm W}_{t} + {\bm \mu} t + {\bm L}^{\bm X}_{t} + {\bm \alpha} {\bm L}^{\bm Y}_{t} \, , \quad 
{\bm Y}_{t} \, =\, {\bm y} + {\bm W}_{t} + {\bm \nu} t + {\bm \beta}   {\bm L}^{\bm X}_{t}+ {\bm L}^{\bm Y}_{t}  \, , \quad 0 \le t < \infty  
\end{equation}
that one obtains by replacing formally the planar Brownian motion $\,({\bm B}_{\cdot}, {\bm W}_{\cdot})\,$ in (\ref{eq: H-R}) by $\,(-{\bm W}_{\cdot}, {\bm W}_{\cdot})$. The system (\ref{G})-(\ref{H}) can be cast in the form (\ref{eq: H-R1}), if we replace formally the triple $\, ({\bm X}_{\cdot}, {\bm Y}_{\cdot},   
{\bm W}_{\cdot})\,$ by the triple $\, (G(\cdot), H(\cdot), 
W(\cdot))\,$ and substitute $\,{\bm \mu } = - (\delta_{2} - \delta_{1})\,$, $\,{\bm \nu} = - (\delta_{3} - \delta_{2}) \,$, $\, {\bm \alpha} = {\bm \beta} = -1\, /\, 2\,$. 

\subsubsection{Existence, Uniqueness and Ergodicity of an Invariant Distribution}
 
 Here is the main result of this subsection.
 
 \begin{thm} 
 \label{prop: 4.1} 
Under the conditions 
\begin{equation} 
\label{ord2}
2 (\delta_{3} - \delta_{2}) + \big( \delta_{1} - \delta_{2} \big)^{-} > 0 \, , \qquad 
2 (\delta_{2} - \delta_{1}) + \big( \delta_{2} - \delta_{3} \big)^{-} > 0 \,  , 
\end{equation}
the process $\,\big(G( \cdot), H(\cdot)\big) \,$ in (\ref{G})--(\ref{H}) is positive-recurrent, and has a unique invariant measure 
$\, {\bm \pi}\,$ with $\, {\bm \pi}  ( (0,\infty)^2 ) =1\,,$  to which its time-marginal distributions  converge   as $\, t \ra \infty\,$.   
\end{thm}

  Let us note that the condition (\ref{ord2}) is a simple recasting of (\ref{eq: P-R}). It is  satisfied in the special case of (\ref{4}); and more generally, under the strict ordering 
 \begin{equation}
\label{ord}
\delta_1\,<\, \delta_2\,<\, \delta_3 \,.
\end{equation}
 This last requirement 
 is strictly stronger than (\ref{ord2}); for example, the choices $\,\delta_{1} = 1\, /\, 3\,$, $\,\delta_{2} = 0\,$, $\,\delta_{3} = 1\,$ 
satisfy (\ref{ord2}) but    not   (\ref{ord}). We note  also   that (\ref{ord2}) implies $\,\delta_{1} < \delta_{3}\,$, as well as at least one of  $\,\delta_{2} > \delta_{1}\,$, $\,\delta_{3} > \delta_{2}\,$; that is, (\ref{ord2}) excludes the possibility   $\,\delta_{3} \le \delta_{2} \le \delta_{1}\,$. These claims are discussed in detail in   Remark \ref{comp}.

\smallskip
\noindent
{\it Proof:} We start with a remark on the positivity of the transition probability $$\,p^{t}(y, A)  :=  \mathbb P \big( (G(t), H(t)) \in A \, \big \vert (G(0), H(0)\big) = y ) \, ,  \qquad \,y = (y_{1}, y_{2}) \in (0, \infty)^{2} ,   ~~\,A \in \mathcal B  \big((0, \infty)^{2}\big) .$$ We recall that $\,G(\cdot) + H(\cdot) \,$ is of finite first variation,    and decreases monotonically until $\,(G(\cdot), H(\cdot))\,$ hits one of the edges, i.e., $\, {\mathrm d} (G(t) + H(t)) = - (\delta_{3} - \delta_{1}) {\mathrm d} t + (1/2) ({\mathrm d} L^{G}(t) + {\mathrm d} L^{H}(t))\,$, $\, t \ge 0 \,.$ 

Consider now a trapezoid $\,\mathcal T_{a,b} := \{ (x, y) \in [0, \infty)^{2} : - x + a \le  y \le - x + b \} \,$  whose intersection with 
$\,A\,$     has positive \textsc{Lebesgue} measure    $\, \text{Leb} (\mathcal T_{a,b} \cap A ) > 0  ,$ for some $  a, b > 0  $. Two cases arise: 

\noindent {\it (i)} \,
If $\,y_{1} + y_{2} \ge b\,$, then the point $\,y\,$ is located in the north-east of $\,\mathcal T_{a,b}\,$,  and we see that    
 \begin{equation}
\label{posit}
p^{t}(y, A) \ge p^{t}(y, A \cap \mathcal T_{a,b}) > 0 
 \end{equation}
 is valid for every $t \ge (y_{1} + y_{2} - a) \, /\, (\delta_{3} - \delta_{1}) ,$ by considering the paths which do not touch the edges. 
 
 \noindent {\it (ii)} \,
 If $\,y_{1}+ y_{2} < b \,$, then considering the paths for which either $ L^{G}(\cdot)  $ or $ L^{H} (\cdot)  $ exhibit an increase and invoking the \textsc{Markov}  property, we see that (\ref{posit}) holds  for every $\,t > 0\,$. 
 
 In either case, we deduce that for every $\,y \in (0, \infty)^{2}\,$, $\, A \in \mathcal B ( (0, \infty)^{2})\,$ with positive \textsc{Lebesgue} measure $\,\text{Leb} (A) > 0 \,$, there exists a positive real number $\,t^{\ast}\,$  such that $\, p^{t}(y, A) > 0 \, $ holds for every $\,t \ge t^{\ast}.$ This shows that the skeleton \textsc{Markov} chain of the process $\,(G(\cdot), H(\cdot)) \,$ is irreducible. 

\smallskip 
\noindent
$\bullet~$ 
Let us define now, inductively, two sequences of stopping times $\,\tau:= \tau_{1} = \inf \{ s \ge 0 : G(s) = 0\}\,$, $\,\sigma := \sigma_{1} = \inf \{ s \ge \tau : H(s) = 0\}\,$, $\,\tau_{n}:= \inf \{ s \ge \sigma_{n-1} : G(s) = 0\}\,$, $\,\sigma_{n} := \inf\{ s \ge \tau_{n}: H(s) = 0\}\,$ for $\,n=2,3,\ldots\,$. 
Also let us define $\,{\bf T}_{0} := \inf \{ s \ge 0: G(s) H(s) = 0\}\,$ and 
\[
{\bf T}_{\dagger} := \inf \{ s \ge 0: G(s) \le x_{0}, H(s) = 0\} \, , 
\qquad 
{\bf T}_{r} := \inf \{ s \ge 0: (G(s), H(s)) \in \mathrm{B}_{0}(r) \} \, , 
\]
where $\,\mathrm{B}_{0}(r)\,$ is the ball of radius $\,r > 0\,$ centered at the origin. 

Most  of the arguments in \citep{MR1198420} 
carry over   smoothly to the degenerate system (\ref{G})-(\ref{H}). In fact, we can replace $\,B (\cdot)\,$ by $\,-W(\cdot)\,$ in the proof of  Propositions 2.1-2.2 of \citep{MR1198420}, 
and deduce that, under (\ref{ord2}),  there exists a large enough $\,x_{0} > 0 \,$ such that for $\,x_{1} - x_{2} \ge x_{0}\,$ we have 
\[
\mathbb E^{(x_{1}-x_{2}, 0)} [ G(\sigma_{1}) ] \le (x_{1} - x_{2})/2 \, , \qquad \mathbb E^{(x_{1}-x_{2}, 0)} [ \sigma_{1} ] \le 2 C (x_{1} - x_{2}) \, , 
\]
where $\,C\,$ is some positive constant. Moreover, again replacing $\,B (\cdot)\,$ by $\,-W(\cdot)\,$  in the first part of the proof of Proposition 2.3 of \citep{MR1198420},   
we deduce   
$\, \mathbb E [ {\bf T}_{\dagger}] \, \le \, C ( 1 + \sqrt{(x_{1}-x_{2})^{2} + (x_{2} - x_{3})^{2}\,} \,) \, $.  
We claim that there exists a constant $\, {\bm \delta} > 0\,$ such that a uniform estimate 
\begin{equation} \label{eq: unif est}
\inf_{0 < y \le x_{0}} \mathbb P^{(y, 0)} \Big( {\bf T}_{\varepsilon} \le {\bf T}_{2x_{0}} \wedge 1 \Big) \ge {\bm \delta} > 0 \, 
\end{equation}
holds; once  (\ref{eq: unif est}) has been established,   positive recurrence under the condition (\ref{ord2}) will follow. 

Instead of showing (\ref{eq: unif est}),  we shall argue under the condition (\ref{ord2}) that for every $\, 0 < \varepsilon < x_{0}\,$, there exists a positive constant $\, {\bm \delta} > 0\,$ such that 
\begin{equation} 
\label{eq: inf-low-bnd}
\inf_{\varepsilon < y \le x_{0}} \mathbb P^{(y,0)} \Big( \widetilde{\bf T}_{\varepsilon} \le \widetilde{\bf T}_{2x_{0}} \wedge {\bf t}_{0}(y) \Big) \, \ge \, {\bm \delta} > 0 \, , 
\end{equation}
where $\,{\bf t}_{0}(y) := (y-(5/6)\varepsilon) \, / \, (\delta_{3} - \delta_{1} - (1/2)(\delta_{3} - \delta_{2})^{+}) > 0\,$, $\,\varepsilon < y \le x_{0}\,$ and 
\[
\widetilde{\bf T}_{r} := \inf \{ s \ge 0 : G(s) + H(s) = r\} \, , \quad r \ge 0 \, . 
\]

In fact, we shall evaluate the smaller probability $\, \inf_{\varepsilon < y \le x_{0}} \mathbb P^{(y,0)} \big( \widetilde{\bf T}_{\varepsilon} \le \widetilde{\bf T}_{2x_{0}} \wedge {\bf t}_{0}(y)\, ,  \widetilde{\bf T}_{\varepsilon} < \tau \big)  , $  where we recall $\,\tau:= \inf \{ s> 0: G(s) = 0\}\,;$  that is, the probability that the process $\,(G(\cdot), H(\cdot))\,$, starting from the point $\,(G(0), H(0)) = (y,0) \,$  on the axis, with $y  \in (0, x_{0}],$ reaches the neighborhood of the origin  before going away from the origin and before attaining the other axis.

We argue   as follows: The process $\,(G(\cdot), H(\cdot))\,$ does not accumulate any local time $\,A(\cdot)\,$  before  
$\,   \tau\,$: 
\[
0 < G(t) \, =\,  y - (\delta_{2}- \delta_{1}) t - W(t) - (1/2) \Gamma(t) \, , \qquad 
0 \le H(t) \, =\, - (\delta_{3} - \delta_{2}) t + W(t) + \Gamma(t) \, ,  
\]
and consequently  $\, G(t) + H(t) \, =\, y - (\delta_{3}-\delta_{1}) t + (1/2) \Gamma(t) \,$, for $\, 0 \le t \le \tau \, $. From the \textsc{Skohokhod} construction, we obtain the upper bound  
\[
\Gamma(t) \, =\, \max_{0\le s \le t }\big( - W(s) + (\delta_{3} - \delta_{2}) s \big)^{+} \le \max_{0 \le s \le t} (-W(s))^{+} + (\delta_{3} - \delta_{2})^{+} t\, , \quad 0 \le t \le \tau \,  
\]
for the local time $\,\Gamma(\cdot)\,$. Thus we obtain 
\begin{equation} \label{eq: G-lower}
G(t) \ge y - \Big( \delta_{2} - \delta_{1} + \frac{1}{\, 2\, } (\delta_{3} - \delta_{2})^{+} \Big) t - W(t) - \frac{1}{\,2\,} \max_{0 \le s \le t} \big(-W(s)\big)^{+} \, , 
\end{equation}
\begin{equation} \label{eq: G+H-upper}
G(t) + H(t) \le y - \Big( \delta_{3} - \delta_{1} - \frac{1}{\, 2\, } (\delta_{3} - \delta_{2})^{+}\Big)  t + \frac{1}{\, 2\, } \max_{0 \le s \le t} \big(- W(s)\big)^{+} \, ; \quad 0 \le t \le \tau \, . 
\end{equation}

Now let us consider the event 
$$\, 
A(y): \,= \,\Big\{  \omega  \in \Omega : \max_{0 \le s \le {\bm t}_0(y)} \lvert W(s, \omega)\rvert \le \varepsilon \, / \, 3 \Big\}\, , \qquad \,\varepsilon < y \le x_{0}\,.
$$ 
Since $\,\delta_{3} - \delta_{1} - (1/2) (\delta_{3} - \delta_{2})^{+} \le \delta_{2} - \delta_{1} + (1/2) (\delta_{3}-\delta_{2})^{+}\,$,    for every $\,\omega \in A(y)\,$ we obtain 
\[
\min_{0 \le t \le {\bm t}_{0}(y)} \Big[ y - \Big( \delta_{2} - \delta_{1} + \frac{1}{\, 2\, } (\delta_{3} - \delta_{2})^{+} \Big) t - W(t,\omega) - \frac{1}{2} \max_{0 \le s \le t} \big(-W(s,\omega)\big)^{+} \Big] \ge \frac{\varepsilon}{\, 3\, } > 0 \, , 
\]
  hence, combining with (\ref{eq: G-lower}), we obtain $\, A(y) \subset \{ {\bf t}_{0}(y) < \tau\} \,$. Moreover, for every $\,\omega \in A(y) \,$ we have 
\[
\min_{0 \le t \le {\bf t}_{0}(y)} \big( G(t,\omega) + H(t,\omega) \big)
\le \min_{0 \le t \le {\bf t}_{0}(y)} \Big[ y - \Big( \delta_{3} - \delta_{1} - \frac{1}{\, 2\, } (\delta_{3} - \delta_{2})^{+}\Big)  t + \frac{1}{\, 2\, } \max_{0 \le s \le t} \big(- W(s,\omega)\big)^{+} \Big] \le \varepsilon \, , 
\]
thus also  
$
\max_{0 \le t \le {\bf t}_{0}(y)} \big(G(t,\omega) + H(t,\omega)\big)  
\le x_{0} + \varepsilon < 2 x_{0} \, . 
$ 

\smallskip
We deduce the set-inclusion $\,A(y) \subset \{ \widetilde{\bf T}_{\varepsilon} \le  \widetilde{\bf T}_{2x_{0}} \wedge {\bf t}_{0}(y), \widetilde{\bf T}_{\varepsilon} < \tau \}  ,$  so the reflection principle  for Brownian motion gives 
\[
\inf_{\varepsilon < y \le x_{0}} \mathbb P^{(y,0)} ( \widetilde{\bf T}_{\varepsilon} \le \widetilde{\bf T}_{2x_{0}} \wedge {\bf t}_{0}(y) ) \ge \inf_{\varepsilon < y \le x_{0}} \mathbb P^{(y,0)} ( \widetilde{\bf T}_{\varepsilon} \le \widetilde{\bf T}_{2x_{0}} \wedge {\bf t}_{0}(y) , \widetilde{\bf T}_{\varepsilon} < \tau) \]
\[
\ge \inf_{\varepsilon < y \le x_{0}} \mathbb P^{(y,0)}(A(y)) 
\ge 1 - \Big(\frac{{\bf t}_{0}(x_{0})}{\, 2\pi\, }\Big)^{1/2} \cdot \frac{4}{\, \varepsilon / 3\, } \cdot \exp \Big( - \frac{ (\varepsilon / 3)^{2}}{2 {\bf t}_{0}(x_{0})} \Big) \,  
\] 
\noindent
(cf. Problem 2.8.2 of \citep{MR1121940}
). Selecting $\varepsilon \in (0,1)$ small enough so that   this right-most expression is positive, and denoting it by $\,{\bm \delta} > 0 ,$  we obtain (\ref{eq: inf-low-bnd}).  We appeal now to the second half   of the proof of 
Proposition 2.3 in \citep{MR1198420}, 
page 393, and conclude that the system (\ref{G})-(\ref{H}) is  positive-recurrent for neighborhoods,  under (\ref{ord2}). %
 
\smallskip
For the remaining claims of the Theorem, let us recall the equations of (\ref{G})-(\ref{H}) written in the \textsc{Harrison \& Reiman} \citep{MR606992} 
form (\ref{eq: mfrakQ})-(\ref{mA}), and note that the process $\, \mathfrak{Z} (\cdot) \,$ of (\ref{Z}) has independent, stationary increments with $\, \mathfrak{Z} (t) =0\,$ and $\,\mathbb{E}\big|  \mathfrak{Z} (1) \big| < \infty\,$. Now, as   is relatively easy to verify (and  shown in Remark \ref{comp}),   the conditions of (\ref{ord2}) imply that the components of the vector
\begin{equation} 
\label{ord4}
- \,\mathcal{R}^{-1}\, \mathbb{E}\big(  \mathfrak{Z} (1) \big) \,=\, {2 \over \,3\,} \begin{pmatrix}
      2 & 1   \\
     1  &  2
\end{pmatrix} \begin{pmatrix}
      \delta_2 - \delta_1    \\
      \delta_3 - \delta_2  
\end{pmatrix} \,=\, {2 \over \,3\,} \begin{pmatrix}
       \delta_2 + \delta_3 - 2 \delta_1   \\
         2 \delta_3  - \delta_1 - \delta_2
\end{pmatrix} \, =:\, \begin{pmatrix}
      \lambda_1    \\
      \lambda_2  
\end{pmatrix} \,=\, {\bm \lambda}  
\end{equation} 
are both strictly positive; cf.$\,$(\ref{ord3}) below. Then Corollary 2.1 in \citep{MR2931282} 
implies that   the planar process   $\,\mathfrak{G} (\cdot) =\big(G( \cdot), H(\cdot)\big) \,$   is positive recurrent, has a unique invariant probability measure $\, {\bm \pi}  ,$  and converges to this measure in distribution as $\, t \ra \infty\,$.  In addition, for any bounded, measurable function  $ f: [0,\infty)^2 \ra \R $   we have the ergodic behavior (strong law of large numbers)
$$
\lim_{T \ra \infty} { 1 \over \,T\,} \int_0^T f \big( G(t), H(t) \big) \, \mathrm{d} t \,=\, \int_{[0,\infty)^2} f(g,h)\,   {\bm \pi} ( \mathrm{d}g, \mathrm{d}h)\,, \quad \text{a.e.}
$$
 The claim  $\, {\bm \pi} \big( (0,\infty)^2\big) =1\,$ follows now from     (\ref{9}) .  \qed

\begin{prop}
{Under the conditions of (\ref{ord2}),  the local times accumulated at the origin by the ``gap" processes $\, G (\cdot)\, $ and $\, H(\cdot)\,$ satisfy in the notation of (\ref{ord4}) the  strong laws of large numbers} 
\begin{equation}
\label{SLLN}
\lim_{t \ar} \frac{\,L^G (t)\,}{t} \, =\,\lambda_1  \,,\qquad  ~~~~~\lim_{t \ar} \frac{\,L^H (t)\,}{t}  \, =\,\lambda_2\,, \qquad a.e.
   \end{equation}
\end{prop}

\noindent
{\it Proof:}    
As we just argued, under the condition (\ref{ord2})    the two-dimensional process $\, (G(\cdot), H(\cdot))\,$ of gaps has a unique invariant probability measure $\, {\bm \pi}\,$ on $\, {\cal B} \big( (0,\infty)^2\big) \,$, to which it converges in distribution. This implies, {\it a fortiori}, that  
\begin{equation}
\label{stab}
\lim_{t \ar} \frac{ \,G (t)\,}{t}\, =\,0 \qquad \hbox{ and } \qquad  \lim_{t \ar} \frac{ \,H (t)\,}{t}\, =\,0
\end{equation}
hold  in distribution, thus also in probability. Back into (\ref{SkorU1}), (\ref{SkorV1}) and in conjunction with the  law of large numbers for the Brownian motion $\, W(\cdot)\,$, these observations give that 
\begin{equation}
\label{SLLN2}
\lim_{t \ar} \frac{\,2\,L^G (t) -  L^H (t)\,}{2\,t} \,=\,  \delta_2 - \delta_1  \,, \qquad \lim_{t \ar} \frac{\,2\, L^H (t) - L^G (t)\,}{2\,t}  \,=\,  \delta_3 -   \delta_2    
\end{equation}
hold in probability, and thus the same is true of (\ref{SLLN}). There exist then   sequences $\, \{ t_k \}_{k \in \N} \subset (0, \infty) \,$ and  $\, \{ \tau_k \}_{k \in \N}  \subset (0, \infty) \,$  which increase strictly to infinity, and  along which we have, a.e., 
$$
\lim_{k \ar} \frac{\,L^G (\tau_k)\,}{\tau_k}  \,=\,{ 2 \over \,3\,} \big( \delta_2 + \delta_3 - 2\, \delta_1  \big)\,, \qquad \lim_{k \ar} \frac{\,L^H (t_k)\,}{t_k} \,=\,{ 2 \over \,3\,} \big( 2\, \delta_3 - \delta_1 - \delta_2 \big).
$$
 Theorem II.2 in \citep{MR222955} 
 gives that   $\, \lim_{t \ar} $  $\big( L^G (t)/ t\big)\,$ and $\, \lim_{t \ar} \big( L^H (t)/ t\big)\,$ do exist almost everywhere 
 (see also  \citep{MR912049}, 
 sections 7, 8).  It follows   that the limits in (\ref{SLLN}) are valid not just in probability, but also  almost everywhere; 
  the same is true then for those of (\ref{stab}).   \qed

\begin{figure}
\begin{minipage}{0.45\linewidth}
\includegraphics[scale=0.45]{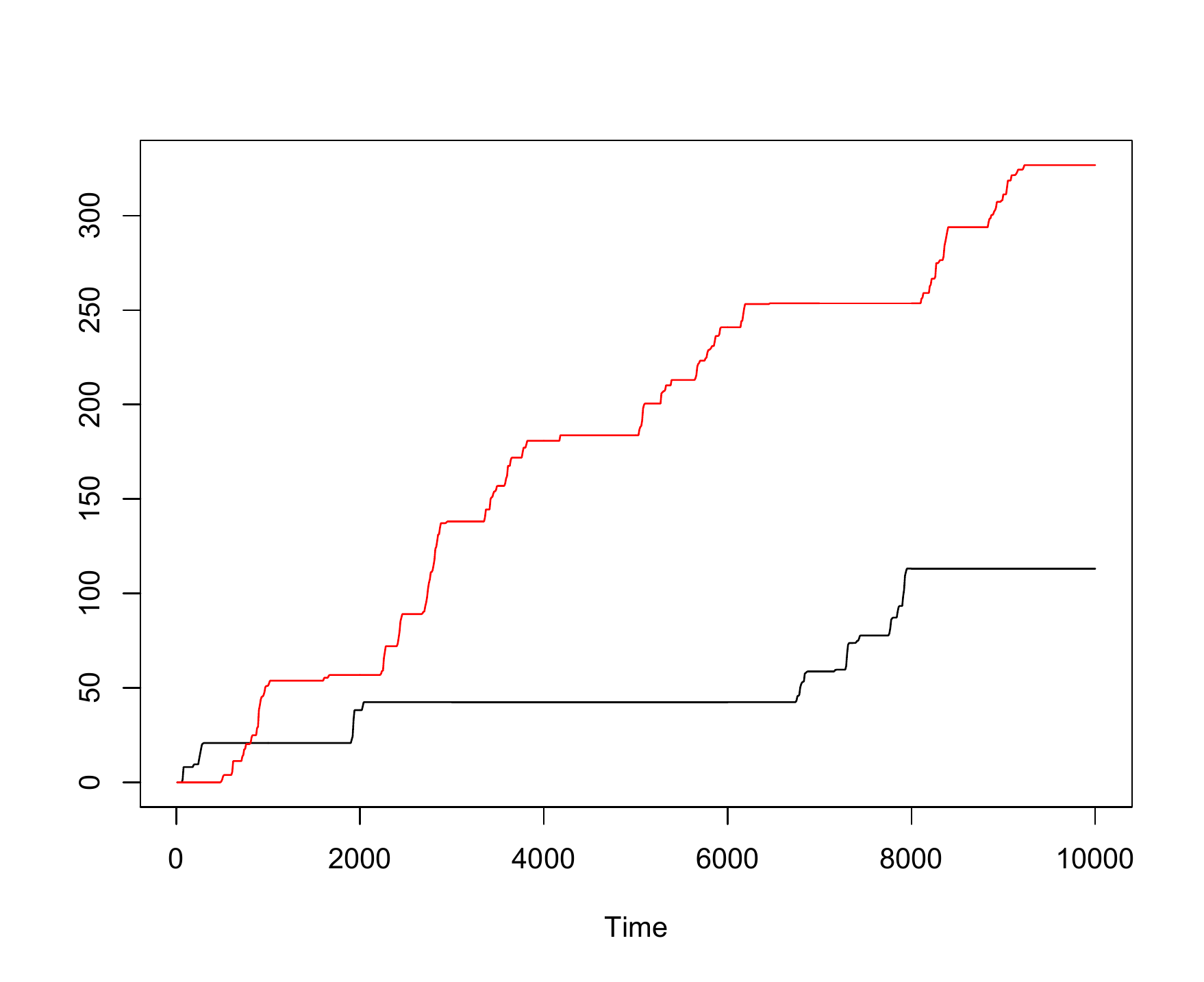}
\end{minipage}	
\begin{minipage}{0.09\linewidth}
\hspace{1cm}
\end{minipage}
\begin{minipage}{0.45\linewidth}
\includegraphics[scale=0.40]{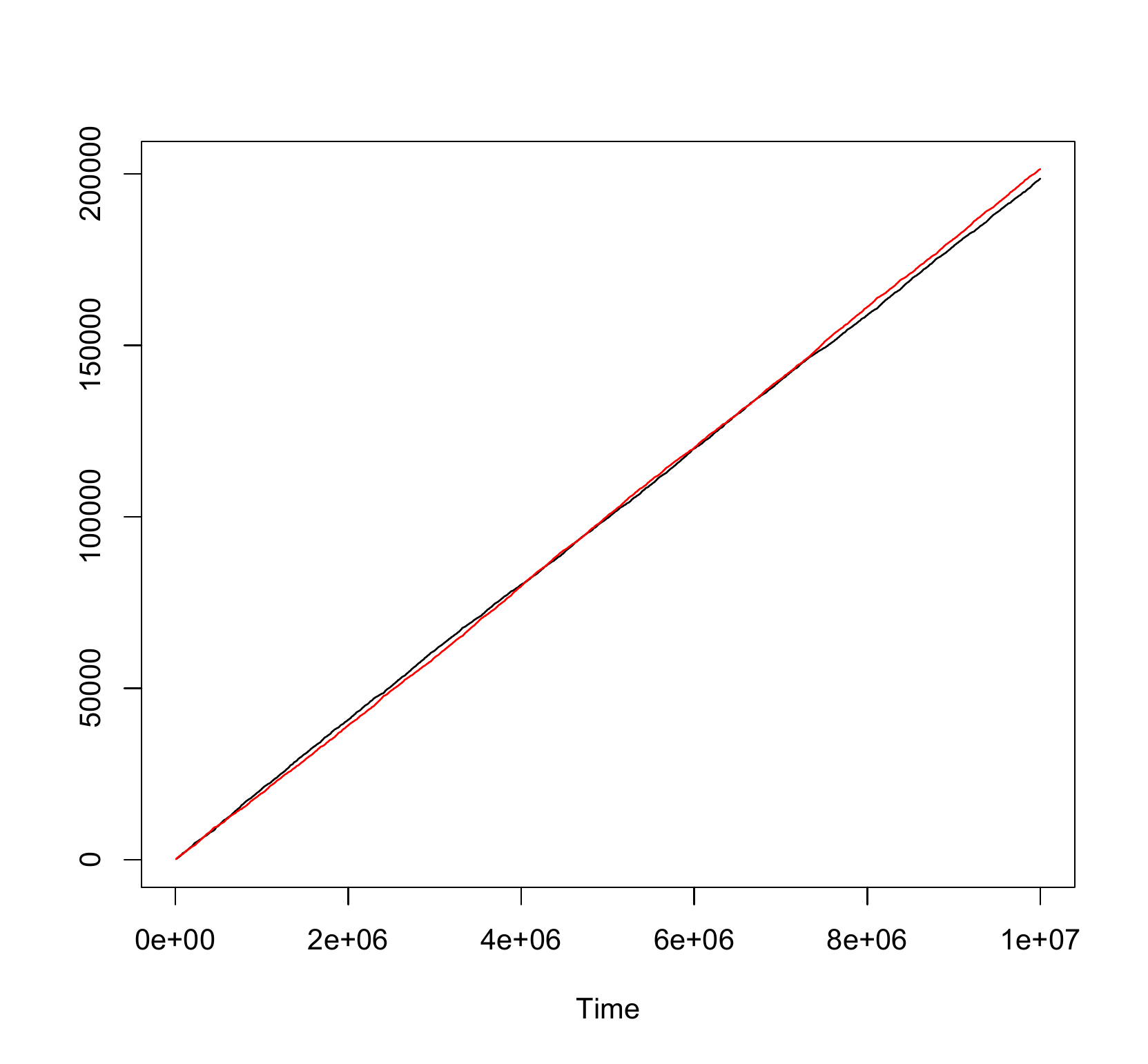}
\end{minipage}	
\caption{Simulated local times $\,L^{G}(\cdot)\,$ (black) and $\,L^{H}(\cdot)\,$ (red) for a short time (left panel) and for a long time (right panel) with $\,\delta_{1} \, =\, 0.01\,$, $\,\delta_{2} \, =\, 0.02\,$, $\,\delta_{3} \, =\, 0.03\,$, thus $\, \lambda_1 =\lambda_{2} \, =\, 0.02\,$. The long-term growth rates converge  in the manner of (\ref{SLLN}), as the time-horizon increases; whereas over short time horizons, the \textsc{Cantor}-function-like nature of local time becomes quite evident.} \label{fig: LT}
\end{figure}

We note that both limits in (\ref{SLLN}) are equal to $ 2  \gamma $ in the special case (\ref{4}); and as a sanity check, we verify in Remark \ref{comp} that both these limits are   strictly positive under the conditions of (\ref{ord2}).  Let us also note that the range of the particles' configuration is a process of finite variation:
$$
R^X_1 (t) - R^X_3 (t)\,=\, G(t) + H(t) \,=\, x_1 - x_3 - \big( \delta_3 - \delta_1 \big)\, t \,+ {1 \over \,2\,} \big( L^G (t) + L^H (t) \big)\,, \quad 0 \le t < \infty.
$$
 It has a linearly decreasing component 
(with slope $\, \delta_3 - \delta_1 $); and a component which increases in \textsc{Cantor}-function-like fashion (as the sum  of local times), a behavior that can be gleaned very clearly  from   Figure 1. 
It satisfies    $\, \lim_{t \ar} \big( \big(  R^X_1 (t) - R^X_3 (t) \big) \big/ t \,\big) =0\,$     a.e., on the strength of (\ref{stab}).

The long-term   growth rates of the local times $\,(L^{G}(\cdot), L^{H}(\cdot))\,$ in (\ref{SLLN}) are  consistent with simulated local times based on the \textsc{Skorokhod} map in \citep{MR606992}. 
The simulations,  reported in Figure \ref{fig: LT}, demonstrate    the  long-term linear growth of these local times   with the rates of (\ref{SLLN}). 

\begin{rem}
[Discussion of Condition (\ref{ord2}), and a Sanity Check]
 \label{comp}
We note that if $\,\delta_{1} \ge\delta_{2} \ge \delta_{3}\,$, the conditions of  (\ref{ord2}) cannot hold; this is   because we have then
\[
2(\delta_{2} - \delta_{1}) + \big(\delta_{2}-\delta_{3}\big)^{-} = \,2(\delta_{2} - \delta_{1}) \le 0 \, , \qquad  
2(\delta_{3} - \delta_{2}) + \big(\delta_{1}-\delta_{2}\big)^{-} = \,2(\delta_{3} - \delta_{2}) \le 0 \, . 
\]
Thus, under (\ref{ord2}), we have either $\,\delta_{2} > \delta_{1}\,$ or $\,\delta_{3} > \delta_{2}\,$. Only three cases are compatible with (\ref{ord2}): 
\[
(\text{i})\, \,  \delta_{1} < \delta_{2} < \delta_{3} \, , \qquad 
(\text{ii})\, \,  \delta_{2} > \delta_{1} ~\text{ and }~ \delta_{2} \ge \delta_{3} \, , \qquad (\text{iii})\, \,  \delta_{3} > \delta_{2} ~\text{ and }~ \delta_{1} \ge \delta_{2} \, . 
\]
It can be shown  that,  in all these cases, the conditions of (\ref{ord2}) imply 
\begin{equation}
\label{ord3a}
\,\delta_{3} > \delta_{1}\,,
\end{equation}
\begin{equation}
\label{ord3}
2 \delta_{3} - \delta_{1} - \delta_{2} > 0\,, \qquad \,\delta_{2} + \delta_{3} - 2 \delta_{1} > 0\,;
\end{equation}
then, the a.e. limits in (\ref{SLLN}) are positive.

The  inequalities of (\ref{ord3}) imply  both (\ref{ord3a}) and  (\ref{ord2}).  As observed by the referee, the condition \eqref{ord3} has the interpretation that if any partition of $\,\{1, 2, 3\} \,$ into two subsets of consecutive indices, the leftmost group of particles has a larger average drift than the rightmost group; cf.  \cite{MR2473654}, page 2187. 
\end{rem}

\subsubsection{Exponential Convergence}


Elementary stochastic calculus applied to the equations (\ref{SkorU1}),  (\ref{SkorV1}) leads to   dynamics 
\begin{equation}
\label{NoLocTimes}
\dx \big( G^2 (t) + G (t) H(t) + H^2 (t) \big)\,=\, \Big[ \, 1 - { \, 3 \,\over 2} \big( \lambda_1 G(t) + \lambda_2 H(t) \big)   \Big] \dx t + \big( H(t) - G(t) \big) \, \dx W(t) 
\end{equation}
devoid of local time terms.  As   $\, \mathfrak{G} (t)= (G(t), H(t))\,$ approaches the origin, the drift in this expression gets close to 1 and pushes the process away from the origin; on the other hand, when either of the components of the vector $  \mathfrak{G} (t) $ gets very large, there is a strong negative drift in the above expression (\ref{NoLocTimes}), which tends to bring the planar process $\mathfrak{G} (\cdot)$ back toward the origin.  This behavior is consistent with the existence of an invariant probability measure for  $\mathfrak{G} (\cdot)$.

We consider now the   function 
\begin{equation}
\label{Lyap}
 V(g,h) \,:=\, \exp\big\{ \sqrt{ g^2 + g h + h^2\,} \,\big\}\, , ~~~~~ \, (g,h) \in [0, \infty)^2 \setminus (0,0)  .
 \end{equation}
This   is of class ${\cal C}^\infty$ on its domain, and  helps us strengthen the conclusions of Theorem 2.3  as follows.

\begin{prop}  
\label{Lyap_func}
Under the  assumptions of Theorem \ref{prop: 4.1},     the function in (\ref{Lyap})   is \textsc{Lyapunov}   for the process $ \mathfrak{G} (\cdot)= (G(\cdot), H(\cdot))  $ in (\ref{G})--(\ref{H}); i.e., there exist constants $\, a, b, \kappa > 0 \,$ such that      
\begin{equation} 
\label{eq: Lyapunov function}
\mathcal Z (\cdot)   :=  V \big(\mathfrak{G} (\cdot)\big) - V\big(\mathfrak{G} (0)\big) + \int^{\cdot}_{0} \big(   \kappa \cdot V \big(\mathfrak{G} ( t)\big) - b \cdot {\bf 1}_{\mathcal T_{a}} \big(\mathfrak{G} ( t)\big)  \big) {\mathrm d} t 
\end{equation}
is a supermartingale, for $\,\mathcal T_{a} \, :=\, \big\{(g,h) \in [0, \infty)^{2}   : g + h \le a \big\}\,.$     In particular, the time-marginal distributions of the positive-recurrent process  $  \mathfrak{G} (\cdot) $ of gaps between ranks,    converge exponentially fast in total variation to the unique invariant probability measure $  {\bm \pi} $ of the process.  
\end{prop}

\noindent {\it Proof:} Applying \textsc{It\^o}'s formula to $\,V(\mathfrak{G} ( \cdot)) \,$ in conjunction with \eqref{NoLocTimes}, we obtain the semimartingale decomposition $\,V(\mathfrak{G} ( t))   =  V(\mathfrak{G} (0)) + M^{V}(t) + A^{V}(t) \,$, where 
\begin{equation*}
\begin{split}
M^{V}(t)   :=  & \int^{t}_{0} \frac{\, V(\mathfrak{G} (s)) \cdot ( G(s)  - H(s)) \,}{\,2 \sqrt{G^{2}(s) + G(s) H(s) + H^{2}(s)} \,} \,{\mathrm d} W(s)   , \quad ~~
A^{V}(t)   :=    \int^{t}_{0} [\mathcal A V] (\mathfrak{G} (s)) \,{\mathrm d} s   , \\
[\mathcal A V](g , h) \, :=\,& \frac{\,V(g, h) \,}{\,2 \sqrt{g^{2}+gh + h^{2}}\,} \Big[ 1 - \frac{\,3\,}{\,2\,} ( \lambda_{1} g + \lambda_{2} h ) \Big]  + \frac{\,V(g, h) ( g- h)^{2}\,}{\,8 (g^{2} + gh + h^{2})\,} - \frac{\,V(g,h) ( g-h)^{2}\,}{\,8 ( g^{2} + gh + h^{2})^{3/2}\,} . 
\end{split}
\end{equation*}
For arbitrary small $\,\varepsilon > 0 \,$ and sufficiently large $\, a > 0 ,$  the drift function $\, [\mathcal A V] (g , h) \,$ satisfies  
\[
[\mathcal A V ] (g, h) \le - \kappa \cdot V(g , h) + b \cdot {\bf 1}_{\mathcal T_{a}}  (g , h)  \, ; \quad (g, h) \in \mathcal T_{\varepsilon, \infty}
\] 
for some $\,\kappa \, :=\, (3/4) \min ( \lambda_{1}, \lambda_{2})  > 0  ,$  $\, b \, :=\, \sup_{} \{ V(g, h) ((1/8) + 1 / (2 \sqrt{g^{2} + g h + h^{2}} )) : (g, h) \in \mathcal T_{\varepsilon, a} \}, $  with the trapezoids $\,\mathcal T_{\varepsilon, a} := \{ (x, y) \in [0, \infty)^{2} : - x + \varepsilon \le y \le - x + a \} \,$ and $\, \mathcal T_{\varepsilon, \infty} \, :=\, \{ (x, y) \in [0, \infty)^{2} : - x + \varepsilon \le y \} \,$. Then $\, \mathcal Z (\cdot) \,$ in \eqref{eq: Lyapunov function} is a local supermartingale satisfying $\, \mathcal Z (t) \ge - V(\mathfrak{G} (0) ) - b t \,$ for $\,t \ge 0 \,$,   hence  a supermartingale by \textsc{Fatou}'s lemma.

The function $ V $ of \eqref{Lyap}  and its derivatives are not defined at the origin; but by Proposition \ref{notriple}, the process  $\mathfrak{G} (\cdot) $ does not attain the origin   when started away from it.   Proposition 3.1 in \citep{MR2499863} 
(and its references, as well as Definitions 5, 6 of \citep{MR3687255} 
in the context of SRBM), shows that  $  V   $ is  a \textsc{Lyapunov} function. As in   the proof of Theorem \ref{prop: 4.1}, $\mathfrak{G} (\cdot)$ has an irreducible skeleton chain  and hence, by Theorem 6.1 of \citep{MR1234295}, 
is aperiodic.  We appeal now to the     results in 
\citep{MR3687255} 
(cf.\,
\citep{MR0133871}, 
\citep{MR1234295}, 
\citep{MR1288127}, 
\citep{MR2499863}, 
also \citep{MR2650048}, 
\citep{MR2830609} 
for fluid paths). These  show that   $ \mathfrak{G} (\cdot) $ is positive recurrent,    has a unique invariant distribution, and is $V-$uniformly ergodic. \qed

 \smallskip

\subsection{
Basic Adjoint Relation   (BAR)  and \textsc{Laplace} Transforms
}
 \label{sec: BAR}

{\it Under the condition (\ref{ord2}), can the invariant probability measure $\, \pib\,$ of the two-dimensional process $\, (G(\cdot), H(\cdot))\,$ of gaps   be computed explicitly?}  

We do not know the answer to this question, but will    try to make some progress on it in the present subsection. We shall show that the joint distribution in steady-state for the pair of gaps $\,(G(\cdot), H(\cdot))\,$, cannot be the product of their marginals (Remark \ref{rem: no product form}); and,  in  
what we call the ``symmetric case" \eqref{sym} for this problem, that this joint distribution of gaps is determined by the distribution of their sum $\,G(\cdot) + H(\cdot) \,$ in steady-state. For the latter, and always in the ``symmetric case'', we offer in Remark \ref{conj} a conjecture that involves the Gamma distribution. 

\smallskip
Let us start, then, by observing that for every bounded continuous function $\,f : [0, \infty)^2 \ra \R\,$ of class $\, \mathcal{C}^{2}_{b}\big((0, \infty)^{2}\big) \cap \mathcal{C}^{1}_{b}\big([0, \infty)^{2} \setminus \{ 0\}\big)$, simple stochastic calculus gives 
\[
f(\mathfrak G(T)) - f(\mathfrak G(0)) \, =\, \int^{T}_{0} \nabla f(\mathfrak G(t)) \cdot {\mathrm d} \big(\mathfrak Z(t) + \mathcal R \mathfrak L(t) \big) + \int^{T}_{0} \frac{1}{\, 2\, } \big(D_{gg}^{2} + D_{hh}^{2} - 2 D_{gh}^{2}\big) f (\mathfrak G (t)) {\mathrm d} t 
\]
where the processes $\, \big(\mathfrak G(\cdot), \mathfrak Z(\cdot), \mathfrak L(\cdot)\big)\,$ and the matrix $\,\mathcal R\,$ are defined in (\ref{eq: mfrakQ})-(\ref{mA}). Taking   expectation on both sides,   then integrating with respect to the invariant probability measure $\, {\bm \pi} \,$  for the planar process $\,\mathfrak G(\cdot) \, =\, \big(G(\cdot), H(\cdot)\big)  $, we obtain by \textsc{Fubini}'s theorem for $\, 0 <T < \infty\,$ the equation 
\begin{equation}
\label{preBAR}
0 \, =\, T   \int_0^\infty \int_0^\infty  \Big[\, \frac{1}{\, 2\, }  \big(D_{gg}^{2} + D_{hh}^{2} - 2 D_{gh}^{2}\big) f(g, h) + {\bm m} \cdot \nabla f(g,h) \,\Big]\, {\bm \pi}({\mathrm d} g, {\mathrm d} h) \qquad \qquad \qquad 
\end{equation}
\[
\qquad \qquad {} + \frac{T}{\, 2\, } \bigg( \int_0^\infty \Big( D_{g} - \frac{1}{\, 2\, } D_{h}  \Big) f(0, h)\,  {\bm \nu}_{1}({\mathrm d} h) +  \int_0^\infty \Big( D_{h} - \frac{1}{\, 2\, } D_{g}  \Big) f(g, 0) \, {\bm \nu}_{2}({\mathrm d} g) \bigg).
\]

We have denoted   by $\, {\bm m}  = \big(  \delta_{1} - \delta_{2} \,,  \,\delta_{2} - \delta_{3}\big)^{\prime}\,$   the drift vector of $\,\mathfrak Z(\cdot)\,$ in (\ref{Z}), (\ref{mA}); and by  $\,{\bm \nu}_{1}\,$ (respectively, by $\,{\bm \nu}_{2}\,$)    the $\,\sigma$-finite measure on the axis $\,\{(g , h) \in [0, \infty)^{2} :  g = 0\}\,$ (respectively, on the axis  $\,\{(g , h) \in [0, \infty)^{2} :  h = 0\}\,$) induced by the vector  $\,\mathfrak L(\cdot) = \big( L^G (\cdot), L^H (\cdot) \big)'\,$ of local times under the invariant probability measure $\,{\bm \pi}\,.$  Namely, 
\begin{equation} 
\label{eq: add L}
\mathbb E_{{\bm \pi}}   \int^{T}_{0} f \big(\mathfrak G(t)\big) \,{\mathrm d} {\mathfrak L}(t)  \, =\, \frac{T}{\, 2\, } \, \bigg( \int_0^\infty f(0, h) \, {\bm \nu}_{1}( {\mathrm d} h) \, , \, \int_0^\infty f(g, 0)\, {\bm \nu}_{2}( {\mathrm d} g)  \bigg)^{\prime},  
\end{equation}
 or equivalently
\begin{equation} 
\label{eq: add L_too}
 {\bm \nu}_{1}(A) \,=\,  \mathbb{E}_{{\bm \pi}} \int_0^2 \mathbf{ 1}_{ \{ H(t)\in A\}} \, \dx L^G (t)\,, \qquad   {\bm \nu}_{2}(A) \,=\,  \mathbb{E}_{{\bm \pi}} \int_0^2 \mathbf{ 1}_{ \{ G(t)\in A\}} \, \dx L^H (t)
\end{equation}
 for $\, A \in \mathcal{B} \big( (0, \infty)\big) $; 
see  \citep{MR222955}.  
Dividing both sides of (\ref{preBAR}) by $\,T /   2\,$, we obtain for the invariant probability measure $\,{\bm \pi}\,$ of   
$\, (G(\cdot), H(\cdot))\,$ in (\ref{G})-(\ref{H}) the {\it Basic Adjoint Relation} (BAR)   
\begin{equation} 
\label{eq: BAR}
\int_0^\infty \int_0^\infty \Big( \big(D_{gg}^{2} + D_{hh}^{2} - 2 D_{gh}^{2}\big)    + 2 \big( \delta_1 - \delta_2 \big) D_g + 2 \big( \delta_2 - \delta_3 \big) D_h   \Big) f(g,h)\, {\bm \pi}({\mathrm d} g ,  {\mathrm d} h) +~~~~~~~~~ 
\end{equation}
\[
~~~~~~~~~~~~~~~~~~~~{}+ \int_0^\infty \Big( D_{g} - \frac{1}{\, 2\, } D_{h}  \Big) f(0, h)\, {\bm \nu}_{1}({\mathrm d} h) +  \int_0^\infty \Big( D_{h} - \frac{1}{\, 2\, } D_{g}  \Big) f(g, 0) \, {\bm \nu}_{2}({\mathrm d} g) \, =\, 0  \,.
\]
   This relationship   was 
   studied in detail  for non-degenerate reflected Brownian motions by \citep{MR912049}. 
 As shown in \citep{MR1873294}, 
 a probability measure $\, {\bm \pi}\,$ on $\, {\cal B} \big( (0,\infty)^2 \big)\,$ is invariant for 
 $ \mathfrak G(\cdot) = (G(\cdot), H(\cdot) )  $ of gaps  if it,  together with two finite measures $\, {\bm \nu}_1\,$, $\, {\bm \nu}_2\,$ on $\, {\cal B}  ( (0,\infty)   ) ,$ satisfies the BAR (\ref{eq: BAR}).   

\subsubsection{ \textsc{Laplace} Transforms and Ramifications  }
 \label{sec: Rami}

The Basic Adjoint Relation of (\ref{eq: BAR}) allows us to express  the \textsc{Laplace} transform  $\,  \widehat{{\bm \pi}}   \,$ of the invariant probability measure $\,  {\bm \pi} \,$ in terms of the \textsc{Laplace} transforms $\,    \widehat{{\bm \nu}}_{1} , \,\widehat{{\bm \nu}}_{2}   \,$ of the measures $\,   {\bm \nu}_{1}, \,{\bm \nu}_{2}\,$ in (\ref{eq: add L}), (\ref{eq: add L_too}) on the axes. Indeed, substituting $\,f(g,h) \, =\, \exp ( - \alpha_{1} g - \alpha_{2} h) \,$ into (\ref{eq: BAR}) with $\, \alpha_1 \ge 0\,$, $\, \alpha_2 \ge 0\,$, we see that the \textsc{Laplace} transforms 
\[
\widehat{{\bm \pi}} ({\bm \alpha}) \, :=\, \widehat{{\bm \pi}} (\alpha_{1}, \alpha_{2}) \, =\, \mathbb E_{{\bm \pi}} \Big[ e^{- \alpha_{1} G(t) - \alpha_{2} H(t)} \Big]  \, , \qquad ~~
\widehat{{\bm \nu}}_{i} (\alpha_{j}) \, :=\, \int_0^\infty e^{-\alpha_{j} x}  {\bm \nu}_{i} ({\mathrm d} x)  
\]
for $\,i \neq j \in \{ 1, 2 \} \,$ satisfy the equation 
$$
\Big[ (\alpha_{1} - \alpha_{2})^{2}  + 2 (\delta_{2}-\delta_{1}) \alpha_{1} + 2 (\delta_{3}-\delta_{2}) \alpha_{2}\Big] \,\widehat{{\bm \pi} }(\alpha_1, \alpha_2)\, =\, ~~~~~~~~~~~
$$
\begin{equation} 
\label{eq: Lap0}
~~~~~~~~~\, =\, \Big( \alpha_{1} - \frac{\alpha_{2}}{\, 2\, } \Big) \,\widehat{{\bm \nu}}_{1}  (\alpha_{2}) + \Big( \alpha_{2} - \frac{\alpha_{1}}{\, 2\, } \Big) \,\widehat{{\bm \nu}}_{2}  (\alpha_{1}) \, . 
\end{equation}

\smallskip
\noindent
The following observations,  
in the form of bullets, are consequences of  this last equation (\ref{eq: Lap0}).

 \smallskip
\noindent $\,\bullet\,$ For any pair $\,  (\alpha_{1}, \alpha_{2}) \in [0, \infty)^{2}\, $ that satisfies $\,(\alpha_{1} - \alpha_{2})^{2}  + 2 (\delta_{2}-\delta_{1}) \alpha_{1} + 2 (\delta_{3}-\delta_{2}) \alpha_{2}  \neq 0\,$, the equation (\ref{eq: Lap0}) yields 
\begin{equation}
 \label{eq: Lap} 
\mathbb E_{{\bm \pi}} \Big[ e^{- \alpha_{1} G(t) - \alpha_{2} H(t)} \Big] \, =\,  \frac{(2\alpha_{1} - \alpha_{2}) \widehat{{\bm \nu}}_{1} (\alpha_{2}) + (2\alpha_{2} - \alpha_{1}) \widehat{{\bm \nu}}_{2} (\alpha_{1})}{\, 2 (\alpha_{1} - \alpha_{2})^{2} + 4 (\delta_{2}-\delta_{1}) \alpha_{1} + 4 (\delta_{3}-\delta_{2}) \alpha_{2} \,  } \, =\, \widehat{{\bm \pi}} (\alpha_{1}, \alpha_{2})\, . 
\end{equation}
Consequently, the invariant distribution $\, {\bm \pi}\,$ on $\, (0, \infty)^2\,$ can be obtained from the measures $\, {\bm \nu}_1\,$, $\, {\bm \nu}_2\,$ of (\ref{eq: add L}) on the two axes.

The reverse is also true: Setting $   \alpha_1 = 2 \alpha_2>0 $ (resp., $  \alpha_2 = 2 \alpha_1>0 $) in  \eqref{eq: Lap}, we get respectively, 
$$
\widehat{\nu_1} (\alpha_2) \,=\, \frac{2}{3}\, \Big( \alpha_2 + 2\big( \delta_2  + \delta_3 \big)- 4 \delta_1 \Big) \, \widehat{\pi} \big( 2 \alpha_2, \alpha_2 \big)\,, \qquad \widehat{\nu_2} (\alpha_1) \,=\, \frac{2}{3}\, \Big( \alpha_1 - 2\big( \delta_2  + \delta_1 \big)+ 4 \delta_3 \Big) \, \widehat{\pi} \big(   \alpha_1, 2 \alpha_1 \big)\,.
$$

\noindent $\,\bullet\,$ If $\,(\alpha_{1}, \alpha_{2}) \in [0, \infty)^{2}\,$ lies on the segment of the parabola 
\begin{equation} 
\label{eq: parabola}
(\alpha_{1} - \alpha_{2})^{2} + 2 (\delta_{2}-\delta_{1}) \alpha_{1} + 2 (\delta_{3}-\delta_{2}) \alpha_{2} \, =\,  0\,, 
\end{equation}
then (\ref{eq: Lap0}) yields 
$\, (2\alpha_{1} - \alpha_{2} ) \, \widehat{{\bm \nu}}_{1}  (\alpha_{2}) + (2 \alpha_{2} - \alpha_{1}) \, \widehat{{\bm \nu}}_{2}  (\alpha_{1}) \, =\, 0 \, $. 
Under the conditions of (\ref{ord2}), the segment is non-empty  provided $\,\delta_{2} < \delta_{1} < (\delta_{2} + \delta_{3})/2 < \delta_{3}\,$ or   $\,\delta_{1} < (\delta_{1} + \delta_{2})/2 < \delta_{3} < \delta_{2}\,$. 

On the other hand, under the condition (\ref{ord}), the segment on the parabola  (\ref{eq: parabola}) 
degenerates to the   origin $\, ( \alpha_1, \alpha_2) = (0,0)\,,$ and thus (\ref{eq: Lap}) holds then for every $\,(\alpha_{1}, \alpha_{2})\in  [0, \infty)^{2} \setminus \{ (0,0)\} \,$.   

\smallskip 
\noindent $\,\bullet\,$ Dividing (\ref{eq: Lap0}) by $\,\alpha_{j}>0\,$, then letting $\,\alpha_{j} \uparrow \infty\,$, $\,\,j = 1, 2\,$, we   obtain  
\begin{equation} 
\label{trace}
\lim_{\alpha_{2}  \uparrow \infty} \alpha_{2} \,\widehat{{\bm \pi}}  (\alpha_{1}, \alpha_{2}) \, =\,  \widehat{{\bm \nu}}_{2} (\alpha_{1}) \, , \quad 
\lim_{\alpha_{1}  \uparrow \infty} \alpha_{1} \,\widehat{{\bm \pi}}  (\alpha_{1}, \alpha_{2}) \, =\,  \widehat{{\bm \nu}}_{1} (\alpha_{2}) \, , \qquad (\alpha_{1}, \alpha_{2}) \in [0, \infty)^{2} \, . 
\end{equation}
We deduce that the measures $\, {\bm \nu}_1\,$, $\, {\bm \nu}_2\,$ of (\ref{eq: add L}) are appropriately normalized traces on the two axes, of the invariant probability measure $\, {\bm \pi}\,$.

 \smallskip 
\noindent $\,\bullet\,$
Now, let us take $\,\alpha_{1} \, =\, \alpha_{2} = \alpha > 0 \,$ in (\ref{eq: Lap0}); we see that the \textsc{Laplace} transform of the invariant distribution  for the sum $\, G(\cdot) + H(\cdot)\,$ of the gaps is expressed as  

\begin{equation} 
\label{eq: Lap Sum}
\mathbb E_{\bm \pi} \Big[ e^{-\alpha (G(T) + H(T))} \Big] \, =\, \widehat{{\bm \pi} }(\alpha, \alpha) \, =\, \frac{\, \widehat{{\bm \nu}}_{1} (\alpha) + \widehat{{\bm \nu}}_{2} (\alpha)\,}{ 4(\delta_{3} - \delta_{1}) } \, =\, \frac{\, \widehat{{\bm \nu}}_{1} (\alpha) + \widehat{{\bm \nu}}_{2} (\alpha)\,}{ 2(\lambda_{1} + \lambda_{2}) } \,\, ; \qquad \alpha > 0 \, . 
\end{equation}
Together with Proposition \ref{notriple}, this shows that the measure $\, {\bm \nu}_1+ {\bm \nu}_2\,$ is supported on 
$\, (0, \infty)$.  

Letting $\,\alpha \downarrow 0\,$ in the above equation gives the total mass of the two measures on the axes under the stationary distribution, namely 
\begin{equation} 
\label{eq: bdry mass}
\big(  {\bm \nu}_{1}  + {\bm \nu}_{2} \big) ((0, \infty)) \, =\,\widehat{ {\bm \nu}}_{1} (0) + \widehat{{\bm \nu}}_{2}  (0)  \, =\, 4 (\delta_{3} - \delta_{1}) > 0 \, ; 
\end{equation}
this is consistent with the strong laws of large numbers (\ref{SLLN}), because of the normalization (\ref{eq: add L}) and 
\[
\lim_{T\to \infty} \frac{1}{\, T\, } \big( L^{G}(T) + L^{H}(T) \big) \, =\,  \frac{1}{\, 2\, } \,\big(  {\bm \nu}_{1}  + {\bm \nu}_{2} \big) \big((0, \infty)\big) \, =\, 2 (\delta_{3} - \delta_{1})\,=\, \lambda_1 + \lambda_2  \, . 
\]
In particular, the two measures $\, {\bm \nu}_1\,$, $\, {\bm \nu}_2\,$ of (\ref{eq: add L}) are both finite. 

 \medskip 
\noindent $\,\bullet\,$ Now let us take the limit in (\ref{eq: Lap}) as $\, \alpha_2 \downarrow 0\,$, to obtain
$$
\mathbb E_{{\bm \pi}} \Big[ e^{- \alpha_{1} G(T)}   \Big] \, =\, \widehat{{\bm \pi}} (\alpha_{1}, 0)\,=\, \frac{\,2\, \widehat{{\bm \nu}}_{1} (0) - \widehat{{\bm \nu}}_{2} (\alpha_{1})}{\, 2  \alpha_{1}   + 4 (\delta_{2}-\delta_{1})\,}  \, ;
$$
next we let $\, \alpha_1 \downarrow 0\,$ and get
$\, 2\, \widehat{{\bm \nu}}_{1} (0) - \widehat{{\bm \nu}}_{2} (0) \,=\, 4 (\delta_{2}-\delta_{1})\,$.
In conjunction with (\ref{eq: bdry mass}), this gives the total mass of each of the two measures on the axes, namely 
  \begin{equation}
\label{eq: bdry mass1}
 {\bm \nu}_{1} \big( (0, \infty) \big) \,=\, \widehat{{\bm \nu}}_{1} (0)
\,=\,{ 4 \over \,3\,} \big( \delta_2 + \delta_3 - 2\, \delta_1  \big)
\,=\, 2\, \lambda_1  
\,=\, 2\,\lim_{T \ar} \frac{\,L^G (T)\,}{T}\,,
\end{equation}
 \begin{equation}
\label{eq: bdry mass2}
 {\bm \nu}_{2} \big( (0, \infty) \big)   \,=\, \widehat{{\bm \nu}}_{2} (0)\,=\,{ 4 \over \,3\,} 
 \big( 2\, \delta_3 - \delta_1 - \delta_2 \big)
 \,=\, 2\, \lambda_2 \,=\,2\,
\lim_{T \ar} \frac{\,L^H (T)\,}{T} \,,
\end{equation}

\smallskip
\noindent
in accordance with (\ref{eq: add L}) and (\ref{SLLN}).   This way,  we express the \textsc{Laplace} transform  for the two marginals 
\begin{equation}
\label{marginal1}
\mathbb E_{{\bm \pi}} \Big[ e^{- \alpha_{1} G(T)}   \Big] \, =\, \widehat{{\bm \pi}} (\alpha_{1}, 0)\,=\, \frac{\,4\, \lambda_{1}  - \widehat{{\bm \nu}}_{2} (\alpha_{1})}{\, 2 \, \alpha_{1}   + 4 (\delta_{2}-\delta_{1})\,}  \,,  
\end{equation}   
\begin{equation}
\label{marginal2}
\mathbb E_{{\bm \pi}} \Big[ e^{- \alpha_{2} H(T)}   \Big] \, =\, \widehat{{\bm \pi}} (0, \alpha_{2} )\,=\, \frac{\,4\, \lambda_{2}  - \widehat{{\bm \nu}}_{1} (\alpha_{2})}{\, 2 \, \alpha_{2}   + 4 (\delta_{3}-\delta_{2})\,}  \,,
\end{equation} 
in terms of the  \textsc{Laplace} transforms of the  traces $\, {\bm \nu}_2\,$  and $\, {\bm \nu}_1\,$, respectively. 

 \newpage
 
\begin{rem}[Absolute Continuity]   
It can be shown as in section 8 of \citep{MR912049} 
that the measures  $\,  {\bm \nu}_{1} (\cdot) \,$ and $\,  {\bm \nu}_{2} (\cdot)\,$ are absolutely continuous with respect to \textsc{Lebesgue} measure; in other words, that there exist probability density functions $\,  {\bm \sigma}_{1} (\cdot) \,$ and $\,  {\bm \sigma}_{2} (\cdot)\,$ on $\, (0, \infty)\,$, such that 
\begin{equation} \label{def: density sigma}
\, {\bm \nu}_j (A) = 2\, \lambda_j \int_A {\bm \sigma}_{j} (z) \, \mathrm{d} z\, , \qquad \,A \in {\cal B} \big([0, \infty)\big)\,, \quad j=1,2\,.
\end{equation} 
 It follows now from (\ref{eq: Lap Sum}) that the   invariant distribution of the sum  of gaps $\, G(\cdot) + H(\cdot)\,$ is also absolutely continuous with respect to \textsc{Lebesgue} measure,  with probability density   $\, \mathbb{P}_{{\bm \pi}} \big( G(T) + H(T) \in \dx z \big) = {\bm \sigma}  (z) \, \dx z\,\,$ given by 
 \begin{equation}
\label{convex}
   {\bm \sigma}  (z) \,=\, \frac{\lambda_1}{\,\lambda_1 + \lambda_2\,} \, {\bm \sigma}_1 (z) + \frac{\lambda_2}{\,\lambda_1 + \lambda_2\,} \, {\bm \sigma}_2  (z) \,, \qquad z \in (0, \infty)\,.
\end{equation}    
\end{rem} 
 
\begin{rem}[No Product Form] 
\label{rem: no product form} 
It is seen from (\ref{eq: add L}), (\ref{eq: bdry mass1}) and the definition of the local time  that 
\[
(2 \lambda_{1})^{-1} \int^{\infty}_{0} e^{-\alpha_{2} h} \,{\bm \nu}_{1}({\mathrm d} h) \, =\, \Big( \mathbb E_{\pi} \big[ L^{G}(T) \big] \Big)^{-1} \cdot \mathbb E_{\pi} \Big[ \int^{T}_{0} e^{-\alpha_{1} G(t) - \alpha_{2} H(t)} {\mathrm d} L^{G}(t) \Big] 
~~~~~~~~~~~~~~~~~~~~~~~~~~~~~~
\]
\[
\, ~~=\,  \Big( \mathbb E_{\pi} \Big[ \lim_{\varepsilon \downarrow 0} \frac{1}{\, 2\varepsilon\, } \int^{T}_{0} {\bf 1}_{\{ G(t) < \varepsilon \}} {\mathrm d} t\Big] \Big)^{-1} \cdot \mathbb E_{\pi} \Big[\lim_{\varepsilon \downarrow 0} \frac{1}{\, 2\varepsilon\, }\int^{T}_{0} e^{-\alpha_{1} G(t) - \alpha_{2} H(t)} {\bf 1}_{\{G(t) < \varepsilon \}} {\mathrm d}t \Big] 
\]
\[
\, =\, \lim_{\varepsilon \downarrow 0} \big( {\bm \pi} (g < \varepsilon) \big)^{-1} \int^{\infty}_{0} \int^{\infty}_{0} e^{-\alpha_{1} g - \alpha_{2} h} \cdot {\bf 1}_{\{g < \varepsilon\}} \,{\bm \pi}({\mathrm d} g, {\mathrm d} h) \, , \quad (\alpha_{1}, \alpha_{2}) \in (0, \infty)^{2}   
\]
\noindent
holds for all $\, T \in (0, \infty)\,$. Hence, by the uniqueness of the \textsc{Laplace} transform and (\ref{def: density sigma}) we obtain 
$$ 
\mathbb{P}_{{\bm \pi}} \big( H(T)  \in \dx z \, \big| \,   G(T)=0\big) = \,{\bm \sigma}_{1} (z) \, \mathrm{d} z
\, ; \quad \text{similarly } \quad \mathbb{P}_{{\bm \pi}} \big( G(T)  \in \dx z \, \big| \,   H(T)=0\big) = \,{\bm \sigma}_{2} (z) \, \mathrm{d} z\,.
$$ 
With this  interpretation  in mind, it becomes clear  that {\it the joint distribution of the two gaps  under the invariant probability measure cannot possibly  be the product of their two marginal distributions. }
\end{rem}

\subsubsection{The General Symmetric Case}
 \label{SymmCase}

Let us consider now the general symmetric case 
\begin{equation}
\label{sym}
\delta_2 - \delta_1\,=\, \delta_3 - \delta_2\,=:\, \lambda / 2 \,>\,0\,;
\end{equation}
the configuration of (\ref{4}) is a special case of this situation, with $\delta_2=0$.

We have now $\,  \lambda_1 = \lambda_2 = \lambda \,$ in (\ref{ord4}),   as well as  
  $\,  {\bm \nu}_{1} (\cdot) \equiv  {\bm \nu}_{2} (\cdot) =: {\bm \nu} (\cdot)\,$, $\,  {\bm \sigma}_{1} (\cdot) \equiv  {\bm \sigma}_{2} (\cdot) \equiv {\bm \sigma} (\cdot)\,$, and (\ref{eq: Lap Sum}), (\ref{eq: Lap}) lead to $\,\widehat{{\bm \nu}}  (\alpha ) \,=\, 2 \,\lambda\, \,\widehat{{\bm \pi}}  (\alpha ,  \alpha)\,$ and to the   {\it functional equation} 
\begin{equation}
\label{sym1}
 \widehat{{\bm \pi}}  (\alpha_{1}, \alpha_{2}) \, =\,\frac{\lambda}{\,  (\alpha_{1} - \alpha_{2})^{2} + \lambda (\alpha_{1} +  \alpha_{2})\,}\, \Big[\,(2\alpha_{1} - \alpha_{2}) \, \widehat{{\bm \pi}}  \big(\alpha_{2}, \alpha_{2}\big) + (2\alpha_{2} - \alpha_{1})  \, \widehat{{\bm \pi}}  \big(\alpha_{1}, \alpha_{1}\big) \, \Big] ~~~~
\end{equation}
 for the \textsc{Laplace} transform of the joint distribution of the gaps. To wit,  in the   symmetric case of (\ref{sym}) and in steady state, {\it the joint distribution of the gaps is determined by the distribution of    their sum} -- or for that matter by the  common  marginal distribution of each of these gaps, as in this case
\begin{equation}
\label{sym2}
\widehat{{\bm \nu}}  (\alpha ) \,=\, 2 \,\lambda\, \,\widehat{{\bm \pi}}  (\alpha ,  \alpha) \,, \qquad
  \widehat{{\bm \pi}}  (\alpha ,  0) \, =\, \widehat{{\bm \pi}}  (0, \alpha ) \, =\,\frac{\lambda}{\,   \alpha  + \lambda  \,}\,\big[\, 2 -  \widehat{{\bm \pi}}  (\alpha ,  \alpha) \,\big]   \,;   \qquad \alpha \ge 0\,.
\end{equation}
\noindent
This last equation suggests that, in the symmetric case of (\ref{sym}), the  marginal invariant distributions of the gaps   have common probability density function
\begin{equation}
\label{sym3}
\frac{\mathbb{P}_{ {\bm \pi}} \big( G(t) \in \dx \xi \big)}{\dx \xi } \,=\, \frac{
\mathbb{P}_{ {\bm \pi}} \big( H(t) \in \dx \xi \big)}{\dx \xi } \,=\, 
{\bm \tau} (\xi)\,  =  \lambda    \left[ \, 2\, e^{\, -   \lambda \xi} \,- \int_0^\xi e^{\, -   \lambda (\xi-z)} \, {\bm \sigma} (z)\, \mathrm{d} z\, \right]    
\end{equation}
 for $\, \xi \in (0, \infty)\,$. 
  In particular,   the invariant distribution  for the sum of the two gaps has finite moment-generating function, thus moments of all orders: 
  \begin{equation}
\label{sat}
  \int_0^\infty e^{\,    \lambda z} \, {\bm \sigma} (z)\, \mathrm{d} z \le 2\,  . 
  \end{equation}

  \begin{rem}[The Average Gaps in Steady-State] 
     \label{Moments}
Always under the condition (\ref{sym}), suppose that the pair of processes $\,( G (\cdot),H(\cdot))\,$ runs under its stationary distribution ${\bm \pi}$. Then by taking expectations in the expression (\ref{NoLocTimes}), one obtains $\,\mathbb E_{ {\bm \pi} } [\,1-(3 \lambda /2) (G (t)+ H(t)) \,]=0\,$; due to symmetry, this shows   
\begin{equation}
\label{FirstMoments}
\mathbb E_{ {\bm \pi} }  \big[G(t)\big]\,=\,\mathbb E_{ {\bm \pi} }\big[H(t)\big] \,=\,\frac{1}{\,3\,\lambda\,}\,.
\end{equation}

The first-moment computation (\ref{FirstMoments}) {\it rules out exponential marginal distributions for the gaps} in this symmetric  case (\ref{sym}). For if   $\, {\bm \tau} (\xi) = \beta \, e^{\, - \beta \, \xi}\,, ~ \xi \in (0, \infty)\,$ were valid for some constant $\, \beta >0\,$, then the equation
$$
{\bm \tau} (\xi) \, e^{\,   \lambda \, \xi}\,=\, \lambda \left[ \, 2 - \int_0^\xi e^{\, \lambda \, z}\, {\bm \sigma} (z)\, \dx z \, \right], \qquad 0 < \xi <   \infty
$$
from (\ref{sym3}) would force $\, \beta = 2\, \lambda\,$ and  $\, {\bm \tau} (\xi) = {\bm \sigma} (\xi)= \beta \, e^{\, - \beta \, \xi}\, $, thus $\, \mathbb E_{ {\bm \pi} }  \big[G(t)\big]=\mathbb E_{ {\bm \pi} }\big[H(t)\big]=1/ \beta = 1 / (2 \lambda) ,$  contradicting (\ref{FirstMoments}). 
\end{rem}

\begin{figure}
\begin{center}
\begin{tabular}{cc}
\includegraphics[scale=0.22]{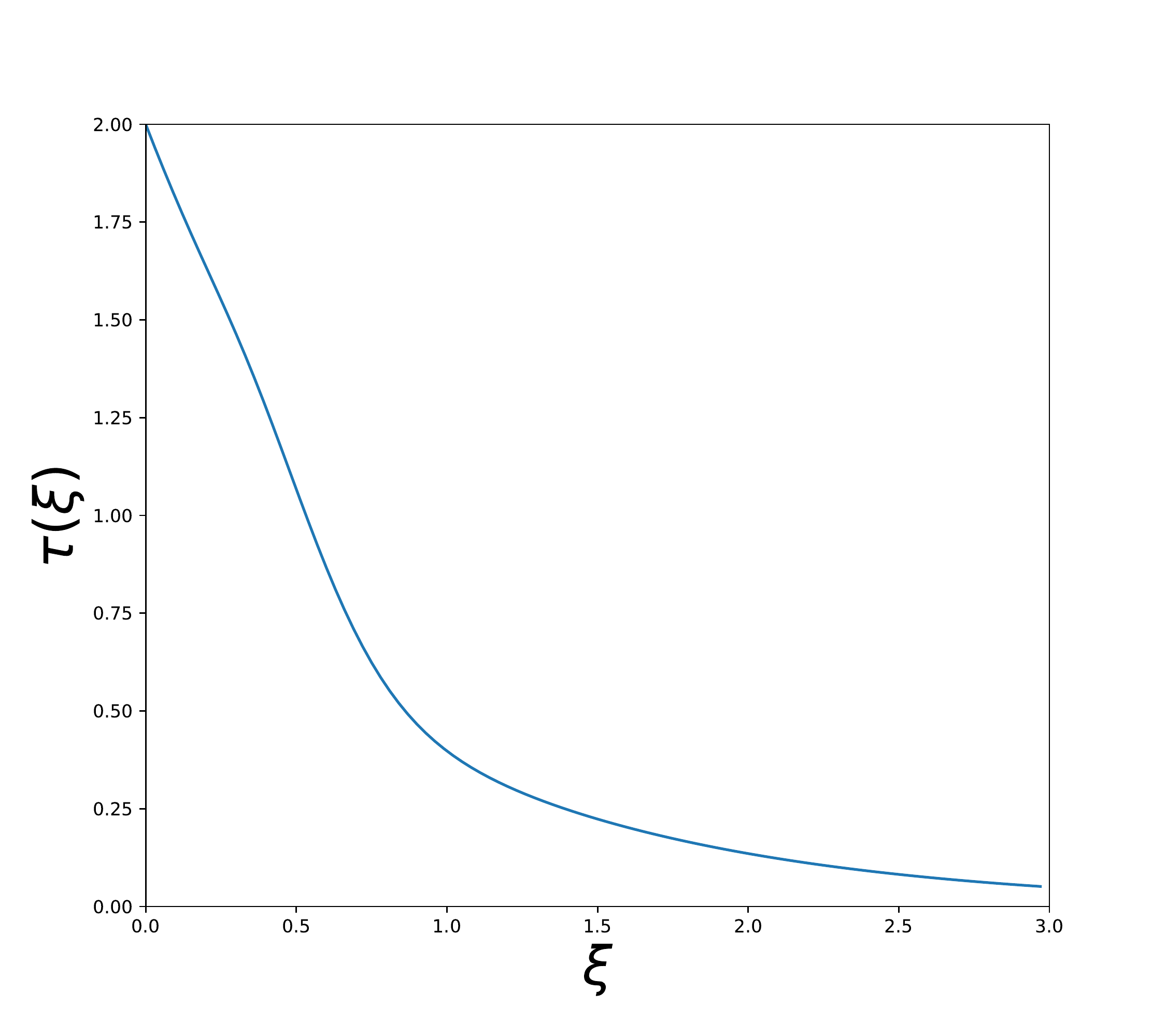}  & 
\includegraphics[scale=0.12]{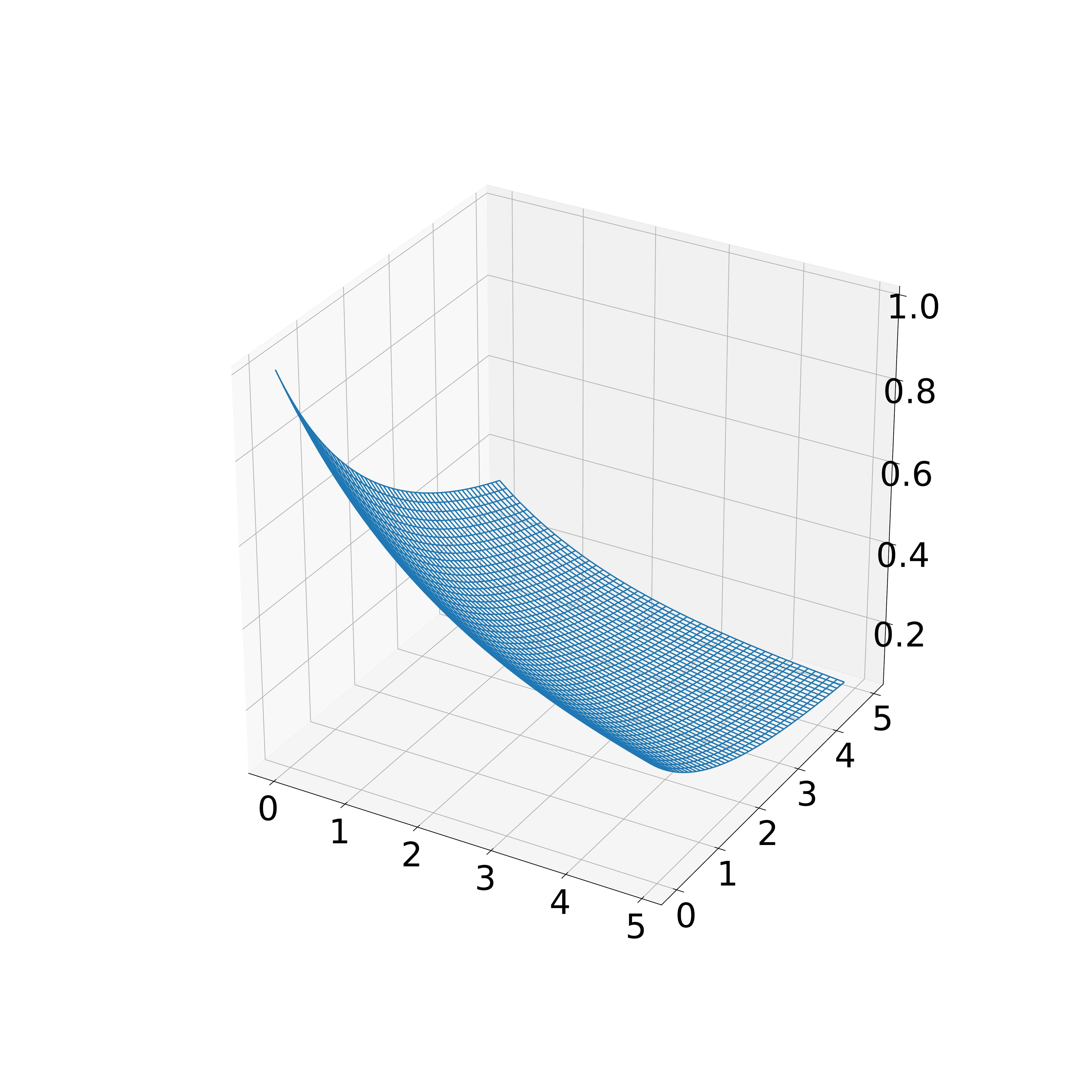} \\
\end{tabular} 
\end{center} 
\caption{\label{fig: taupihat} The marginal probability density function $\, {\bm \tau} (\cdot) \,$ in (\ref{eq: tau-conj}) (left) and the joint Laplace transform $\, \widehat{\bm \pi} (\alpha_{1}, \alpha_{2}) \,$ in (\ref{eq: pihat-conj}) (right) under the conjecture on $\,{\bm \sigma} (\cdot) \,$ in (\ref{eq: sigma-conj}). }
\end{figure}

\begin{rem}[A Conjecture Involving the Gamma Distribution] 
\label{conj}
Always in the symmetric case (\ref{sym}), we conjecture that under the stationary distribution $\,{\bm \pi} \,$, the density function $\,{\bm \sigma} (\cdot) \,$  for the sum of the gaps  $\, G (\cdot) + H (\cdot) \,$ is the Gamma probability density with parameters $\, (\lambda {\mathrm u} , (2/3) {\mathrm u} ) \,$, i.e., 
\begin{equation}
 \label{eq: sigma-conj}
{\bm \sigma} (\xi) \, =\,  \frac{(\lambda \mathrm u)^{2 {\mathrm u}/3} }{ {\bm \Gamma} ( 2 \,{\mathrm u} / 3 ) }\,\, \xi^{\,(2 {\mathrm u}/3) -1} \, e^{\,-\lambda {\mathrm u} \xi}  \, , \qquad 0 < \xi <   \infty \,. 
\end{equation} 
Here  $\, {\bm \Gamma} (z) = \int_0^\infty x^{z-1} e^{-x} {\mathrm d} x \,$ is the Gamma function, and the positive constant $\,{\mathrm u}\,$   the unique solution of the transcendental equation $\, 2 u \log ( u/ (u-1)) \, =\,  3 \log 2\,.$   Equivalently, $\,{\mathrm u}\,$  is given as  
\begin{equation}
{\mathrm u} \, :=\,  \frac{\, 3 \log 2 \, }{\, 3 \log 2 + 2\, {\bm {\mathrm W}} ( - (3 \log 2) / (4 \sqrt{2}) )} \, 
\end{equation} 
in terms of $\, {\bm {\mathrm W}} (\cdot) \,,$   the {\it \textsc{Lambert} W-function}  or ``product logarithm", with the property $\, z = {\bm {\mathrm W}} (z e^z)$.

With this probability density function  $\, {\bm \sigma} (\cdot) \, $ as in (\ref{eq: sigma-conj}), the condition  (\ref{sat}) is satisfied as equality: 
\[
\int^{\infty}_{0} e^{\lambda z} {\bm \sigma} (z) {\mathrm d} z \, =\,  \Big( \frac{{\mathrm u } }{{\mathrm u } - 1}\Big)^{2 {\mathrm u}/3} \, =\,  2 \, . 
\]
It follows from (\ref{sym3}) that the common marginal probability density for the gaps becomes then 
\begin{equation} 
\label{eq: tau-conj}
{\bm \tau} ( \xi ) \, =\,  \lambda e^{-\lambda \xi} \Big[ \Big( \int^{\infty}_{0} - \int^{\xi}_{0} \Big) e^{\lambda z} {\bm \sigma} (z) {\mathrm d} z \Big] \, =\,  \lambda e^{-\lambda \xi} \cdot \int^{\infty}_{\xi} e^{\lambda z} {\bm \sigma} (z ) {\mathrm d} z \, ,\qquad 0 < \xi <   \infty 
\end{equation}
under the invariant distribution; that the \textsc{Laplace} transform   in (\ref{sym1}) of  the joint  distribution of the gaps takes for   $\, (\alpha_{1}, \alpha_{2}) \in (0, \infty)^2\,   $ the form 
\begin{equation} 
\label{eq: pihat-conj}
\widehat{\bm \pi}(\alpha_{1}, \alpha_{2}) \, =\,  \frac{{\lambda} }{\, (\alpha_{1} - \alpha_{2})^{2} + \lambda (\alpha_{1} + \alpha_{2}) \, } \bigg[ \, \big(2\alpha_{1} - \alpha_{2}\big) \Big( \frac{\lambda {\mathrm u}}{ \lambda {\mathrm u}+\alpha_{2}} \Big)^{\frac{2{\mathrm u}}{3}} + ( 2 \alpha_{2} - \alpha_{1}) \Big( \frac{\lambda {\mathrm u}}{ \lambda {\mathrm u}+\alpha_{1}}\Big)^{\frac{2{\mathrm u}}{3}} \, \bigg]  
\end{equation} 
 as in Figure \ref{fig: taupihat}; and that the first-moment condition of (\ref{FirstMoments}) holds, namely,  
\[
\mathbb E_{\bm \pi} [G(t) ] = \mathbb E _{\bm \pi} [ H(t) ] \, =\,  \int_{0}^{\infty} \xi \cdot \lambda e^{-\lambda \xi} \Big[ \int^{\infty}_{\xi} e^{\lambda z}  {\bm \sigma} (z) {\mathrm d} z \Big] {\mathrm d} \xi 
 \, =\,  \frac{1}{\, 3 \, \lambda \, }\, . 
\]
\end{rem}

\section{Ballistic    Middle Motion, Diffusive Hedges   
 } 
 \label{sec5}

We take up in this section the ``obverse" of the three-particle system in (\ref{1})-(\ref{3}), by which we mean replacing the equations in (\ref{1}) by 
\begin{equation}
\label{A1}
X_i (\cdot) \,=\, x_i + \sum_{k=1}^3\, \delta_k \int_0^{\, \cdot} \1_{ \{ X_i (t) = R^X_{k} (t)\} } \, \dx t +  \int_0^{\, \cdot} \left(  \1_{ \{ X_i (t) = R^X_{1} (t)\} } + \1_{ \{ X_i (t) = R^X_{3} (t)\} } \right)   \dx B_i (t)  
\end{equation}
for $\,i=1, 2, 3 ,$  and replacing in the notation of (\ref{ranks}) and (\ref{LT}) the conditions of (\ref{2}), (\ref{3}) by 
\begin{equation}
\label{A2}
     \int_0^{\, \infty} \1_{ \{ R^X_{k} (t) = R^X_{\ell} (t)\} } \, \dx  t \,=\, 0\,, ~~~~~\forall ~~ k < \ell  \,;\qquad ~~~L^{R^X_1 - R^X_3} (\cdot) \, \equiv 0.
   \end{equation}
 The processes $B_1 (\cdot), B_2 (\cdot), B_3 (\cdot),$ are again independent scalar Brownian motions.  
 
  \smallskip
It is now the leading and laggard particles that undergo diffusion, and the  particle in the middle  that ``goes   ballistic". Once again, the dynamics of the system (\ref{A1}) involve dispersion functions that are both discontinuous and degenerate.

In contrast to Proposition  \ref{notriple}, however, we shall see here that {\it ``the two Brownian motions can eventually  squeeze the ballistic motion in the middle",} and thus triple points can occur; in fact, {\it   with probability one in the case $\delta_{1} = \delta_{2} = \delta_{3}\,$}.\, Yet also, that the resulting   triple collisions are ``soft", in that the   local time  $\,L^{R^X_1 - R^X_3} (\cdot) \,$ associated with them is identically equal to zero,  as postulated in  the second requirement of 
(\ref{A2}). The first  requirement there, mandates that all collisions are non-sticky.

\subsection{Analysis} \label{sec: 6.1}

Let us assume that a weak solution  to this system of (\ref{A1}), (\ref{A2}) has been constructed on an appropriate filtered probability space $\, (\Omega, \F, \Prob),$ $\mathbb{F}=  \{ \F (t)  \}_{0 \le t < \infty}\,$. Reasoning as before, we have the analogues 
\begin{equation} \label{eq: RX1-3} 
\begin{split}
\hspace{-0.3cm} R^X_{1} (t)\, &=\, x_1 + \delta_1\, t + W_1(t)+ {1 \over \,2\,} \, \Lambda^{(1,2)}(t) \, , \, \,  R^X_{2} (t)\, =\, x_2 + \delta_2\, t   - {1 \over \,2\,} \, \Lambda^{(1,2)}(t) + {1 \over \,2\,} \, \Lambda^{(2,3)}(t) \, , \\
R^X_{3} (t)\, &=\, x_3 + \delta_3\, t  + W_3(t)-   {1 \over \,2\,} \, \Lambda^{(2,3)}(t) \, ; \quad t \ge 0 
\end{split} 
\end{equation}
 of (\ref{RX1})-(\ref{RX2}) in the notation of (\ref{Lambda}). As in    (\ref{W2}), the processes  
\begin{equation}
\label{W1,3}
 W_k (\cdot)\,:=\, \sum_{i=1}^3 \int_0^{\, \cdot} \1_{ \{ X_i (t) = R^X_{k} (t)\} } \, \dx B_i (t)\,, ~~~~~~\, k=1, 3\,
\end{equation} 
  are independent Brownian motions by the \textsc{P. L\'evy} theorem. It is   fairly clear that the center of gravity of this system evolves as Brownian motion with drift, since   
$ 
\sum_{i=1}^3 X_i (t) = x + \delta \, t + \sqrt{2\,}\, Q(t)\,   
 $
for $\,x= x_1 + x_2 + x_3\,,$   $\,\delta= \delta_1 + \delta_2 + \delta_3\,,$ and $ \,Q(\cdot) = (W_1(\cdot) + W_3(\cdot)) /   \sqrt{2\,}\,$ is standard  Brownian motion.

 \smallskip
Now, the gaps 
$\,G (\cdot):= R^X_{1} (\cdot)- R^X_{2} (\cdot)\,$ and $\, H (\cdot):=R^X_{2} (\cdot)- R^X_{3} (\cdot)\,$ are given as 
$$
G (t) \,=\,U (t) + L^G (t) \,, \qquad H (t) \,=\,V (t) + L^H (t)\,, ~~~~~~~~~~0 \le t < \infty
$$
in the manner of (\ref{G}), (\ref{H}), where again $\, L^G (\cdot) \equiv \Lambda^{(1,2)}(\cdot),$ $\, L^H (\cdot) \equiv \Lambda^{( 2,3)}(\cdot),$ and  
$$
U(t)  := x_1 - x_2 - \big(\delta_2 - \delta_1 \big) \,t  + W_1(t) - {1 \over \,2\,} \, L^{H}(t)\,, \quad V(t) : = x_2 - x_3 - \big(\delta_3 - \delta_2 \big)\, t  - W_3(t) - {1 \over \,2\,} \, L^{G}(t)\,.
$$
The theory of the \textsc{Skorokhod} reflection problem provides   the  system of equations linking   the two local time processes $\, L^G (\cdot) \,$, $\, L^H (\cdot) \,$, an analogue  of the  system (\ref{5}), (\ref{6}):
\begin{equation}
\label{A3}
L^G (t) \,=\, \max_{0 \le s \le t} \big( - U (s) \big)^+\,=\, \max_{0 \le s \le t} \Big(-(x_1 - x_2) + \big(\delta_2 - \delta_1 \big) \,s  - W_1(s) + {1 \over \,2\,} \, L^{H}(s) \Big)^+ 
   \end{equation}
\begin{equation}
\label{A4}
L^H (t) \,=\, \max_{0 \le s \le t} \big( - V (s) \big)^+\,=\, \max_{0 \le s \le t} \Big(-(x_2 - x_3) + \big(\delta_3 - \delta_2 \big) \,s  + W_3(s) + {1 \over \,2\,} \, L^{G}(s) \Big)^+ .
   \end{equation}

\subsection{Synthesis} 
\label{sec: 6.2}

Starting again with   given real numbers $\, \delta_1, \,\delta_2,\,\delta_3\,$  and $ \, x_1 > x_2 > x_3\,$, we  construct   a filtered probability space $\, (\Omega, \widetilde{\F}, \Prob),$ $ \widetilde{\mathbb{F}} = \big\{ \widetilde{\F} (t) \big\}_{0 \le t < \infty}\,$ which supports  three independent, standard  Brownian motions $\, W_k (\cdot)\,$, $\,k =1, 2, 3.$ We  consider the    analogue  
 \begin{equation}
\label{A5}
A (t) \,=\,   \max_{0 \le s \le t} \Big(-(x_1 - x_2) + \big(\delta_2 - \delta_1 \big) \,s  - W_1(s) + {1 \over \,2\,} \, \Gamma (s) \Big)^+\,,\quad 0 \le t <\infty
\end{equation}
 \begin{equation}
\label{A6}
\Gamma (t) \,=\,  \max_{0 \le s \le t} \Big(-(x_2 - x_3) + \big(\delta_3 - \delta_2 \big) \,s  + W_3(s) + {1 \over \,2\,} \,A(s) \Big)^+\,,\quad 0 \le t <\infty
\end{equation}
of the system of equations (\ref{A3}) and (\ref{A4}) for two continuous, nondecreasing and adapted processes $\, A(\cdot)\,$ and $\, \Gamma (\cdot)\,$ with $\, A(0) = \Gamma (0) =0\,$. Once again, the theory of   \citep{MR606992} 
guarantees the existence of a unique continuous    solution $\, \big( A(\cdot), \Gamma (\cdot) \big)\,$ for  the system (\ref{A5}), (\ref{A6}), 
adapted to the filtration $\,\mathbb{F}^{\,(W_1, W_3)}\,$ generated by the 2-D Brownian motion $(W_1 (\cdot), W_3 (\cdot))$:  
 \begin{equation}
\label{A14}
\mathfrak{F}^{\,(A, \Gamma)} (t) \,\subseteq \, \mathfrak{F}^{\,(W_1, W_3)} (t)\,, \qquad 0 \le t < \infty\,.   
\end{equation} 
With the processes $\,   A(\cdot), \,\Gamma (\cdot)  \,$ thus in place, we consider the continuous semimartingales
\[
U(t) := x_1 - x_2 - \big(\delta_2 - \delta_1 \big) \,t  + W_1(t) - {1 \over \,2\,} \, \Gamma(t)\,, \quad V(t) := x_2 - x_3 - \big(\delta_3 - \delta_2 \big)\, t  - W_3(t) - {1 \over \,2\,} \, A(t) ~~
\]
 and then ``fold" them to obtain their \textsc{Skorokhod} reflections 
 \begin{equation}
\label{A7}
G ( t) \,:=\, U ( t) +\max_{0 \le s \le t} \big( - U (s) \big)^+ \, 
= \, x_1 - x_2 - \big(\delta_2 - \delta_1 \big) \,t  + W_1(t) - {1 \over \,2\,} \, \Gamma(t) + A(t) \, \ge \, 0
\end{equation}
\begin{equation}
\label{A8}
H ( t) \,:=\,   V ( t) +\max_{0 \le s \le t} \big( - V (s) \big)^+ \,
= \, x_2 - x_3 - \big(\delta_3 - \delta_2 \big) \,t  - W_3(t) - {1 \over \,2\,} \, A(t) + \Gamma(t)\, \ge \, 0
\end{equation}
for $\, t \in [0, \infty)\,$. This system of equations (\ref{A7}), (\ref{A8}) can be cast in the \textsc{Harrison-Reiman} form 
$$
\begin{pmatrix}
      G(t)   \\
      H(t)  
\end{pmatrix}
  \,=\, \begin{pmatrix}
      G(0)   \\
      H(0)  
\end{pmatrix} + \mathfrak{Z} (t) + \mathcal{R}\, \mathfrak{L} (t)\,, \qquad 0 \le t < \infty 
$$
of (\ref{eq: mfrakQ}), now with covariance matrix
$$
\mathcal{C}\, := \begin{pmatrix}
      1 &     0 \\
        0 &   1
\end{pmatrix}, \qquad \text{reflection matrix} \qquad  
\mathcal{R}\, =\, \mathcal{I} - \mathcal{Q}\,, \quad \mathcal{Q}\, := \begin{pmatrix}
      0 &     1/2 \\
        1/2 &   0
\end{pmatrix} \, , \, \, \text{ and } 
$$
$$
\mathfrak{L} (t) \,= \begin{pmatrix}
      L^G(t)    \\
      L^H (t)  
\end{pmatrix},\qquad 
 \mathfrak{Z} (t) \, = \begin{pmatrix}
       ( \delta_1 - \delta_2) t + W_1(t)      \\
       ( \delta_2 - \delta_3) t - W_3(t)  
\end{pmatrix}\,, \qquad 0 \le t < \infty\,.
$$
We obtain easily the analogues 
\begin{equation}
\label{A9}
\int_0^\infty \1_{ \{ G(t)>0\} } \, \dx A(t) \,=\,0\,, \qquad \int_0^\infty \1_{ \{ H(t)>0\} } \, \dx \Gamma (t) \,=\,0\,, 
\end{equation}
\begin{equation}
\label{A10}
\int_0^\infty \1_{ \{ G(t)=0\} } \, \dx t \,=\,0\,, \qquad \int_0^\infty \1_{ \{ H(t)=0\} } \, \dx t \,=\,0 
\end{equation}
of the properties in (\ref{10}), (\ref{9}) using, respectively, the theories of the \textsc{Skorokhod} reflection problem and of semimartingale local time. We claim that we also have here the analogues
 \begin{equation}
\label{A11}
\int_0^\infty \1_{ \{ H(t)=0\} } \, \dx A (t) \,=\,0\,, \qquad
\int_0^\infty \1_{ \{ G(t)=0\} } \, \dx \Gamma(t) \,=\,0\ 
\end{equation}
of the properties in (\ref{11}), though now for a different reason. 

Let us elaborate: The system of (\ref{A7}), (\ref{A8}) characterizes a non-degenerate, two-dimensional Brownian motion $\, ( G(\cdot),$ $ H(\cdot))\,$ with drift $\, ( \delta_1 - \delta_2, \, \delta_2 - \delta_3)\,,$   reflected off the faces of the nonnegative quadrant. But now, in contrast to the situation prevalent in Section \ref{sec1}, it becomes perfectly possible for this planar motion to hit the corner of the nonnegative orthant with positive probability. In fact, according to   Theorem 2.2 of \citep{MR792398} 
\footnote{~    This theory is not directly applicable to the setting of Section \ref{sec1}, or to that of the Appendix (Section \ref{skew-elastic}), because there the driving Brownian motions are one-dimensional.  
}   (see also \citep{MR3325099} 
\citep{MR2680554}), 
 {\it   when $\delta_1 =\delta_2 =\delta_3 $   this process will hit   the corner of the quadrant  with probability one:}  
$\,\mathbb P ( G(t) = H(t)= 0, ~ \text{ for some } t > 0 ) =1$.  

 Yet, we have always 
 \begin{equation}
\label{A12}
\int_0^\infty \1_{ \{ G(t)=0\} } \, \dx \Gamma(t) \,=\, \int_0^\infty \1_{ \{ G(t)=H(t)=0\} } \, \dx \Gamma(t) \,=\,0\,,
\end{equation}
 \begin{equation}
\label{A13}
\int_0^\infty \1_{ \{ H(t)=0\} } \, \dx A (t) \,=\, \int_0^\infty \1_{ \{G(t)= H(t)=0\} } \, \dx A (t) \,=\,0\,, 
\end{equation}
again with probability one.  Here the first two equalities come from those in (\ref{A9}),  and the second two equalities from Theorem 1 in \citep{MR921820}.  
 The claims in  (\ref{A11}) are thus established.

Armed with the properties (\ref{A9})-(\ref{A11}), we obtain here again the identifications $\, L^G (\cdot) \equiv  A(\cdot)\,$, $\, L^H (\cdot)   \equiv  \Gamma(\cdot)\,$ of the processes $\,A(\cdot), \, \Gamma (\cdot)\,$ in (\ref{A5}), (\ref{A6}) as local times. Details are omitted,   as they are very similar to what was done before. 

\smallskip
\noindent
$\bullet~$ {\it Construction of the Ranked Motions:} We introduce now, by analogy with (\ref{R1})-(\ref{R2}), the   processes 
 \begin{equation} 
 \label{eq: RX1-3b} 
 \begin{split}
R_{1} (t)\,:=\, & x_1 + \delta_1\, t + W_1 (t) + {1 \over \,2\,} \, A(t) \, , \, \, \quad 
R_{2} (t)\,:=\,  x_2 + \delta_2\, t  - {1 \over \,2\,} \, A(t) + {1 \over \,2\,} \, \Gamma (t) \, , 
\\
R_{3} (t)\,:=\, & x_3 + \delta_3\, t  + W_3 (t)-   {1 \over \,2\,} \, \Gamma (t) 
\end{split}
\end{equation}
 \smallskip
\noindent
for $\, 0 \le t < \infty\,$,  and note again the relations $\, R_1 (\cdot) - R_2 (\cdot) = G (\cdot)\ge 0\,$, $\, R_2 (\cdot) - R_3 (\cdot) = H (\cdot) \ge 0\,$ and the comparisons  $\, R_1 (\cdot)   \ge     R_2 (\cdot)  \ge   R_3 (\cdot)   \,$. The range
$$
R_1 (t) - R_3 (t) = G(t) + H(t)= x_1 - x_3 + \big( \delta_1 - \delta_3 \big)\, t + W_1 (t) - W_3 (t) + \frac{1}{\,2\,} \Big( A(t) + \Gamma (t) \Big)\,, \quad 0 \le t < \infty
$$
 is a nonnegative semimartingale with $\, \langle R_1   - R_3 \rangle (t)=2\,t\,$ and local time at the origin 
\begin{equation} 
\label{eq: no LTat0} 
L^{R_1 - R_3} (\cdot)\,=\, \int_0^{\, \cdot} \, \1_{ \{ G(t) + H(t) =0 \} }\, \Big[\, \big( \delta_1 - \delta_3 \big)\,\mathrm{d} t  + \frac{1}{\,2\,} \Big( \mathrm{d} A(t) + \mathrm{d} \Gamma (t) \Big)\,\Big]\,=\,0
\end{equation}
by virtue of (\ref{LT}) and  (\ref{A10}), (\ref{A11}). This is in accordance with the second property posited in (\ref{A2}).

Whereas, we argued already that, at least when $\,\delta_{1} \, =\,  \delta_{2} \, =\, \delta_{3}\,$, the first time of a triple collision  is a.e. finite: i.e., $\Prob ({\cal S} < \infty) =1$ for 
 \begin{equation}
\label{A17}
{\cal S} \,:=\, \inf \big\{ t \ge 0 : R_1 (t) = R_3 (t) \big\} \, =\, \inf\{t \ge 0 : G(t) \, =\,  H(t) \, =\,  0 \}.
\end{equation}

\begin{rem}
[On the Structure of Filtrations] 
 It follows from (\ref{eq: RX1-3b}), (\ref{A14})  that the so-constructed triple  $(R_1 (\cdot), R_2 (\cdot), R_3 (\cdot))$ is adapted to the filtration $\,\mathbb{F}^{\,(W_1, W_3)}\,$ of the planar Brownian motion $(W_1 (
\cdot),  W_3 (\cdot)) $:
 \begin{equation}
\label{A15}
\mathfrak{F}^{\,(R_1, R_2, R_3)} (t) \,\subseteq \, \mathfrak{F}^{\,(W_1, W_3)} (t)\,, \qquad 0 \le t < \infty\,.  
\end{equation}
  On the other hand, the identifications
 $$
 A(\cdot) =  L^G (\cdot) =    L^{R_1 - R_2} (\cdot)\,, \qquad  \Gamma (\cdot) =  L^H (\cdot) =    L^{R_2 - R_3} (\cdot)
 $$
show that     $(A (\cdot),$ $ \Gamma (\cdot))$ is adapted to the filtration $\,\mathbb{F}^{\,(R_1, R_2,  R_3)}\,$ generated by the triple  $(R_1 (
\cdot), R_2 (\cdot),$ $ R_3 (\cdot));$   on account  of (\ref{eq: RX1-3b}), it follows that the same is true of the   2-D Brownian motion $(W_1 (
\cdot),  W_3 (\cdot)).$ In other words, the reverse inclusion of (\ref{A15}) is also valid, and we conclude that the triple  $(R_1 (
\cdot), R_2 (\cdot),$ $ R_3 (\cdot))$ and the pair $(W_1 (\cdot),  W_3 (\cdot))$ generate exactly the same filtration: 
\begin{equation}
\label{A16}
\mathfrak{F}^{\,(R_1, R_2, R_3)} (t) \,= \, \mathfrak{F}^{\,(W_1, W_3)} (t)\,, \qquad 0 \le t < \infty\,.  
\end{equation}
\end{rem}

\noindent
$\bullet~$ {\it Construction of the Individual Motions Up Until a Triple Collision:} The   methodologies   deployed   in  \S \ref{3.3}, show here as well how to construct a {\it strong} solution to the system (\ref{A1}) subject to the requirements of (\ref{A2}), up until the first time ${\cal S}$ of (\ref{A17}) when  a triple collision occurs. The difference here, of course, is that this can happen now in finite time, with positive probability; in fact,   with probability one, i.e.,   $\Prob ({\cal S} < \infty) =1,$ when $\,\delta_{1} \, =\, \delta_{2} \, =\,  \delta_{3} \,$ as we have seen. 

Thus, we need to find another way to construct a solution {\it beyond} this time, that is, on the event $\,\{S < \infty\}$. For concreteness, and in order to simplify terminology and notation, we shall assume for the remainder of the present subsection that this event has full $\, \mathbb P-$measure.

 \smallskip
\noindent
$\bullet~$ {\it Construction of the Individual Motions After a Triple Collision:} 
 In order to construct the processes that satisfy (\ref{A1})    after the first triple collision time  $ \,\mathcal S ,$  we consider the excursions of the rank-gap process $\,(G(\cdot), H(\cdot))\,$   and unfold them, by permuting randomly the names of the individual components.  

\smallskip
More precisely, for the  semimartingales $\,G(\cdot) \,$ and $\, H(\cdot) \,$ let us define the first passage time $$\, \sigma_{0} \, :=\, \inf \big\{ t \ge 0 : G(t) \wedge H(t) \, =\,  0 \big\}  \, ,$$  the zero sets $$\,\mathfrak{Z}^G:=  \{t \ge 0: G(t) \, =\, 0 \}\, , \qquad \mathfrak{Z}^H:= \{t \ge 0: H(t) \, =\, 0 \}\,,$$ and the corresponding countably-many excursion intervals $\, \{\mathcal C_{\ell}^{G}, \ell \in \mathbb N\}\,$, $\, \{ \mathcal C^{H}_{m}, m \in \mathbb N\}  $ away from the origin in a measurable manner, i.e., 
\[
\mathbb R_{+} \setminus \mathfrak{Z}^G  \,\, =\,  \bigcup_{\ell \in \mathbb N} \mathcal C_{\ell}^{G} \, , \qquad 
\mathbb R_{+} \setminus \mathfrak{Z}^H \,\, =\,  \bigcup_{m \in \mathbb N} \mathcal C_{m}^{H} \, . 
\]
We need to permute the indices in a proper and consistent way, so we define the   permutation matrices
\begin{equation} \label{eq: permutations}
\mathfrak{P}_{1,2} \, :=\, \left ( \begin{array}{ccc} 0 & 1 & 0 \\ 1 & 0 & 0 \\ 0 & 0 & 1 \\ \end{array} \right )   , \qquad 
\mathfrak{P}_{2,3} \, :=\, \left ( \begin{array}{ccc} 1 & 0 & 0 \\ 0 & 0 & 1 \\ 0 & 1 & 0 \\ \end{array} \right )   . 
\end{equation}
Here $\,\mathfrak{P}_{1,2}\,$   permutes the first   and   second elements, and $\,\mathfrak{P}_{2,3}\,$  permutes the second   and   third elements.    
 
 \smallskip
We   enlarge   the probability space with I.I.D. random (permutation) matrices $\, \{ \Xi^{G}_{\ell, m},\, \ell \in \mathbb N, \,m \in \mathbb N\} \,$  and $\, \{ \Xi^{H}_{\ell, m}, \,\ell \in \mathbb N, \, m \in \mathbb N\}  ,$  independent of each other and of the filtration $\, \mathbb F^{R^{}}(\cdot)\,$ generated by the rank process $\, (R^{}_{1}(\cdot), R^{}_{2}(\cdot), R^{}_{3}(\cdot))'\,$. Here,  for each $\, (\ell, m ),$ the random matrix $\,\Xi^{G}_{\ell, m}\,$ takes each of the values in $\,\{{\cal I}  ,\,\mathfrak{P}_{1, 2}\}\,$ 
with 
probability $\,1  /  2\,$; whereas   $\, \Xi_{\ell, m}^{H}\,$ takes each of the values in $\,\{ {\cal I},\,\mathfrak{P}_{  2, 3}\}\,$ 
with 
probability $\,1  /   2\,$. With these ingredients we introduce the simple, matrix-valued  process 
\begin{equation} \label{eq: eta}
\begin{split}
{\bm \eta}(\cdot) \, :=\,  \sum_{\ell \in \N} \sum_{m \in \N}  {\bf 1}_{\mathcal C^{G}_{\ell} \cap \mathcal C^{H}_{m} \cap [\sigma_{0}, \infty)}(\cdot) & \Big( \big(  \Xi^{G}_{\ell, m} - {\cal I} \big) \, {\bf 1}_{\{\inf \mathcal C_{\ell}^{G} > \inf \mathcal C_{m}^{H}\}}  
 \\
& \,\,\, \, \, {} 
+ \big(\Xi^{H}_{\ell, m} - {\cal I} \big) \, {\bf 1}_{\{\inf \mathcal C_{\ell}^{G} < \inf \mathcal C_{m}^{H}\}}  \Big) 
\end{split}
\end{equation}
and then define the matrix-valued process $\, Z(\cdot) \,$ as the solution to the stochastic integral equation  
\begin{equation} 
\label{eq: Z}
Z(\cdot) \, =\, \mathcal{I} + \int^{\cdot}_{0} Z(t)\, {\mathrm d} {\bm \eta} (t) \,.  ~~   
\end{equation}
To  construct this solution, we proceed via an approximating scheme as in (\ref{eq: tauZe})-(\ref{eq: Zeps}) below. 

\smallskip
The definition of the process $\,{\bm \eta}(\cdot) \,$ in (\ref{eq: eta}), after the time $\, \sigma_{0}\,$, is understood as follows:

\smallskip
\noindent
{\bf (i)} On the interval $\,\mathcal C_{\ell}^{G} \cap \mathcal C^{H}_{m}\,$ of the excursion which starts from a point in $\, \mathfrak{Z}^G\,$ (i.e., $\, \inf \mathcal C_{\ell}^{G} > \inf \mathcal C_{m}^{H}\,$), the simple process $\, {\bm \eta} (\cdot) \,$ assigns   to this excursion the non-zero matrix 
\begin{equation} \label{eq: P12I}
\mathfrak{P}_{1, 2} - \mathcal{I} \, =\, \left ( \begin{array}{ccc} -1 & 1 & 0 \\ 1 & -1 & 0 \\ 0  & 0 & 0 \\ \end{array} \right) 
\end{equation}
with probability $\,1/2\,,$ or the $\,{\bm 0}\,$ matrix with probability $\,1/2\,$.  
 
\smallskip
\noindent
  {\bf (ii)} On the interval $\,\mathcal C_{\ell}^{G} \cap \mathcal C^{H}_{m}\,$ of the excursion which starts from a point in $\,\mathfrak{Z}^H\,$ (i.e., $\, \inf \mathcal C_{\ell}^{G} < \inf \mathcal C_{m}^{H}\,$), the simple process $\, {\bm \eta}(\cdot) \,$ assigns  to this excursion  the non-zero matrix 
\begin{equation} 
\label{eq: P23I}
\mathfrak{P}_{2, 3} - \mathcal{I} \, =\, \left ( \begin{array}{ccc} 0 & 0 & 0 \\ 0 & -1 & 1 \\ 0  & 1 & -1 \\ \end{array} \right) 
\end{equation}
with probability $\,1/2\,,$ or the $\,{\bm 0}\,$ matrix with probability $\,1/2\,$. 

\smallskip
\noindent
{\bf (iii)} When the excursion starts from the corner $\, \{t \ge 0: G(t) = H(t) = 0\}\,$ (that is, $\, \inf \mathcal C_{\ell}^{G} \, =\,  \inf \mathcal C_{m}^{H}\,$ for some $\,\ell\,$ and $\,m$),  then the process $\, {\bm \eta}  (\cdot) \,$ assigns the $\,{\bm 0}\,$ matrix  to this excursion.

\medskip
The value $\,Z( t) \,$ of the matrix-valued process  defined in (\ref{eq: Z}) represents the product of (countably many, random) permutations listed in (\ref{eq: permutations}), until time $\,t \ge 0\,$. Since   products of permutations are also   permutations, the process $\,Z(\cdot)\,$ takes values in the collection of permutation matrices. 

\smallskip
Finally, with the rank process $\,R(\cdot) =(R^{}_{1}(\cdot), R^{}_{2}(\cdot), R^{}_{3}(\cdot))'\,$ constructed as in (\ref{eq: RX1-3b}), we define the vector  process 
\begin{equation} 
\label{eq: XZR}
X(\cdot) \, \equiv \, \big( X_{1}(\cdot), X_{2}(\cdot), X_{3}(\cdot) \big)^{\prime}\, :=\,  Z(\cdot) R^{}(\cdot) \, . 
\end{equation}
$\bullet \,$ We introduce at this point the enlarged filtration $\,\mathbb{F}  :=  \{  \mathfrak F (t), \,t \ge 0 \}  \,$ via $\, \mathfrak F(t)   :=  \widetilde{\mathfrak F  }(t) \vee    \mathfrak{F}^{Z}(t) \,$. Since the sequences of  I.I.D. random   matrices $\, \{ \Xi^{G}_{\ell, m};\, \ell \in \mathbb N, \,m \in \mathbb N\} \,$  and $\, \{ \Xi^{H}_{\ell, m};\, \ell \in \mathbb N, \,m \in \mathbb N\} \,$ are independent of $\, \mathbb F^{R}\,$,  it can be shown as in \citep{MR2494646} 
that both  triples $\, (W_{1}(\cdot), W_{2}(\cdot), W_{3}(\cdot)) \,$ and $\, (R_{1}(\cdot), R_{2}(\cdot), R_{3}(\cdot)) \,$ are semimartingales  of this enlarged filtration $\, \mathbb F\,$. 

\smallskip
We can   state now  and prove the following  result.

\begin{thm} 
\label{basic_result}
On the filtered probability space $\, (\Omega, \mathcal F,   \mathbb P ), \mathbb{F}   =   \{  \mathfrak F (t)  \}_{t \ge 0} \,$ just constructed, and with the process $X(\cdot)$ as in (3.27), there exists a three-dimensional Brownian motion $B(\cdot) =  ( B_{1}(\cdot), B_{2}(\cdot), B_{3}(\cdot) )^{\prime}$  such that $\, (\Omega, \mathcal F,   \mathbb P ), \,\mathbb{F}   =   \{  \mathfrak F (t)  \}_{t \ge 0} \, , \, ( X(\cdot), B(\cdot)  )\,$ is  a weak solution for the   system (\ref{A1}), 
(\ref{A2}). 
 
  This solution is unique in the sense of the probability distribution;  thus, $X(\cdot)$ has the strong \textsc{Markov} property. 
  It is also pathwise unique and strong, up until the first time $\,\mathcal S\,$ a triple collision occurs; 
however, both pathwise uniqueness and strength  fail  after time $\,\mathcal S\,.$  
\end{thm}

\begin{proof}  We split the argument in three distinct parts. 

\smallskip
\noindent
{\it (i) Existence:}  We   show that, on a suitable filtered probability space with   independent Brownian motions $B_1 (\cdot), B_2 (\cdot), B_3 (\cdot)$, the process $\,X(\cdot) \,$ defined by  (\ref{eq: XZR}), with $ Z(\cdot)  $ in (\ref{eq: Z}) and $  {\bm \eta}  (\cdot) $ in (\ref{eq: eta}),  satisfies the dynamics (\ref{A1}) and the requirement (\ref{A2}).   
The proof is based on the technique of unfolding   semimartingales,   in the manner of \citep{MR3795064} 
for \textsc{Walsh} semimartingales.  

We start by defining recursively the sequence $\,\{ \tau^{\varepsilon}_{\ell}, \ell \in \mathbb N_{0}\}\,$ of stopping times as $\, \tau_{0}^{\varepsilon} \, :=\,  0\,$, 
\begin{equation} \label{eq: tauZe}
\begin{split}
\tau^{\varepsilon}_{2\ell + 1} \, :=\, &  \inf \{ t > \tau^{\varepsilon}_{2\ell} : G(t) \wedge H(t) \, \ge   \varepsilon \} \, , \\ 
\tau^{\varepsilon}_{2\ell + 2} \, :=\,  & \inf \{ t > \tau^{\varepsilon}_{2\ell + 1} : G(t) \wedge H(t) \, =\, 0 \} \, , 
\end{split}
\end{equation}
along with the approximating processes $\,X^{\varepsilon}(\cdot) \, :=\, Z^{\varepsilon}(\cdot) R^{}(\cdot) \,,$ where  
\begin{equation} 
\label{eq: Zeps}
Z^{\varepsilon}(\cdot) \, =\,  \mathcal{I} +\int^{\cdot}_{0} Z^{\varepsilon} (t) {\mathrm d} {\bm \eta}^{\varepsilon}(t)\, , ~~   ~\qquad {\bm \eta}^{\varepsilon}(\cdot) \, :=\,  \sum_{\ell \in \mathbb N} {\bm \eta}(\cdot) \, {\bf 1}_{[\tau^{\varepsilon}_{2\ell+1}, \tau^{\varepsilon}_{2\ell+2})}(\cdot) \, 
\end{equation}
for every $\, \varepsilon \in ( 0,1)$. For these approximating processes,    the product rule gives
\begin{equation} \label{eq: XvareProd}
 X^{\varepsilon}(\cdot) \, =\, \int^{\cdot}_{0}{\mathrm d} \big( Z^{\varepsilon}(t) R^{}(t) \big)  \, =\, \int^{\cdot}_{0}Z^{\varepsilon}(t) \, {\mathrm d} R^{}(t) + \int^{\cdot}_{0} {\mathrm d} Z^{\varepsilon}(t) \,    R^{}(t) \,.
\end{equation}
Now, as $\,\varepsilon \downarrow 0\,$, the process $\,X^{\varepsilon}(\cdot)\,$ converges to $\,X(\cdot) \, =\, Z(\cdot) R^{}(\cdot)\,$ in (\ref{eq: XZR}), and the first term on the right-hand side converges  in probability to the stochastic integral $\,\int^{\cdot}_{0} Z(t)\, {\mathrm d} R^{}(t)\,$. 

 \medskip
Let us analyze the semimartingale dynamics of this last integral. Since  $\,Z(\cdot)\,$ is a permutation-matrix-valued process,     the absolutely continuous  finite-variation (``drift")  components of $\,\int^{\cdot}_{0} Z(t)\, {\mathrm d} R^{}(t)\,$ are 
\[
\int^{\cdot}_{0} Z(t) \left( \begin{array}{c} \delta_{1} \\ \delta_{2} \\ \delta_{3} \\ \end{array} \right )  {\mathrm d} t \, =\, \int^{\cdot}_{0}  \sum_{k=1}^{3}  \left( \begin{array}{c} \delta_{k} {\bf 1}_{\{X_{1}(t) \, =\, R^{}_{k}(t) \}} \\ \delta_{k} {\bf 1}_{\{X_{2}(t) \, =\, R^{}_{k}(t) \}} \\ \delta_{k} {\bf 1}_{\{X_{3}(t) \, =\, R^{}_{k}(t) \}} \end{array} \right)  {\mathrm d} t \,  .
\]
  Similarly, the martingale (``noise") components of $\,\int^{\cdot}_{0} Z(t)\, {\mathrm d} R^{}(t)\,$ are given by 
\begin{equation} 
\begin{split}
\int^{\cdot}_{0}Z(t) \left( \begin{array}{c} {\mathrm d} W_{1}(t) \\ 0 \\  {\mathrm d} W_{3}(t) \\ \end{array} \right) \, &=\,  \int^{\cdot}_{0} \left( \begin{array}{c} {\bf 1}_{\{X_{1}(t) \, =\,  R^{}_{1}(t) \}} {\mathrm d} W_{1}(t) + {\bf 1}_{\{X_{1}(t) \, =\,  R^{}_{3}(t) \}}  {\mathrm d}W_{3}(t) \\
 {\bf 1}_{\{X_{2}(t) \, =\,  R^{}_{1}(t) \}} {\mathrm d} W_{1}(t) + {\bf 1}_{\{X_{2}(t) \, =\,  R^{}_{3}(t) \}} {\mathrm d} W_{3}(t) \\
 {\bf 1}_{\{X_{3}(t) \, =\,  R^{}_{1}(t) \}} {\mathrm d} W_{1}(t) + {\bf 1}_{\{X_{3}(t) \, =\,  R^{}_{3}(t) \}} {\mathrm d} W_{3}(t) \\ \end{array}  \right)
\\
\, &=\, \int^{\cdot}_{0} \left ( \begin{array}{c}  (\1_{ \{ X_1 (t) = R_{1} (t)\} } + \1_{ \{ X_1 (t) = R_{3} (t)\} } )   \,{\mathrm d}  B_1 (t)   \\  (\1_{ \{ X_2 (t) = R_{1} (t)\} } + \1_{ \{ X_2 (t) = R_{3} (t)\} } ) \,  {\mathrm d}  B_2 (t) \\
( \1_{ \{ X_3 (t) = R_{1} (t)\} } + \1_{ \{ X_3 (t) = R_{3} (t)\} } ) \,  {\mathrm d}  B_3 (t) \\ \end{array} \right);
\end{split}
\end{equation} 
 here, on account of the P. \textsc{L\'evy} theorem,   the processes 
\begin{equation} 
\label{the_bees}
B_{i}(\cdot) \, :=\,  \sum_{k=1}^{3} \int^{\cdot}_{0}{\bf 1}_{\{X_{i}(t) \, =\, R^{}_{k}(t)\}} {\mathrm d} W_{k}(t) \,, \qquad i \, =\,  1, 2, 3
\end{equation} 
are independent $\, \mathbb F-$Brownian motions     (recall   that $\, W_{1}(\cdot),  W_{2}(\cdot), W_{3}(\cdot)\,$ are independent $\, \mathbb F-$Brownian motions). Finally, the local time components contributed by the term $\,\int^{\cdot}_{0} Z(t) {\mathrm d} R (t) \,$ are  
\begin{equation} 
\label{eq: LTZGH} 
\begin{split}
& \int^{\cdot}_{0} Z(t) \left( \begin{array}{c} 
(1/2) \, {\mathrm d} L^{G}(t) \\ 
- (1/2) \, {\mathrm d} L^{G}(t) + (1/2)\,  {\mathrm d} L^{H}(t)  \\
- (1/2)\,  {\mathrm d} L^{H}(t) \\ \end{array} \right) 
\, =\,
 \frac{1}{2} \int^{\cdot}_{0}  Z(t)  \left(  \left( \begin{array}{c} 1 \\ -1 \\ 0 \\ \end{array} \right ) {\mathrm d} L^{G}(t) +  \left( \begin{array}{c} 0 \\ 1 \\ -1 \\ \end{array} \right) {\mathrm d} L^{H}(t) \right)  . 
\end{split} 
\end{equation}

On the other hand, in the limit as $\,\varepsilon \downarrow 0\,$ of the term $\, \int^{\cdot}_{0} {\mathrm d} Z^{\varepsilon} (t) R (t) \,$ in \eqref{eq: XvareProd},   local time components appear and cancel  those in (\ref{eq: LTZGH}). More precisely, by (\ref{eq: Z}), we have 
\begin{equation} 
\label{eq: ZeRX}
\int^{T}_{0} {\mathrm d} Z^{\varepsilon}(t)    R^{}(t) =  \int^{T}_{0} Z^{\varepsilon}(t) \big({\mathrm d} {\bm \eta}^{\varepsilon} (t)\big) R^{}(t) \, ,\qquad \int^{T}_{0} {\mathrm d} {\bm \eta}^{\varepsilon}(t) R^{}(t) =  \sum_{\{\ell \,  : \, \tau^{\varepsilon}_{2\ell+1} \le T\}} {\bm \eta}^{\varepsilon}( \tau^{\varepsilon}_{2\ell+1})  R^{}(\tau^{\varepsilon}_{2\ell+1})   .
\end{equation}
The random vector $\, {\bm \eta}^{\varepsilon}( \tau^{\varepsilon}_{2\ell+1}) \, R^{}(\tau^{\varepsilon}_{2\ell+1}) \,$ can take values 
\[
\big(\mathfrak{P}_{1,2} - \mathcal{I} \big)R^{}(\tau_{2\ell+1}^{\varepsilon})\, =\,  \left( \begin{array}{c} - R^{}_{1}(\tau_{2\ell+1}^{\varepsilon} )  +  R^{}_{2}(\tau_{2\ell+1}^{\varepsilon} )  \\ R^{}_{1}(\tau_{2\ell+1}^{\varepsilon} ) - R^{}_{2}(\tau_{2\ell+1}^{\varepsilon} )  \\ 0 \\  \end{array} \right) \, =\,  \varepsilon \left( \begin{array}{c} -1 \\ 1 \\ 0 \\ \end{array} \right ) \,  
\]
or $\,0 ,$ each with equal probability $\,1/2\,$, if it corresponds to the excursion from  $ \mathfrak{Z}^G$ for sufficiently small $\,\varepsilon > 0 \,$; and it can take values 
\[
\big(\mathfrak{P}_{2,3} - \mathcal{I} \big) R^{}(\tau_{2\ell+1}^{\varepsilon} )\, =\, \left( \begin{array}{c} 0  \\  - R^{}_{2}(\tau_{2\ell+1}^{\varepsilon} )  + R^{}_{3}(\tau_{2\ell+1}^{\varepsilon} )  \\ R^{}_{2}(\tau_{2\ell+1}^{\varepsilon} )  - R^{}_{3}(\tau_{2\ell+1}^{\varepsilon} )  \\ \end{array} \right) 
 \, =\, \varepsilon \left( \begin{array}{c} 0 \\ -1 \\ 1 \\ \end{array} \right) 
\]
or $\,0 , $ each with probability $\,1/2\,$, if it corresponds to the excursion from  $ \mathfrak{Z}^H$, for sufficiently small $\,\varepsilon > 0 \,. $ We are deploying here (\ref{eq: eta}), (\ref{eq: P12I})-(\ref{eq: P23I}), and the   continuity of the sample paths of $\, R (\cdot)\,$.

 \smallskip
We recall now the excursion-theoretic characterization of  semimartingale  local time (Theorem VI.1.10 of \citep{MR1725357}
): for a  continuous scalar semimartingale, the rescaled number of its ``downcrossings'' approximates in a   weak-law-of-large-numbers fashion the local time accumulated at the origin  (see   
\citep{MR2494646},   
\citep{MR3332719}, 
and the proof of Theorem 2.1 in \citep{MR3795064}, 
for   crucial applications of this result).  Applying  this approximation to each set of excursions from $\,\mathfrak Z^{G}\,$ and $\,\mathfrak Z^{H}\,$ in the summation of (\ref{eq: ZeRX}), for  $\,G(\cdot)\,$ and $\,H(\cdot)\,,$  respectively,   leads to the  limiting behavior 
\[
\int^{T}_{0} {\mathrm d} {\bm \eta}^{\varepsilon}(t) R^{}(t) \, \,\xrightarrow[\varepsilon \downarrow 0]{} \, \,\int^{T}_{0} \frac{1}{2} \left( \begin{array}{c} -1 \\ 1 \\ 0 \\ \end{array} \right ) {\mathrm d} L^{G}(t) + \int^{T}_{0}\frac{1}{2} \left( \begin{array}{c} 0 \\ -1 \\ 1 \\ \end{array} \right) {\mathrm d} L^{H}(t)  
\]
 in probability. Combining this limit with (\ref{eq: ZeRX}), we obtain the convergence in probability 
\begin{equation} 
\label{eq: ZeRXT}
\int^{T}_{0} {\mathrm d} Z^{\varepsilon}(t) R^{}(t) \, \,\xrightarrow[\varepsilon \downarrow 0]{} \,\, \frac{1}{2} \int^{T}_{0} Z(t) \left[  \left( \begin{array}{c} -1 \\ 1 \\ 0 \\ \end{array} \right ) {\mathrm d} L^{G}(t) +  \left( \begin{array}{c} 0 \\ -1 \\ 1 \\ \end{array} \right) {\mathrm d} L^{H}(t)\right]    
\end{equation}
 and observe that the local time components in (\ref{eq: LTZGH}) are cancelled by the limit (\ref{eq: ZeRXT}) of $\, \int^{T}_{0}{\mathrm d} Z^{\varepsilon} (t) R^{}(t)\,$.

We conclude that the process $\, X(\cdot) \,$  in (\ref{eq: XZR}) satisfies the requirements of   (\ref{A1})-(\ref{A2}), and yields  a weak solution $\, (\Omega, \mathcal F,   \mathbb P ), \,\mathbb{F}   =   \{  \mathfrak F (t)  \}_{t \ge 0}  , \, ( X(\cdot), B(\cdot)  )\,$   for this system as described above.      

\smallskip
\noindent  {\it (ii) Uniqueness in Distribution:}  Suppose that there are two probability measures $\, \mathbb P_{1},  \, \mathbb P_{2}  $  under which $\, X(\cdot) \,$ in (\ref{eq: XZR}) satisfies (\ref{A1})-(\ref{A2}), and $  B(\cdot)  $  is   three-dimensional Brownian motion. For $\,j \, =\,  1, 2\,$ we have $\, \mathbb P_{j}( \mathcal S < \infty) \, =\,  1 \,$.  
By   analogy with the  discussion in \S\,\ref{3.3}, up to the first  triple collision time $\, \mathcal S\,$   in (\ref{exp}) this solution is pathwise unique, thus also {\it strong;} that is, adapted to the filtration $\,\mathbb{F}^{\,(B_1, B_2, B_3)}\,$ generated by the 3-D Brownian motion $(B_1 (\cdot), B_2 (\cdot),  B_3 (\cdot)).$   Hence, its probability distribution is uniquely determined over the interval  $\,[0, \mathcal S);$  in other words, $\, \mathbb P_{1} \equiv \mathbb P_{2}   \,$ on $\, \mathcal F ({\mathcal S-})\,$.

At time $  t =   \mathcal S $  we have $\, X_{1}(\mathcal S) =  X_{2}(\mathcal S)=  X_{3}(\mathcal S) \,,$ $\, \mathbb P_{j}\,$-a.e., for $  j =  1, 2,$  and   ties are resolved in favor of the lowest index.  For $\, t > \mathcal S\,,$ every given ``name" appears in each rank equally likely,   since the system (\ref{A1})-(\ref{A2}) is invariant under permutations; in particular, for every $\, t > 0 \,$, we have 
\begin{equation} \label{eq: equally likely}
\mathbb P_{j} \big( X_{i}(t) \, =\,  R_{k}^{X}(t) \, \vert \, \, t > \mathcal S \big) \, =\,   
1/3\, ; \qquad (i \, , k ) \, \in \, \{ 1, 2, 3 \}\, , \, ~~j \, =\,  1, 2 \, .  
\end{equation}
Here the probability distribution of the rank process $\, R_{k}^{X}(\cdot) \,$, $\, k \, =\, 1, 2, 3\,$ in (\ref{eq: RX1-3}) is uniquely determined through (\ref{eq: RX1-3b}) by the probability distribution of the reflected Brownian motion $\, (G(\cdot), H(\cdot)) \,$ in subsection \ref{sec: 6.2}. Since the probability distribution of $\, X(t) \,$, $\, t > \mathcal S\,$ is determined by the  process $\, R^{X}(\cdot) \,$ of ranks and the name-rank correspondence, it is uniquely determined for $\, t \ge \mathcal S\,$.

Standard arguments based on the \textsc{Markov} property, allow us now to extend these considerations to the finite-dimensional distributions:  $\, \mathbb P_{1}(\,\cdot \cap \{ t \ge \mathcal S\}) \equiv \mathbb P_{2}(\,\cdot \cap \{ t \ge \mathcal S\}) \,$ for every $\, t > 0 \,$. Combining these considerations with the uniqueness in distribution before time $  \mathcal S$, we deduce that the weak solution we constructed is unique  in distribution, that is, $  \mathbb P_{1}   \equiv \mathbb P_{2}  \, $ on $\, \mathcal F (\infty)\,$.

\smallskip 
\noindent {\it (iii) \,Failure of Pathwise Uniqueness, and of Strength:} In the construction of the matrix-valued processes $\, {\bm \eta} (\cdot) \,$ in (\ref{eq: eta}) and $\, Z(\cdot) \,$ in (\ref{eq: Z}), the excursion starting from the corner $\, \{t \ge 0: G(t) = H(t) = 0 \}\,$ {\it does not} appear explicitly, because the triple collision local time $\,L^{G+H}(\cdot) \, =\,  L^{R_{1}^{X} - R_{3}^{X}} (\cdot) \,$ is identically equal to zero  as in (\ref{eq: no LTat0}). The corresponding construction of $  X(\cdot) $ does not change the name-rank correspondence immediately before or after the triple collision. Since the triple collision local time $\, L^{G+H}(\cdot) \,$ does not grow,  one may perturb in the above construction the weak solution,  by randomly permuting the names of particles  immediately after the triple collision time $\, \mathcal S\,$ --- and still obtain the same stochastic dynamics (\ref{A1})-(\ref{A2}),   hence the same probability distribution for 
$  X(\cdot) $. 

Then the resulting sample path of $  X(\cdot)  $ is different from the original sample path, so pathwise uniqueness fails. But here we have uniqueness in distribution, so the solution of (\ref{A1})-(\ref{A2}) cannot be strong after the first triple collision  $\, \mathcal S\,;$ this is because uniqueness in distribution,  coupled with strong existence, implies pathwise uniqueness (the ``dual  \textsc{Yamada-Watanabe} theorem"  of 
\citep{MR1128494}, 
\citep{MR1978664}). 
\end{proof}

To the best of our knowledge, the result of Theorem \ref{basic_result}, that the solution ceases to be strong after  the first triple collision, is the first of its kind for competing particle systems. We conjecture that this feature holds in general such systems with $\,n \ge 3\,$ particles. Figure \ref{Fig: Ballistic} shows simulated paths of particles in \eqref{A1} with $\, \delta_1 = -0.5$, $\,\delta_2 =   0\,$ and $\, \delta_3 = 0.5\,$, based on the construction discussed in Theorem \ref{basic_result}. 


 \begin{figure}  
  \begin{center}
    \includegraphics[scale=0.25
    ]{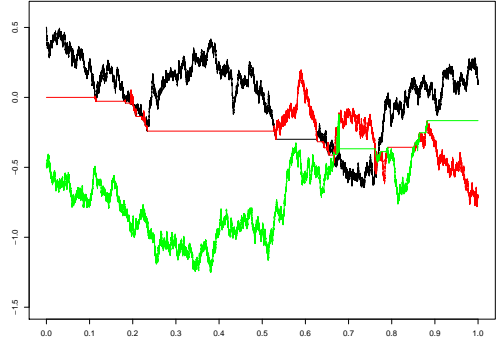}
    \caption{Simulated processes; Black $\,=X_1(\cdot)\,$,  Red $\,=X_2 (\cdot)\,$,  Green $\,=X_3(\cdot)\,$. Here we have taken $\, \delta_1 = -0.5$, $\,\delta_2 =   0\,$ and $\, \delta_3 = 0.5\,$ in (\ref{A1}). }\label{Fig: Ballistic}
  \end{center}
\end{figure}


\begin{rem}
In \eqref{A3}-\eqref{A4},  as well as    in the construction  (\ref{eq: RX1-3b})-(\ref{A16}) of the ranks,    only the Brownian motions $W_1 (\cdot)$ and $W_3 (\cdot)$ appear. By contrast, {\it all three} ``rank-specific" Brownian motions  $W_1 (\cdot)$, $W_2 (\cdot)$ and $W_3 (\cdot)$ are needed   in \eqref{the_bees} for constructing the  ``driving" or ``name-specific"   Brownian motions $B_1 (\cdot)$, $B_2 (\cdot)$, $B_3 (\cdot)$. This is common in situations where the quadratic variation of a driving local martingale can vanish,  and an additional,  independent randomness is needed to ``re-ignite" the motion --- as for instance in the proof of the \textsc{Doob} representation of continuous local martingales with quadratic variation which is absolutely continuous with respect to \textsc{Lebesgue} measure.  
\end{rem}

\begin{rem}
[Some Open Questions] 
The above approach to   (\ref{A1})-(\ref{A2}) is  akin to the construction of the \textsc{Walsh} Brownian motion, and to the splitting stochastic flow of the \textsc{Tanaka} equation. It would be interesting to examine the solvability of (\ref{A1})-(\ref{A2}) via the spectral measures of   classical/non-classical noises, and via the theory of stochastic flows  developed by 
\citep{MR1487755}, 
\citep{MR1816931}, 
\citep{MR1943894} 
and 
\citep{MR2060298}, \citep{MR2052863} 
  (see also \citep{MR2553243} 
and its references). It would also be quite interesting to determine   whether the filtration of the post$-{\cal S}$ process $X(\cdot)$ might fail, in the spirit of \citep{MR1487755}, 
to be generated by {\it any} Brownian motion  of {\it any} dimension. We leave these issues to further research. 
\end{rem}

\subsection{Local Time Considerations: The Case of Equal Drifts} 
\label{}

When $ \delta_{1} = \delta_{2} = \delta_{3} ,$  it is possible to describe the local behavior of the semimartingale reflected Brownian motion $ (G(\cdot), H(\cdot)) $ at the corner of the quadrant  and  in the manner of \citep{MR890921}, 
as follows.

Let us denote by $\,(\varrhob_{\cdot}, \varthetab_{\cdot})\,$ the system (\ref{A7})-(\ref{A8}) in polar coordinate in $\,\mathbb R^{2}\,$, i.e., $\,0 \le G(\cdot) \, =\, \varrhob_{\cdot} \cos (\varthetab_{\cdot})\,$, $\, 0 \le H(\cdot) \, =\, \varrhob_{\cdot} \sin (\varthetab_{\cdot})\,$. In the notation of \citep{MR792398}, 
this system corresponds to planar Brownian motion reflected on the faces of the nonnegative quadrant at angles  $\,\theta_{1} \, \equiv \, \theta_{2} \, :=\, \arctan (1/2)\,$ relative to the interior normals there, thus 
$$
\alpha \,:=\, \frac{\,2(\theta_{1} + \theta_{2}) \,}{ \pi}\, =\, \frac{\,2 \,}{ \pi} \,\arctan \big(4/3 \big) \,\in\,  \left( \frac{\,1 \,}{ 2}\,, \frac{\,2 \,}{ 3} \right).
$$
From the theory of \citep{MR792398}, 
we know that the process $\,(G(\cdot), H(\cdot))\,$, started in the interior of the quadrant, hits eventually the vertex (0,0) with probability one, but does not get absorbed there: it manages to escape from the vertex, though it hits immediately the boundary of the quadrant (cf. \citep{MR890921}, 
section 3). We   define the function 
$$\,
\varphi (\rho, \theta) \, =\, \rho^{\alpha} \cos ( \alpha \theta - \theta_{1})\,, \qquad \,0 \le \rho < \infty \, , ~~~\, 0 \le \theta \le \pi/2 \,,
$$ 
and note $\, (2\, / \sqrt{5}) \le \cos(\alpha \theta - \theta_{1}) \le 1\,$ for $\, 0 \le \theta \le \pi / 2\,$.
Always with  $\,\delta_{1} = \delta_{2} = \delta_{3}\,$,   the process $\,\varphi ( \varrhob_{\cdot}, \varthetab_{\cdot})\,$ is a nonnegative, continuous local submartingale; the   continuous, adapted, non-decreasing process in its \textsc{Doob-Meyer} decomposition is a constant multiple of  
\begin{equation} 
\label{eq: LTatO}
0 \le \Lambda^{\bullet} (\cdot) \, :=\, \frac{\, \alpha (2-\alpha)\, }{2} \lim_{\varepsilon \downarrow 0}\, \varepsilon^{1 - (2/\alpha)} \int^{\cdot}_{0} \Big(\cos \big( \alpha \varthetab_{t} - \theta_{1} \big)\Big)^{(2/\alpha) - 2}\cdot {\bf 1}_{[0, \varepsilon)} \big( \varphi( \varrhob_{t}, \varthetab_{t}) \big) \, {\mathrm d} t \,,
\end{equation}
(Lemma 2.8 in \citep{MR890921}, 
p.\,305),   in the sense of convergence in probability. The   continuous, increasing, additive functional   $\, \Lambda^{\bullet} (\cdot) \,$ is supported on $\, \big\{ t \ge 0: G(t) = H(t) =0\big\}\,$.

The  expression in (\ref{eq: LTatO}) provides a measure of how this kind of occupation time grows, as a function of the angles of reflection in this quadrant. 
Since $\,\alpha < 1\,$, we deduce from (\ref{eq: LTatO}) that the semimartingale $\,\varrho_{\cdot}\,$ does not accumulate semimartingale local time   at the origin, i.e., $\,L^{\varrhob}(\cdot) \equiv 0\,$; see the Remark \ref{5.2} below. Likewise, the semimartingale $\,\varphi(\varrhob_{\cdot}, \varthetab_{\cdot})\,$ does not accumulate     local time  at the origin, i.e., $\,L^{\varphi (\varrhob, \varthetab)}(\cdot) \equiv 0\,$; and this, despite the fact that  we can find the continuous, increasing  additive functional $\, \Lambda^{\bullet} (\cdot) \,$ in (\ref{eq: LTatO}), also called {\it ``local time  for $\, \varphi (\varrhob_{\cdot}, \varthetab_{\cdot}) \,$ at the corner".} 

For a similar phenomenon in \textsc{Bessel} processes of dimension $\,\delta \in (1, 2)\,$, see \,Exercise XI\,(1.25) of 
\citep{MR1725357}, 
and Appendix A.1  of  \citep{MR2807968}. 

\begin{rem}
 \label{5.1}
Let us denote by $\,\mathbb P^{\bullet}\,$ ({\it respectively} $\,\mathbb E^{\bullet}\,$) the probability measure ({\it respectively} expectation) induced by 
(\ref{A7})-(\ref{A8}) with $\,\delta_{1} = \delta_{2} = \delta_{3}\,$, $\,G(0) = H(0) = 0\,$. Let us rescale $\,\Lambda^{\bullet}  (\cdot)\,$ by 
\[
\,L^{\bullet} (t) \, :=\, \Big( \mathbb E^{\bullet} \Big[ \int^{\infty}_{0} e^{-s} {\mathrm d} \Lambda^{\bullet}  (s) \Big] \Big)^{-1} \, \Lambda^{\bullet}  (t) \, \,; \qquad 0 \le t < \infty \, . 
\]
In this case the right continuous inverse $\,\tau^{\bullet}(u) \, :=\,  \inf \{ t \ge 0: L^{\bullet}_{t} >u\} \,$ of the map $\,t \mapsto L^{\bullet} (t)\,$ is a stable subordinator of index $\, \kappa = \alpha /  2\,$ and rate $\,1\,$ under $\,\mathbb P^{\bullet}\,$, i.e., $$\,\log \,\mathbb E^{\bullet} \big( \exp \big( - \lambda \tau^{\bullet}(u)\big)\big) \,=\, - \,u \, \lambda^{\alpha / 2} \, , \qquad \,t, u > 0\,.$$ As a result, the set $\, \big\{ t \ge 0: G(t) = H(t) =0\big\}\,$ has \textsc{Haussdorff} dimension $\, \kappa = \alpha /  2\,$, and its \textsc{Haussdorff} measure is known. For the details of excursions of the semimartingale reflected Brownian motion from the corner of the quadrant, see 
\citep{MR890921}, 
\citep{MR998631}. 
It is also shown in \citep{MR890921} 
that the measure with the density right below is invariant for the process $\,(\varrhob_{\cdot}, \varthetab_{\cdot})\,$:
$$\,
f (\rho, \theta) \, =\, \rho^{-\alpha} \cos ( \alpha \theta - \theta_{1})\,, \qquad \,0 \le \rho < \infty \, , ~~~\, 0 \le \theta \le \pi/2 \,.
$$  
\end{rem}

\begin{rem}
 \label{5.2}
Suppose that   $\,\delta_{1} = \delta_{2}= \delta_{3}\,$; that there exist a smooth function $\,\widetilde{\varphi}(\rho, \theta)\,$ for $\, 0 < \rho < \infty\,$, $\,0 \le \theta \le \pi / 2\,$ and real constants $\,r_{0} > 0\,$, $\,c_{1} > 0\,$, $\,c_{2} < 1\,$, $\,c_{3} > 0\,$, $\,p > 0\,$ such that $\, c_{0}:= (2/\alpha) - 1 - p (1-c_{2}) > 0\,$, $\, [c_{3}\, \varphi (\rho, \theta)]^{p} \le \widetilde{\varphi} (\rho, \theta) \, $ for every $\, 0 \le \rho \le r_{0}\,$, $\,0 \le \theta \le \pi \, / \, 2\,$;  and that $\, \widetilde{\varphib}_{\cdot} \, :=\, \widetilde{\varphi}(\varrhob_{\cdot}, \varthetab_{\cdot}) \,$ is a semimartingale with quadratic variation $\, \langle \widetilde{\varphib} \rangle_{\cdot}\,$ and 
\[
\int^{\cdot}_{0} {\bf 1}_{[0, r_{0})} ( \widetilde{\varphib}_{t} )\, {\mathrm d} \langle \widetilde{\varphib}\rangle_{t} \le c_{1}\int^{\cdot}_{0} {\bf 1}_{[0, r_{0})}(\widetilde{\varphib}_{t}) \cdot \lvert \widetilde{\varphi}(\varrhob_{t}, \varthetab_{t})\rvert^{c_{2}} \, {\mathrm d} t \, .
\]
Then it follows from (\ref{eq: LTatO}) that {\it the semimartingale local time $\,L^{ \widetilde{\varphi}(\varrhob, \varthetab)}(\cdot)\,$ for $\, \widetilde{\varphi}(\varrhob_{\cdot}, \varthetab_{\cdot})\,$ does not accumulate at the origin, i.e., $\,L^{ \widetilde{\varphib} }(\cdot)  \equiv L^{ \widetilde{\varphi}(\varrhob, \varthetab)}(\cdot) \equiv 0\,$. }

\medskip
Indeed, since $\, 2\, /\, \sqrt{5} \le \cos (\alpha \theta - \theta_{1})) \le 1\,$ for $\, 0 \le \theta \le \pi / 2\,$ and 
\begin{equation*}
\begin{split}
& \frac{1}{\, u \, } \int^{\cdot}_{0} {\bf 1}_{[0, u)} ( \widetilde{\varphib}_{t}) {\mathrm d} \langle \widetilde{\varphib} \rangle_{t} \le \frac{c_{1}}{\, \varepsilon^{1-c_{2}}\,} \int^{\cdot}_{0} {\bf 1}_{[0, u)} ([c_{3} \varphi(\varrhob_{t}, \varthetab_{t})]^{p}) {\mathrm d} t =  \frac{c_{1}}{\, \varepsilon^{1-c_{2}}\, } \int^{\cdot}_{0} {\bf 1}_{[0, \, \, u^{1/p}\, /\, c_{3})} (\varphi(\varrhob_{t}, \varthetab_{t})) {\mathrm d} t \\
& \, =\, \frac{c_{1} }{\, (c_{3}\, \varepsilon)^{p(1-c_{2})}\, } \int^{\cdot}_{0} {\bf 1}_{[0, u)}(\varphi(\varrhob_{t}, \varthetab_{t})) {\mathrm d} t \le c_{1} c_{3}^{-p(1-c_{2})} \varepsilon^{c_{0} + 1 - (2/\alpha)} \int^{\cdot}_{0} {\bf 1}_{[0, \varepsilon)}(\varphi(\varrhob_{t}, \varthetab_{t}))  {\mathrm d} t \, , 
\end{split}
\end{equation*}
where $\,\varepsilon:= u^{1/p}\, /\, c_{3}\,$ for $\, 0 < u \le r_{0}\,$, combining these estimates with (\ref{eq: LTatO}), letting $\, u \downarrow 0\,$, and hence $\,\varepsilon \downarrow 0\,$, we obtain the convergence in probability 
\begin{equation} \label{eq: LTrho}
L^{ \widetilde{\varphi}(\varrhob, \varthetab)}(\cdot) \, :=\, \lim_{u \downarrow 0} \frac{1}{\, 2u\, } \int^{\cdot}_{0} {\bf 1}_{[0, u)} ( \widetilde{\varphib}_{t}) {\mathrm d} \langle \widetilde{\varphib}\rangle_{t} \, =\, 0 \, . 
\end{equation}

With this claim we can verify $\,L^{\varrhob}(\cdot) \equiv 0\,$ by choosing $\,\widetilde{\varphi}(\rho, \theta) := \rho\,$ for $\, 0 \le \rho < \infty\,$, $\,0 \le \theta \le \pi/2\,$ with $\,p \, :=\,  1\, / \, \alpha > 0\,$, $\,c_{1} \, :=\, 1\,$, $\,c_{2} \, :=\, 0\,$, $\,c_{3} \, :=\, 1\,$, $\,r_{0} \, :=\, 1\,$ and $c_{0} \, :=\, (1\,/\,\alpha) - 1 > 0\, $, since $\,0 < \alpha < 1\,$. Similarly, it can be verified that $\,L^{\varphi(\varrhob, \varthetab)}(\cdot) \equiv 0\,$ because $\, \lvert \nabla \varphi (\rho, \theta) \rvert^{2} \, =\, \alpha^{2} \rho^{2\alpha - 2}\,$, $\, {\mathrm d} \langle \varphi (\varrhob, \varthetab) \rangle_{t} \, / \, {\mathrm d} t \, =\, \alpha^{2} \varrhob_{t}^{2\alpha - 2}\,$, $\,c_{1}\, :=\, (\sqrt{5}/2)^{c_{2}} \alpha^{2} > 0\,$, $\,c_{2}\, :=\, 2\alpha - 2 < 1\,$, $\,c_{3}\, =\, 1 \,$, $\,p \, :=\, 1\,$, $\,r_{0} \, =\, 1\,$, $\,c_{0} \, :=\, (2/\alpha) - 1 - (3 - 2\alpha) = (2/\alpha) + 2\alpha - 4 \approx 0.568 > 0\,$. 
\end{rem}

\begin{appendix}
%
%
\section{Middle Diffusion,  Ballistic Hedges, Skew-Elastic Collisions}
\label{skew-elastic}

Double collisions were completely ``elastic" in the systems of Sections \ref{sec1} and \ref{sec5}: when two particles there collided, they split their collision local times evenly. We   study here briefly a variant of the system (\ref{1}) --- with the same purely ballistic motions for the leader and laggard particles, and the same diffusive motion for the middle particle --- but now with ``skew-elastic" collisions, as in \citep{MR3062434}, 
between the second- and third-ranked particles. 

More precisely, we consider   in the  notation of (\ref{ranks}), (\ref{LT}), and with $  \delta_1, \,\delta_2,\,\delta_3 $, $   x_1 > x_2 > x_3 $   given real numbers, the system of equations,  first introduced and studied in \citep{FernholzER11} : 
\begin{equation}
\label{B1}
X_i (\cdot) \,=\, x_i + \sum_{k=1}^3\, \delta_k \int_0^{\, \cdot} \1_{ \{ X_i (t) = R^X_{k} (t)\} } \, \dx t  +  \int_0^{\, \cdot} \1_{ \{ X_i (t) = R^X_{2} (t)\} } \, \dx B_i (t) 
~~~~~~~~~~~~~~~~~~~~~~~~~
\end{equation}
$$
~~~~~~~~~~~~~\,~~~~~~~~~~~~~~~~~~~~ 
 + \,  \int_0^{\, \cdot} \1_{ \{ X_i (t) = R^X_{2} (t)\} } \, \dx L^{R^X_{2} - R^X_{3}} (t) +   \int_0^{\, \cdot} \1_{ \{ X_i (t) = R^X_{3} (t)\} } \, \dx L^{R^X_{2} - R^X_{3}} (t) \,
$$
for $\,i \, =\, 1, 2, 3\,$.  We shall try to find a weak solution to this system; in other words, construct   a filtered probability space $\, (\Omega, \F, \Prob),$ $\mathbb{F}= \big\{ \F (t) \big\}_{0 \le t < \infty}\,$ rich enough to accommodate  independent Brownian motions $\, B_1 (\cdot), \,B_2 (\cdot), $ $B_3 (\cdot)  \,$ and   continuous  semimartingales   $\,X_1 (\cdot), \,X_2 (\cdot), \,X_3 (\cdot)  $  so that,   with probability one,  the equations of (\ref{B1}) are satisfied,  along with  the ``non-stickiness" and ``soft triple collision" requirements 
\begin{equation}
\label{B3}
   \int_0^{\, \infty} \1_{ \{ R^X_{k} (t) = R^X_{\ell} (t)\} } \, \dx  t \,=\, 0\,, ~~~~~\forall ~~ k < \ell \,;\qquad~~  ~ L^{R^X_{1} - R^X_{3}} (\cdot) \equiv 0
    \,.
\end{equation}

\subsection{Analysis}
 \label{sec22}

Assuming that such a weak solution to the system of (\ref{B1}), (\ref{B3})   has been constructed, the ranked processes $\, R^X_k (\cdot)\,$ as in (\ref{ranks}) are continuous semimartingales with decompositions
\begin{equation}
\label{B4}
R^X_{1} (t)\,=\, x_1 + \delta_1 \, t + {1 \over \,2\,} \Lambda^{(1,2)}(t) \,, \qquad R^X_{3} (t)\,=\, x_3 + \delta_3 \, t  +   {1 \over \,2\,} \, \Lambda^{(2,3)}(t) 
\end{equation}
\begin{equation}
\label{B5}
R^X_{2} (t)\,=\, x_2 + \delta_2\, t + W(t) - {1 \over \,2\,} \, \Lambda^{(1,2)}(t) + {3 \over \,2\,} \, \Lambda^{(2,3)}(t)
\end{equation} 
 \noindent
by analogy with (\ref{RX1})-(\ref{RX2});  though also  with the clear difference, that the collision local time $\Lambda^{( 2,3)}(\cdot)$ is not split now evenly between the second- and third-ranked particles, but rather in a 1:3 proportion.  We are using here the exact same notation for the standard Brownian motion $  W(\cdot) $ as in (\ref{W2}), and  for the collision local times $\,\Lambda^{(k,\ell)}(\cdot)  \,$ as in (\ref{Lambda}).  For the gaps $\, G (\cdot)= R^X_1 (\cdot)- R^X_2 (\cdot)\,$, $  H (\cdot)= R^X_2 (\cdot)- R^X_3 (\cdot)\,$  we have   the \textsc{Skorokhod}-type  representations of the form (\ref{SkorU})-(\ref{SkorV}), now with 
$$
U(t)  = x_1 - x_2 + \big(\delta_1 - \delta_2 \big) \,t  - W(t) - {3 \over \,2\,} \, L^{H}(t)\,, \quad V(t)  = x_2 - x_3 + \big(\delta_2 - \delta_3 \big)\, t  + W(t) - {1 \over \,2\,} \, L^{G}(t)\,.
$$
 
Whereas, from the theory of the \textsc{Skorokhod} reflection problem we obtain now the relationships linking  the two local time processes $\, L^G (\cdot)  \,$ and $\, L^H (\cdot)  \,$, namely 
\begin{equation}
\label{B7}
L^G (t) \,=\, \max_{0 \le s \le t} \big( - U (s) \big)^+\,=\, \max_{0 \le s \le t} \Big(  x_2 - x_1 + \big(\delta_2-\delta_1 \big) \,s  + W(s) + {3 \over \,2\,} \, L^{H}(s) \Big)^+\,,
\end{equation}
\begin{equation}
\label{B8}
L^H (t) \,=\, \max_{0 \le s \le t} \big( - V (s) \big)^+\,=\, \max_{0 \le s \le t} \Big( x_3 - x_2+ \big(\delta_3-\delta_2 \big) \,s  - W(s) + {1 \over \,2\,} \, L^{G}(s) \Big)^+ 
\end{equation}
 $\bullet~$ 
The resulting system for the two nonnegative gap processes
\begin{equation}
\label{B9}
G (t) \,=\, x_1 - x_2 + \big(\delta_1 - \delta_2 \big) \,t  - W(t) - {3 \over \,2\,} \, L^{H}(t) +   L^{G}(t)\,, \qquad 0 \le t < \infty
\end{equation}
\begin{equation}
\label{B10}
H (t) \,=\, x_2 - x_3 + \big(\delta_2 - \delta_3 \big)\, t  + W(t) - {1 \over \,2\,} \, L^{G}(t) +   L^{H}(t)\,, \qquad 0 \le t < \infty 
\end{equation}
\noindent
 is again of the \textsc{Harrison \& Reiman} \citep{MR606992} 
 type (\ref{eq: mfrakQ}). It amounts to reflecting off the faces of the nonnegative orthant  the degenerate, two-dimensional Brownian motion $\, \mathfrak{Z} (\cdot) \,$ as in (\ref{Z}), with   drift vector $\, {\bm m} = \big(\delta_1 - \delta_2, \,\delta_2 - \delta_3\big)'\,,$  covariance matrix $\,\mathcal{C}\,$ 
as   in (\ref{mA}),   but now with reflection matrix
$$
\mathcal{R}\, :=\, \mathcal{I} - \mathcal{Q}\,, \qquad \mathcal{Q}\,  = \begin{pmatrix}
      0 &     3/2 \\
        1/2 &   0
\end{pmatrix}\,,\qquad \text{thus} \qquad \mathcal{R}^{-1} {\bm m}\, = \,2\,\begin{pmatrix}
        \,2\, \delta_1 + \delta_2 - 3\,\delta_3  \,     \\
          \delta_1 + \delta_2  - 2\, \delta_3     
\end{pmatrix} .  
$$
Here   the matrix $\, {\cal Q}\,$ has 
spectral radius strictly less than 1, and   the {\it skew-symmetry condition} 
$\,\mathcal{R} + \mathcal{R}^\prime = 2 \, \mathcal{C}\,$ 
of \citep{MR912049} 
is satisfied by these covariance and reflection matrices.

\subsection{Synthesis} 
\label{sec33}

Let us  start now with  given real numbers $\, \delta_1, \,\delta_2,\,\delta_3\,$, and $ \, x_1 > x_2 > x_3\,$, and construct   a filtered probability space $\, (\Omega, \F, \Prob),$ $\mathbb{F}= \big\{ \F (t) \big\}_{0 \le t < \infty}\,$ rich enough to support  a standard Brownian motion $\, W(\cdot)$.   
By analogy with (\ref{B7})-(\ref{B8}), we consider  the following system of equations  for two continuous, nondecreasing and adapted processes $\, A(\cdot)\,$ and $\, \Gamma (\cdot)\,$ with $\, A(0) = \Gamma (0) =0\,$:
 \begin{equation}
\label{B11}
A (t) \,=\,   \max_{0 \le s \le t} \Big( x_2 - x_1 + \big(\delta_2 - \delta_1 \big) \,s  + W(s) + {3 \over \,2\,} \, \Gamma (s) \Big)^+\,,\qquad 0 \le t <\infty
\end{equation}
 \begin{equation}
\label{B12}
\Gamma (t) \,=\,  \max_{0 \le s \le t} \Big( x_3 - x_2 + \big(\delta_3 - \delta_2 \big) \,s  - W(s) + {1 \over \,2\,} \,A(s) \Big)^+\,,\qquad 0 \le t <\infty.
\end{equation}
 
Theorem 1 of \citep{MR606992} 
 guarantees that this system has a unique continuous solution $\, ( A(\cdot), \Gamma (\cdot)) ,$   
adapted to the smallest right-continuous filtration $\,\mathbb{F}^W$ to which the driving Brownian motion $W(\cdot)$ is itself adapted.   
With this   solution in place, we construct the continuous semimartigales 
\begin{equation}
\label{B13}
U(t) := x_1 - x_2 + \big(\delta_1 - \delta_2 \big) \,t  - W(t) - {3 \over \,2\,} \, \Gamma(t)\,, \quad V(t) := x_2 - x_3 + \big(\delta_2 - \delta_3 \big)\, t  + W(t) - {1 \over \,2\,} \, A(t)\,,~~
\end{equation}
  and then ``fold" them, to obtain their \textsc{Skorokhod} reflections   
 \begin{equation}
\label{B14}
G ( t) \,:=\, U ( t) +\max_{0 \le s \le t} \big( - U (s) \big)^+ \, = \, x_1 - x_2 +\big(\delta_1 - \delta_2 \big) \,t  - W(t) - {3 \over \,2\,} \, \Gamma(t) + A(t) \, \ge \, 0
\end{equation}
\begin{equation}
\label{B15}
H ( t) \,:=\,   V ( t) +\max_{0 \le s \le t} \big( - V (s) \big)^+ \, = \, x_2 - x_3 + \big(\delta_2 - \delta_3 \big) \,t  + W(t) - {1 \over \,2\,} \, A(t) + \Gamma(t)\, \ge \, 0
\end{equation}
for $\, t \in [0, \infty)$. As before, for these two continuous, nonnegative semimartingales   the theories of the \textsc{Skorokhod} reflection problem and of semimartingale local time give, respectively,  
\begin{equation}
\label{B16}
\int_0^\infty \1_{ \{ G(t)>0\} } \, \dx A(t) \,=\,0\,, \qquad \int_0^\infty \1_{ \{ H(t)>0\} } \, \dx \Gamma (t) \,=\,0 \,,
\end{equation}
\begin{equation}
\label{B17}
\int_0^\infty \1_{ \{ G(t)=0\} } \, \dx t \,=\,0\,, \qquad \int_0^\infty \1_{ \{ H(t)=0\} } \, \dx t \,=\,0\,.
\end{equation}
We claim also  the additional properties
  \begin{equation}
\label{B18}
\int_0^\infty \1_{ \{ G(t)=0\} } \, \dx \Gamma(t) \,=\,0\,, \qquad
\int_0^\infty \1_{ \{ H(t)=0\} } \, \dx A (t) \,=\,0\,.
\end{equation}
\noindent
 Indeed, focusing on the first property (the other is handled  similarly), we see that 
$\,
\int_0^\infty \1_{ \{ G(t)=0\} } \, \dx \Gamma(t) \,=\, \int_0^\infty \1_{ \{ G(t)=H(t)=0\} } \, \dx \Gamma(t) \,=\,0\,$ 
holds,  from the second equality in (\ref{B16}) and    Theorem 1 in \citep{MR921820}.  

  \smallskip
\noindent
$\bullet~$
We need to identify the regulating processes $\, A(\cdot)\,$, $\, \Gamma (\cdot)\,$ as local times. We start  by observing
$$
L^G (\cdot) = \int_0^{\, \cdot} \1_{ \{ G(t)=0\} }\, \dx G(t) = \int_0^{\, \cdot} \1_{ \{ G(t)=0\} }\,\Big[\, \dx A(t) - \frac{\,3\,}{2} \dx \Gamma (t) - \dx W (t) + \big( \delta_1 - \delta_2\big)\, \dx t \,\Big] 
$$
from (\ref{LT}) and (\ref{B14}). The last (\textsc{Lebesgue}) and next-to-last (\textsc{It\^o}) integrals in this expression vanish on the strength of (\ref{B17}), whereas the third-to-last integral vanishes on account of (\ref{B18}); so we deduce the identification $\, L^G (\cdot) = \int_0^{\,\cdot} \1_{ \{ G(t)=0\} }\, \dx A(t) \equiv A(\cdot)\,$, where the second equality comes on the  heels of (\ref{B16}). We identify similarly   $\, L^H (\cdot)   \equiv \Gamma (\cdot)\,$.

 \smallskip
\noindent
$\bullet~$ 
By analogy with (\ref{B4})-(\ref{B5}), we construct  now the 
$\,\mathbb{F}^W-$adapted 
{\it  processes of ranks}
\begin{equation}
\label{B19}
R_{1} (t)\,:=\, x_1 + \delta_1\, t + {1 \over \,2\,} \, A(t)\,, \qquad R_{3} (t)\,:=\, x_3 + \delta_3 \, t  +   {1 \over \,2\,} \, \Gamma (t) \,,
\end{equation}
\begin{equation}
\label{B20}
R_{2} (t)\,:=\, x_2 + \delta_2 \, t + W(t) - {1 \over \,2\,} \, A(t) + {3 \over \,2\,} \, \Gamma (t)
\end{equation}
  and note   $\, R_1 (\cdot) - R_2 (\cdot) = G (\cdot)\ge 0\,$, $\, R_2 (\cdot) - R_3 (\cdot) = H (\cdot) \ge 0\,$,  thus  $\, R_1 (\cdot) \, \ge \,   R_2 (\cdot) \, \ge \,  R_3 (\cdot)   \, $ and
  $$
R_1  (t) - R_3 (t) \,=\, G(t) + H(t) \,=\,x_1 - x_3 + \big( \delta_1 - \delta_3 \big)\,t + \frac{1}{\,2\,} \big[ \, A (t) - \Gamma (t) \, \big]  .
  $$
 The continuous  process $\, R_1 (\cdot)- R_3 (\cdot) \ge 0 \,$ is of finite first variation on compact intervals,   so   its local time at the origin vanishes, as   posited in (\ref{B3}): $\,  L^{R_1    - R_3} (\cdot) \equiv 0\,$. The other properties posited there are direct consequences of (\ref{B17}).    Finally, the identifications $A(\cdot) \equiv L^G (\cdot) \equiv L^{R_1 - R_2} (\cdot)$,  $\Gamma (\cdot) \equiv L^H (\cdot) \equiv L^{R_2 - R_3} (\cdot)$ show, in conjunction with (\ref{B20}), that the rank vector process $(R_1  (\cdot), R_2  (\cdot), R_3  (\cdot))$  and the scalar, standard Brownian motion $W(\cdot)$ generate the   same filtration. 

\smallskip
\noindent
$\bullet~$
We can construct now on a suitable filtered  probability space  
independent Brownian motions $\, B_1 (\cdot),$ $B_2 (\cdot),$ $B_3 (\cdot)  \,$ and   continuous, adapted processes $X_1 (\cdot), \,X_2 (\cdot), \,X_3 (\cdot)  \,$  so that,   with probability one,  the equations of (\ref{B1}) are satisfied,  along with those   of (\ref{B3}),  up until the first time of a triple collision
\begin{equation}
\label{B22}
\mathcal{S} \,:=\, \big\{ t \ge 0 \,:\, X_1 (t) = X_2 (t) = X_3 (t)  \big\}\,,
\end{equation}
as well as $\, R^X_k (t) = R_k (t)\,$, $ \, 0 \le t < \mathcal{S}\,$ for $\, k=1,2,3\,$. Just as before, this is done by considering the particles two-by-two in the manner of \citep{MR3055258}, 
and applying the results in \citep{MR3062434}, \citep{MR3055262}. 

 A simulation of the paths of the resulting process $(R_1(\cdot), R_2 (\cdot), R_3 (\cdot))\,$ with $\, \delta_1 = -1$, $\,\delta_2 =   -2\,$ and $\, \delta_3 = -1\,$ is depicted in Figure \ref{Fig: Elastic}, reproduced here from \citep{FernholzER11}. 
{\it We believe, but have not been able to show,  that   $\, \mathbb{P} ( \mathcal{S} = \infty) =1\,$ holds in this case.}


 \begin{figure}  
  \begin{center}
    \includegraphics[scale=0.41
    ]{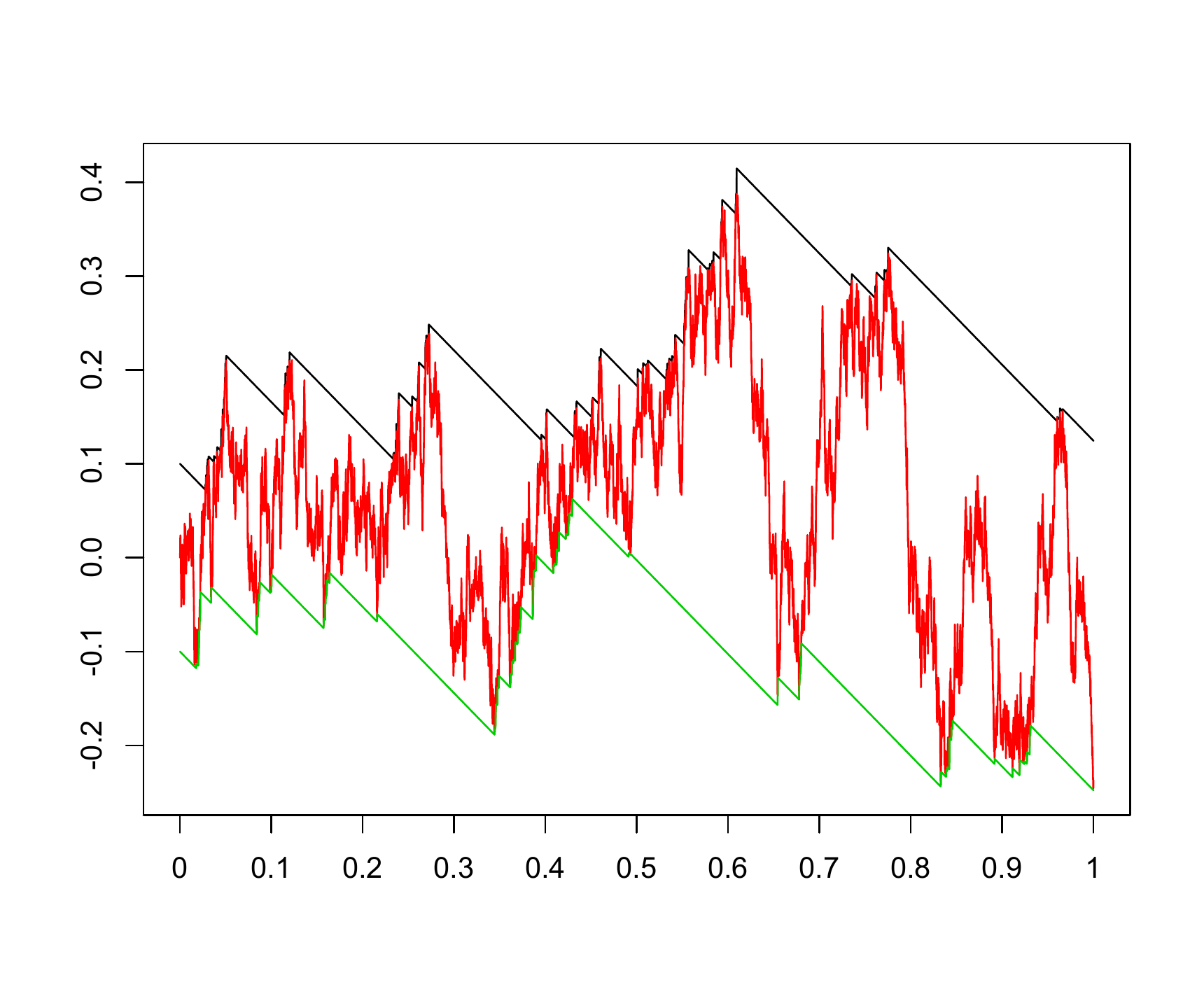}
    \caption{Simulated processes; Black $\,=R_1(\cdot)\,$,  Red $\,=R_2 (\cdot)\,$,  Green $\,=R_3(\cdot)\,$. Here we have taken $\, \delta_1 = -1$, $\,\delta_2 =   -2\,$ and $\, \delta_3 = -1\,$ in (\ref{B1}).  We are indebted to Dr. \textsc{E.R. Fernholz} for this picture.}\label{Fig: Elastic}
  \end{center}
\end{figure}


\subsection{Invariant Distribution} 
\label{sec44}

Under the conditions
\begin{equation}
\label{B23}
3 \, \delta_3 \, > \, 2\, \delta_1 + \delta_2\,\,, \qquad 2 \, \delta_3 \, > \,   \delta_1 + \delta_2\,\,,
\end{equation}
\noindent
  both components of the vector
\begin{equation}
\label{La}
{\bm \lambda} \equiv \begin{pmatrix}
        \,\lambda_1     \,     \\
          \lambda_2        
\end{pmatrix}\,:=\, 
-\, \mathcal{R}^{-1} {\bm m}\, = \,2\,\begin{pmatrix}
        \,3\,\delta_3 - 2\, \delta_1 - \delta_2     \,     \\
          2\, \delta_3 - \delta_1 - \delta_2        
\end{pmatrix}
\end{equation}
 are   positive numbers. Repeating the reasoning in subsection \ref{sec4},   
 we deduce that here again the  two-dimensional, degenerate process $\,\big(G( \cdot), H(\cdot)\big) \,$ of gaps is positive recurrent, has a unique invariant  measure $\, {\bm \pi}\,$ with $\, {\bm \pi}  ( (0,\infty)^2 ) =1\,$, and converges to this probability measure in distribution as $\, t \ra \infty\,$.

 For instance, we have  in the   present context the analogue 
 \[
\dx \big( G^2 (t) + 3G (t) H(t) + 3H^2 (t) \big)\,=\, \Big[ \, 1 - { \, 1 \,\over 2} \big( \lambda_1 G(t) + 3 \lambda_2 H(t) \big)   \Big] \dx t + \big( 3 H(t) + G(t) \big) \, \dx W(t) 
\]
of the dynamics (\ref{NoLocTimes});  and the function $\, V(g,h) = \exp \big\{ \sqrt{ g^2 + 3g h + 3 h^2\,} \, \big\},$ the analogue   of (\ref{Lyap}),  is shown as in Proposition \ref{Lyap_func} to be a \textsc{Lyapunov} function for the semimartingale reflecting Brownian motion $\, (G(\cdot), H(\cdot))$.     This process is thus seen to be positive recurrent, and to have a   unique  invariant distribution. 
  We also deduce, just as before, the strong laws of large numbers
 $$
 \lim_{t \ar} \frac{\,L^G (t)\,}{t} \, =\,\lambda_1 \,=\, 2 \, \big( 3\,\delta_3 - 2\, \delta_1 - \delta_2\big)   \,,\qquad  ~\lim_{t \ar} \frac{\,L^H (t)\,}{t}  \, =\,\lambda_2\,=\, 2 \, \big( 2\,\delta_3 -   \delta_1 - \delta_2\big)\,.
 $$

On the other hand,  the fact that the covariance matrix $\, {\cal C}\,$ and the reflection matrix $\, {\cal R}\,$ satisfy the skew-symmetry condition $\,\mathcal{R} + \mathcal{R}^\prime = 2 \, \mathcal{C}\,$ implies  that {\it the invariant probability measure for the  vector process $\,\big(G( \cdot), H(\cdot)\big) \,$ of gaps should be the product of exponentials}
\begin{equation}
\label{InvMeas}
{\bm \pi} \big( \dx g, \dx h \big) \,=\, 4\, \lambda_1 \,\lambda_2 \,  e^{\, - 2\, \lambda_1\, g\, - 2\, \lambda_2\, h} \,     \dx g\, \dx h\,, \qquad (g,h) \in (0, \infty)^2\,.
\end{equation}
This  extends the results of \citep{MR912049} \citep{MR877593} 
to the degenerate case discussed here. 

 \medskip
\noindent
{\it Proof of the Claim in (\ref{InvMeas}):} 
This   claim  can be verified as in section 9 of \citep{MR912049};  
for completeness, we present now the details.   As shown in that paper  and in \citep{MR1873294}, 
it is enough to find two measures $\,{\bm \nu_1}\,$ and $\,{\bm \nu_2}\,$ on $\,(0, \infty)\,$ so that the appropriate form 
\begin{equation} 
\label{eq: BAR1}
\int_0^\infty \int_0^\infty \Big( \big(D_{gg}^{2} + D_{hh}^{2} - 2 D_{gh}^{2}\big)    + 2 \big( \delta_1 - \delta_2 \big) D_g + 2 \big( \delta_2 - \delta_3 \big) D_h   \Big) f(g,h)\, {\bm \pi}({\mathrm d} g ,  {\mathrm d} h) +~~~~~~~~~ 
\end{equation}
\[
~~~~~~~~~~~~~~~~~~~~{}+ \int_0^\infty \Big( D_{g} - \frac{1}{\, 2\, } D_{h}  \Big) f(0, h)\, {\bm \nu}_{1}({\mathrm d} h) +  \int_0^\infty \Big( D_{h} - \frac{3}{\, 2\, } D_{g}  \Big) f(g, 0) \, {\bm \nu}_{2}({\mathrm d} g) \, =\, 0   
\]
of the Basic Adjoint Relationship for the system of (\ref{B9}), (\ref{B10}),   holds  for every function $ f$ of class $C^2\big( (0, \infty)^2\big)\,$. Once again, selecting  $\, f(g,h) = \exp \big( - \alpha_1 g - \alpha_2 h\big)\,$ for $\, \alpha_1 \ge 0 \,$ and $\, \alpha_2 \ge 0\,$ and substituting in (\ref{eq: BAR1}), we obtain the equation
\begin{equation} 
\label{eq: Lap1}
\Big( (\alpha_{1} - \alpha_{2})^{2}  + 2 (\delta_{2}-\delta_{1}) \alpha_{1} + 2 (\delta_{3}-\delta_{2}) \alpha_{2}\Big) \,\widehat{{\bm \pi} }(\alpha_1, \alpha_2)\, =\, \Big( \alpha_{1} - \frac{\alpha_{2}}{\, 2\, } \Big) \,\widehat{{\bm \nu}}_{1}  (\alpha_{2}) + \Big( \alpha_{2} - \frac{3\,\alpha_{1}}{\, 2\, } \Big) \,\widehat{{\bm \nu}}_{2}  (\alpha_{1}) \,   ~~~
\end{equation}
linking the \textsc{Laplace} transforms of the measures $\, {\bm \pi}\,$ and $\,{\bm \nu_1}\,$, $\,{\bm \nu_2}\,$.

In accordance with the guess (\ref{InvMeas}) that we are trying to  establish, let us posit the product-form   
\begin{equation}
\label{prod_form}
{\bm \pi} \big( \dx g, \dx h \big) \,=\, p_1 (g) \, p_2 (h) \, \dx g \, \dx h \,, \qquad (g,h) \in (0, \infty)^2
\end{equation}
for the invariant distribution of the process of gaps; here $\, p_1 (\cdot)\,$ and $\, p_2 (\cdot)\,$ are probability density functions on the positive half-line. We denote by $\, \widehat{p}_1 (\cdot)\,$ and $\, \widehat{p}_2 (\cdot)\,$ the   \textsc{Laplace} transforms of these density functions,   set 
$\, c_j \,:=\, \lim_{\alpha \rightarrow \infty}\, \big( \alpha \,  \widehat{p}_j (\alpha) \big)\,$, $\, j=1, 2\,$, 
and note that (\ref{prod_form}) implies then $\, \widehat{{\bm \pi} }(\alpha_1, \alpha_2)= \widehat{p}_1 (\alpha_1) \,  \widehat{p}_2 (\alpha_2) \,$.  We divide now the resulting expression  (\ref{eq: Lap1}) by $\, \alpha_1 >0\,$ (respectively, by $\, \alpha_2 >0\,$), then send $\, \alpha_1  \,$ (respectively, by $\, \alpha_2  \,$) to infinity; the results are, respectively, 
$$
\, \widehat{{\bm \nu}}_{1}  (\alpha_{2})\,=\, c_1 \, \widehat{p}_2 (\alpha_2)\,, \qquad 
\widehat{{\bm \nu}}_{2}  (\alpha_{1})\,=\, c_2 \, \widehat{p}_1 (\alpha_1)\,.
$$
Substituting these expressions back into (\ref{eq: Lap1}) gives 
$$
\widehat{p}_1 (\alpha_1) \,  \widehat{p}_2 (\alpha_2)\, \big( (\alpha_{1} - \alpha_{2})^{2}  + 2 (\delta_{2}-\delta_{1}) \alpha_{1} + 2 (\delta_{3}-\delta_{2}) \alpha_{2}\big)\,=\,~~~~~~~~~~~~~~~~~~~~~~~~~~~~~~~~~~~~~~~~
$$
$$
~~~~~~~~~~~~~~~~~~~~~~~~~~~~~~~~~~~~~~~~~~~~~~
=\, c_1 \, \widehat{p}_2 (\alpha_2)\, \big( \alpha_{1} - (1 /2)  \alpha_{2}  \big) \,
+ c_2 \, \widehat{p}_1 (\alpha_1)\, \big( \alpha_{2} - (3 /2)  \alpha_{1}  \big) \,;
$$
whereas, setting $\, \alpha_2=0\,$ (respectively, $\, \alpha_1=0$) in this last equation, we obtain   
$$
c_1 \,=\, \widehat{p}_1 (\alpha_1) \, \big( \alpha_1 + 2 \big( \delta_2 - \delta_1 \big) + (3 /2) \, c_2 \big)\,, \qquad c_2 \,=\, \widehat{p}_2 (\alpha_2) \, \big( \alpha_2 + 2 \big( \delta_3 - \delta_2 \big) + (1 /2) \, c_1 \big) .
$$
On account of the rather obvious properties $\, \widehat{p}_1 (0) = \widehat{p}_2 (0) =1\,$, we obtain the system of equations
$$
c_1 \,=\,   2 \big( \delta_2 - \delta_1 \big)  + (3 /2) \, c_2\,\,, \qquad 
c_2 \,=\,   2 \big( \delta_3 - \delta_2 \big)  + (1 /2) \, c_1\,.
$$
 The solution to this system is now rather trivially $\, c_1 = 2\, \lambda_1\,$, $\, c_2 = 2\, \lambda_2\,$ in the notation of (\ref{La}); this leads to the transforms $\, \widehat{p}_j (\alpha) = ( 2 \lambda_j) / ( \alpha + 2 \lambda_j)\,$, $ \alpha \ge 0\,$, and thence to the   respective   densities and measures
\begin{equation}
\label{marginals}
 p_1 ( g) \,=\, 2 \,\lambda_1 \, e^{\, - 2\, \lambda_1\, g}\,, ~~~~g>0\,,    
 \qquad p_2 ( h) \,=\, 2 \,\lambda_2 \, e^{\, - 2\, \lambda_2\, h} \,, ~~~~h>0\,,
\end{equation}
\begin{equation}
\label{NUs}
 {\bm \nu}_{2}({\mathrm d} g) \,=\, 4 \,\lambda_2 \,\lambda_1  \, e^{\, - 2\, \lambda_1\, g}\, \dx g \,, %
  \qquad  {\bm \nu}_{1}({\mathrm d} h) \,=\, 4 \,\lambda_1 \,\lambda_2 \, e^{\, - 2\, \lambda_2\, h} \,  \dx h \,.
\end{equation}
  These   have total masses   $\, {\bm \nu}_{1}((0, \infty)) = 2\, \lambda_1\,$ and $\, {\bm \nu}_{2}((0, \infty)) = 2\, \lambda_2\,,$  just as in  (\ref{eq: bdry mass1}), (\ref{eq: bdry mass2}). For the  two   density functions of  (\ref{marginals}), the product measure   (\ref{prod_form}) satisfies the Basic Adjoint Relationship (\ref{eq: BAR1}), and is thus the invariant distribution of  the  two-dimensional process $ \big(G( \cdot), H(\cdot)\big)  $ of gaps.  
  \qed 
\end{appendix}

%
%

\section*{Acknowledgements}
We are grateful to Dr.\,E.\,Robert \textsc{Fernholz} for initiating this line of research, prompting us over the years to continue it, and providing   simulations for the paths of the processes involved. We are indebted to Drs.\,\,Jiro \textsc{Akahori},  David \textsc{Hobson}, Chris \textsc{Rogers}, Johannes \textsc{Ruf},  Mykhaylo \textsc{Shkolnikov},   Minghan \textsc{Yan},  and especially   Andrey \textsc{Sarantsev},   for   discussions about the problems treated here,  and for their advice and   suggestions. The referee read the paper with great diligence and care, corrected a serious error, and made several incisive comments and suggestions, for which we are deeply grateful. 

The first author was supported in part by  the National Science Foundation under grants DMS-1615229 and DMS-2008427. The second author was supported  by  the National Science Foundation under  grants NSF-DMS-0905754, NSF-DMS-1405210 and NSF-DMS-2004997. 

\end{document}